\documentclass[11pt, a4paper]{article}

\usepackage[utf8]{inputenc}
\usepackage[english]{babel}
\usepackage{dutchcal}
\usepackage[T1]{fontenc}

\usepackage{amsmath}
\usepackage{amssymb}
\usepackage{amsthm}
\usepackage{tikz}
\usepackage{tikz-cd}
\usepackage{hyperref}
\usepackage{cleveref}

\usepackage{footmisc}
\usepackage{afterpage}
\usepackage{mathrsfs}
\usepackage{sseq}
\usepackage{spectralsequences}
\usepackage{extarrows}
\usepackage{enumitem}
\usepackage{stmaryrd}
\usepackage{pdflscape}

\hypersetup{
  colorlinks   = true, 
  urlcolor     = blue, 
  linkcolor    = blue, 
  citecolor   = blue 
}
\newcommand{\llpar}{\llparenthesis}
\newcommand{\rrpar}{\rrparenthesis}

\newcommand{\ab}{\mathrm{ab}}

\newcommand{\aff}{\mathrm{aff}}

\newcommand{\cn}{\mathrm{cn}}

\newcommand{\Ell}{\mathrm{Ell}}

\newcommand{\dualisable}{\mathrm{dual}}

\newcommand{\fin}{\mathrm{fin}}
\newcommand{\fpqc}{\mathrm{fpqc}}

\newcommand{\gl}{\mathrm{gl}}

\newcommand{\inj}{\mathrm{inj}}

\newcommand{\km}{{\mathrm{KM}}}
\newcommand{\KM}{\km}
\newcommand{\nc}{{\mathrm{nc}}}

\newcommand{\ori}{\mathrm{or}}
\newcommand{\Ori}{\mathrm{Or}}

\newcommand{\perf}{\mathrm{perf}}

\newcommand{\sub}{\mathrm{sub}}

\newcommand{\st}{\mathrm{st}}

\newcommand{\Tate}{\mathrm{Tate}}

\newcommand{\ud}{\mathrm{ud}}

\DeclareMathOperator{\Spec}{Spec}
\newcommand{\relSpec}{{\underline{\mathrm{Spec}}}}

\DeclareMathOperator{\GL}{GL}
\DeclareMathOperator{\SL}{SL}

\newcommand{\1}{\mathbf{1}}

\newcommand{\BG}{\mathbf{B}\mathcal{G}}

\DeclareMathOperator{\KU}{KU}
\newcommand{\gKU}{{\mathbf{KU}}}
\newcommand{\gKO}{{\mathbf{KO}}}
\newcommand{\gTMF}{{\mathbf{TMF}}}

\DeclareMathOperator{\KO}{KO}

\newcommand{\Sph}{\mathbf{S}}

\newcommand{\TMF}{\mathrm{TMF}}

\DeclareMathOperator{\Ab}{Ab}

\DeclareMathOperator{\CAlg}{CAlg}
\DeclareMathOperator{\tworing}{2CAlg}
\DeclareMathOperator{\Cat}{Cat}
\newcommand{\PrL}{{\mathrm{Pr}^L}}
\newcommand{\PrLst}{{\mathrm{Pr}^L_\st}}
\newcommand{\PrLstomega}{{\mathrm{Pr}^L_{\st,\omega}}}

\DeclareMathOperator{\FG}{FG}
\DeclareMathOperator{\Hol}{Hol}

\DeclareMathOperator{\Fun}{Fun}

\DeclareMathOperator{\map}{map}
\DeclareMathOperator{\Glo}{Glo}

\DeclareMathOperator{\Inj}{Inj}
\DeclareMathOperator{\Sub}{Sub}

\DeclareMathOperator{\Ind}{Ind}
\DeclareMathOperator{\Mod}{Mod}

\newcommand{\Lat}{\mathrm{Lat}}

\DeclareMathOperator{\PreAb}{PreAb}
\DeclareMathOperator{\Perf}{Perf}

\DeclareMathOperator{\QCoh}{QCoh}

\DeclareMathOperator{\Span}{Span}

\DeclareMathOperator{\Stk}{Stk}
\DeclareMathOperator{\Sp}{Sp}
\DeclareMathOperator{\SpDM}{SpDM}
\newcommand{\Spc}{\mathcal{S}}

\DeclareMathOperator{\Top}{Top}
\newcommand{\Tori}{{\mathrm{Tori}}}

\DeclareMathOperator{\Aut}{Aut}
\newcommand{\colim}{\mathrm{colim}\,}

\DeclareMathOperator{\Map}{Map}

\DeclareMathOperator{\Hom}{Hom}

\newcommand{\op}{\mathrm{op}}

\newcommand{\ul}[1]{\underline{#1}}

\newcommand{\dual}[1]{{\widehat{#1}}}
\newcommand{\Cech}{{\check{C}}}

\newcommand{\BT}{\mathbf{BT}}

\newcommand{\calC}{\mathcal{C}}

\newcommand{\tildE}{\widetilde{E}}
\newcommand{\E}{\mathbf{E}}
\newcommand{\sfE}{\mathsf{E}}

\newcommand{\calE}{\mathcal{E}}
\newcommand{\F}{\mathbf{F}}
\newcommand{\calF}{\mathcal{F}}
\newcommand{\G}{\mathbf{G}}
\newcommand{\calG}{\mathcal{G}}

\newcommand{\M}{\mathsf{M}}

\renewcommand{\O}{\mathcal{O}}

\renewcommand{\P}{\mathbf{P}}
\newcommand{\calP}{\mathcal{P}}

\newcommand{\Q}{\mathbf{Q}}

\newcommand{\T}{\mathrm{T}}
\newcommand{\calT}{\mathcal{T}}

\newcommand{\W}{\mathsf{W}}
\newcommand{\sfW}{\mathsf{W}}

\newcommand{\X}{\mathsf{X}}

\newcommand{\Y}{\mathsf{Y}}
\newcommand{\sfY}{\mathsf{Y}}

\newcommand{\Z}{\mathbf{Z}}
\newcommand{\sfZ}{\mathsf{Z}}

\newcommand{\al}{\alpha}

\newcommand{\ga}{\gamma}
\newcommand{\Ga}{\Gamma}

\theoremstyle{plain}
\newtheorem{theorem}[equation]{Theorem}\numberwithin{equation}{section}
\Crefname{theorem}{{Thm}.\negthickspace}{{Ths}.\negthickspace}
\newtheorem{prop}[equation]{Proposition}
\Crefname{prop}{{Pr}.\negthickspace}{{Prs}.\negthickspace}
\newtheorem{theoremalph}{Theorem}

\Crefname{theoremalph}{{Thm}.\negthickspace}{{Thms}.\negthickspace}

\Crefname{coralph}{{Cor}.\negthickspace}{{Cors}.\negthickspace}

\newtheorem{lemma}[equation]{Lemma}
\Crefname{lemma}{{Lm}.\negthickspace}{{Lms}.\negthickspace}%\AddToHook{env/lemma/begin}{    \crefalias{equation}{Lemma}}
\newtheorem{cor}[equation]{Corollary}
\Crefname{cor}{{Cor}.\negthickspace}{{Cors}.\negthickspace}%\AddToHook{env/cor/begin}{     \crefalias{equation}{Corollary}}

\Crefname{conjecture}{{Conj}.\negthickspace}{{Conjs}.\negthickspace}%\AddToHook{env/conjecture/begin}{     \crefalias{equation}{Conjecture}}

\theoremstyle{definition}
\newtheorem{mydef}[equation]{Definition}
\Crefname{mydef}{{Def}.\negthickspace}{{Defs}.\negthickspace}

\Crefname{variation}{{Var}.\negthickspace}{{Vars}.\negthickspace}

\theoremstyle{remark}\numberwithin{equation}{section}
\newtheorem{example}[equation]{Example}
\Crefname{example}{{Ex}.\negthickspace}{{Exs}.\negthickspace}
\newtheorem{remark}[equation]{Remark}
\Crefname{remark}{{Rmk}.\negthickspace}{{Rmks}.\negthickspace}

\Crefname{nonexample}{{Nonex}.\negthickspace}{{Nonexs}.\negthickspace}

\Crefname{warning}{{Wrn}.\negthickspace}{{Wrns}.\negthickspace}
\newtheorem{question}[equation]{Question}
\Crefname{question}{{Q}.\negthickspace}{{Qs}.\negthickspace}

\Crefname{figure}{{Fig.}\negthickspace}{{Figs.}\negthickspace}
\Crefname{footnote}{{Fn.}\negthickspace}{{Fn.}\negthickspace}

\Crefname{part}{{\textsection}\negthickspace}{{\textsection}\negthickspace}
\Crefname{chapter}{{\textsection}\negthickspace}{{\textsection}\negthickspace}
\Crefname{section}{{\textsection}\negthickspace}{{\textsection}\negthickspace}
\Crefname{subsection}{{\textsection}\negthickspace}{{\textsection}\negthickspace}
\Crefname{appendix}{{\textsection}\negthickspace}{{\textsection}\negthickspace}

\usepackage{parskip}

\makeatletter
\def\thm@space@setup{%
  \thm@preskip=\parskip \thm@postskip=0pt
}
\makeatother
\begingroup
    \makeatletter
    \@for\theoremstyle:=mydef,remark,plain,theorem,defn\do{%
        \expandafter\g@addto@macro\csname th@\theoremstyle\endcsname{%
            \addtolength\thm@preskip\parskip
            }%
        }
\endgroup
\usetikzlibrary{spath3}
\tikzset{between/.style n args={2}{/tikz/spath/at end path construction={
    \tikzset{spath/split at keep middle={current}{#1}{#2}}
}}}

\begin{document}
\title{
On Galois extensions of geometric fixed point spectra
}
%On Galois extensions of geometric fixed point spectra
%Free actions on geometric fixed points
%On the global Weyl action on geometric fixed points
%

\author{\href{http://www.jmdavies.org}{Jack Morgan Davies}%\footnote{\url{davies@uni-wuppertal.de}
%\\ \emph{(dedicated to Erna Antonia)}
}
\date{\today}
\maketitle

\begin{abstract}
    In this article, the cyclotomic Galois action on topological $K$-theory adjoined with a primitive $n$th root of unity and the famous $GL_1(\Z/n)$- and $GL_2(\Z/n)$-Galois actions on topological modular forms with $\Ga_1(n)$- and $\Gamma(n)$-level structures are unified and generalised. This is done by defining a quotient stack in derived algebraic geometry parametrising constant étale subgroups of $\P$-divisible groups, and studying how these quotient stacks and various notions of torsors interact with Galois extensions formed by taking algebraic and categorical invariants. Taking global sections then yields the titular Galois extensions on $K$-geometric fixed points, recovering and refining the well-known examples above and providing new ones.
    
    As an application, the $\infty$-category of perfect modules over a variety of $H$-equivariant ring spectra $R$ are decomposed into simple pullbacks of nonequivariant categories, leading to Mayer--Vietoris sequences for localising invariants of $R$. For example, this occurs for equivariant topological $K$-theory for all $p$-groups as well as any finite nonabelian simple group of order less than 500.
\end{abstract}    

\setcounter{tocdepth}{2}
\tableofcontents

%%%%%%%%%%%%%%%%%%%%%%%%%%%%%%%%%%%%%%%%%%%%%%%%%%%%%%%%%%%%%%%
%%%%%%%%%%%%%%%%%%%%%%%%%%%%%%%%%%%%%%%%%%%%%%%%%%%%%%%%%%%%%%%
%%%%%%%%%%%%%%%%%%%%%%%%%%%%%%%%%%%%%%%%%%%%%%%%%%%%%%%%%%%%%%%
%\addcontentsline{toc}{section}{Introduction}
\section{Introduction}
Rognes' notion of a brave new Galois extension \cite{rognes_ICM_algebraic_K-theory} is one of the most pervasive and elegant concepts in higher (or derived) algebra: an action of a finite group $G$ on a commutative ring spectrum ($= \E_\infty$-ring) $B$ is called \emph{Galois} if, when viewed as an algebra over its fixed points $B^{hG} = A$, the natural map of algebras
\[B\otimes_A B \to \prod_G B, \qquad b_1\otimes b_2\mapsto (b_1 g(b_2))_{g\in G}\]
is an equivalence. The map $A \to B$ is called a \emph{$G$-Galois extension}. A key concept in this theory is \emph{Galois descent}, so studying invariants $E$ of ring spectra such that $E(A) \simeq E(B)^{hG}$ is an equivalence for each $G$-Galois extension $A \to B$.

Galois descent has many applications. For example, writing $A=\TMF$ for Hopkins' $\E_\infty$-ring of \emph{topological modular forms} \cite{hopkinsfirsttmficm,hopkinssecondtmficm}, the descent spectral sequence computing its homotopy groups can be computed from certain $\GL_2(\Z/n)$-Galois extensions on $\TMF[\tfrac{1}{n}]$ arising from $\Ga(n)$-level structures and group cohomology. Galois descent is also key to analysing spectral Grothendieck duality \cite{drewvesna,vesnaduality}, Picard groups \cite{lennartandhill,mathewstoja,tmfwls}, Brauer groups \cite{antieau_meier_stojanoska_brauergrouptmf,gepnerlawson}, and étale fundamental groups \cite{mathew_Galois}, just to name a few applications.

Even our understanding of universal invariants such as algebraic $K$-theory is strongly influenced by Galois descent. A famous conjecture of Ausoni--Rognes \cite[Conj.1.2]{chromatic_red_shift_ausoni_rognes} states that if $A \to B$ is a $G$-Galois extension of $K(n)$-local ring spectra, then the induced map on algebraic $K$-theory
\[K(A) \to K(B)^{hG}\]
should be a $K(n+1)$-local equivalence. This has been confirmed in many cases of interest in \cite{BMCSY_descent_cyclotomic_redshift,CMNN_descent_and_AR_conjecture,CMNN_descent_and_vanishing}. In fact, the failure of Ravenel's famous telescope conjecture \cite{ravenellocalisation}, disproven in \cite{telescope}, is a result of $K(n+1)$-local algebraic $K$-theory satisfying Galois hyperdescent and the failure of $T(n+1)$-local algebraic $K$-theory to satisfy Galois hyperdescent in certain instances, for $n\geq 1$.

The goal of this article is to explore new Galois actions on $\E_\infty$-rings from both an algebro-geometric and an equivariant homotopy-theoretic perspective. First, we show that for a finite abelian group $K$ and a $\P$-divisible group $\G$, the natural $\Aut(\dual{K})$-action on the \emph{stack of injections} $\Inj(\dual{K}, \G)$ defines a flat affine torsor of spectral stacks over its quotient (\Cref{main_geometric}). By applying global sections or $\infty$-categories of quasi-coherent sheaves to these torsors, we show that the associated extensions of $\E_\infty$-rings or presentably symmetric monoidal stable $\infty$-categories are Galois (\Cref{cor:galois_from_torsor}). These results are then translated into equivariant homotopy theory using Lurie's tempered cohomology theories of \cite{ec3} and their refinement to globally equivariant $\E_\infty$-rings of \cite{temperedglobal} (\Cref{main_algebraic,cor_main:globalisation}). As an application, we study when the $\infty$-categories of perfect modules of these equivariant tempered cohomology theories decompose into simple pullbacks of nonequivariant categories (\Cref{ssec:general_decompositions}). This gives a valuable inroad to computing invariants of genuinely equivariant $\E_\infty$-rings.

%%%%%%%%%%%%%%%%%%%%%%%%%%%%%%%%%%%%%%%%%%%%%%%%%%%%%%%%%%%%%%%
\subsection*{The geometric statement}
One classical way to characterise a well-behaved $G$-action on a scheme $\Y$ is to ask that the quotient stack $\Y/G$ be represented by a scheme or spectral algebraic space. In the $\infty$-category of spectral stacks, we define an \emph{affine flat $G$-torsor} over a fixed stack $\M$ as a spectral $\M$-stack $\Y$ with $G$-action such that the quotient stack $\X = \Y / G$ is relatively affine and flat over $\M$. These conditions on the structure morphism of $\X$ restrict the potential actions on $G$. For instance, we show in \Cref{pr:flatness_and_affine_torsors} that such $G$ actions are precisely those such that when pulled back to an affine over $\M$, the associated map between affine stacks induces a classical Galois extension of commutative rings on $\pi_0$.

The affine flat torsors at the heart of this article is constructed out of \emph{$\P$-divisible groups} $\G$, such as the torsion of an abelian variety, over a spectral stack $\M$. For each finite abelian group $K$, one defines a homomorphism stack $\Hom(\dual{K}, \G)$ from the constant group stack on the Pontryagin dual of $K$ into $\G$ together with its natural $\Aut(\dual{K})$-action. This does not define an affine torsor though: the zero homomorphism $\dual{K} \to \G$ has stabiliser $\Aut(\dual{K})$, and consequently, the quotient stack $\Hom(\dual{K}, \G)/\Aut(\dual{K})$ is not affine over the base $\M$.

In \cite{temperedglobal}, together with William Balderrama and Sil Linskens, we define an open substack of $\Hom(\dual{K}, \G)$ classifying injective homomorphisms $\Inj(\dual{K}, \G)$. The fundamental geometric observation of this article is that the $\Aut(\dual{K})$-action on the homomorphism stack lifts $\Inj(\dual{K}, \G)$ and defines an affine flat torsor over $\M$:

\begin{theoremalph}\label{main_geometric}
    Let $K$ be a finite abelian group and $\G$ be an oriented $\P$-divisible group over a stack $\M$. Then the quotient map
    \[\varphi\colon \Inj(\dual{K}, \G) \to \Inj(\dual{K}, \G) / \Aut(\dual{K}) = \Sub(\dual{K}, \G)\]
    is an affine flat $\Aut(\dual{K})$-torsor over $\M$ with target the \emph{stack of subgroups}.
\end{theoremalph}

The existence of an affine flat torsor immediately gives Galois extensions in several other interesting $\infty$-categories. Writing $\varphi\colon \Y \to \X$ for a $G$-torsor (such as that from \Cref{main_geometric}, for instance) then the induced map of $\E_\infty$-quasi-coherent $\O_X$-algebras
\begin{equation}\label{eq:algebra_map_intro}\O_\X \to \varphi_\ast \O_\Y\end{equation}
is a faithful $G$-Galois extension in $\CAlg(\QCoh(\X))$ (\Cref{pr:torsor_to_Galois_inQCoh}). Applying $\QCoh(-)$ as a functor also yields a faithful $G$-Galois extension
\begin{equation}\label{eq:prlalgebra_map_intro}\varphi^\ast \colon \QCoh(\X) \to \QCoh(\Y)\end{equation}
now viewed as a functor of presentably symmetric monoidal stable $\infty$-categories (\Cref{pr:torsor_to_Galois_inPRL}). In fact, the existence of these two Galois extensions needs no affineness nor flatness hypotheses on our torsor $\varphi \colon \Y \to \X$.

The titular $G$-Galois extensions in this article arise from applying the global sections functor $\Ga\colon \CAlg(\QCoh(\X)) \to \CAlg_{\Ga(\X)}$ to the Galois extension (\ref{eq:algebra_map_intro}). Generally speaking, global sections does not preserve $G$-Galois extensions, but it does if $\X$ is \emph{0-affine}, ie, the functor $\Ga$ above is an equivalence. As the $G$-torsor of \Cref{main_geometric} is relatively affine, in this case, $\X$ is 0-affine if and only if $\M$ is 0-affine. In other words, in addition to (\ref{eq:algebra_map_intro}) and (\ref{eq:prlalgebra_map_intro}), if $\Y \to \X$ is a \emph{0-affine $G$-torsor} over $\M$, meaning that $\X \to \M$ is a 0-affine map,  we have faithful $G$-Galois extensions
\begin{equation}\label{eq:perfalgebra_map_intro}\Ga(\O_\X) \to \Ga(\O_\Y), \qquad \Perf(\Ga(\O_\X)) \to \Perf(\Ga(\O_\Y))\end{equation}
between their global sections and their $\infty$-categories of perfect modules (\Cref{cor:0semiaffine_gives_galoisextension} \& \Cref{pr:torsor_to_Galois_2ring}), respectively. The first Galois extension of (\ref{eq:perfalgebra_map_intro}) takes place in the $\infty$-category of $\E_\infty$-rings $\CAlg$ and the second in the $\infty$-category of \emph{2-rings} $\tworing$, in other words, $\E_\infty$-algebras in $\Cat_\infty^\perf$, the $\infty$-category of small idempotent complete stable $\infty$-categories and exact functors of \cite{blumgeptab}.

This is where our Galois extensions on the geometric fixed points of equivariant cohomology theories come from: by interpreting the global sections of stacks of injections from \Cref{main_geometric} as the geometric fixed points of Lurie's equivariant tempered cohomology theories.

%%%%%%%%%%%%%%%%%%%%%%%%%%%%%%%%%%%%%%%%%%%%%%%%%%%%%%%%%%%%%%%
\subsection*{Tempered cohomology theories}
In \cite{ec3}, Lurie describes how a $\P$-divisible group $\G$ over a stack $\M$ gives rise to an equivariant cohomology theory, simultaneously for all finite groups, called the associated \emph{tempered cohomology theory}. Roughly speaking, the $K$-fixed points of this theory are the global sections of the homomorphism stack $\Hom(\dual{K}, \G)$ for a finite abelian group $K$. In \cite{temperedglobal,gepner2024global2ringsgenuinerefinements}, these tempered cohomology theories are shown to be represented by \emph{global $\E_\infty$-rings} $\Ga(\ul{\O}_\G)$, so a compatible system of genuinely $H$-equivariant $\E_\infty$-rings for all finite groups $H$. The notation $\Ga(\ul{\O}_\G)$ emphasises this object as an enhancement of the usual global sections $\Ga(\O_\M)$ of $\M$ courtesy of $\G$. Indeed, the underlying $\E_\infty$-ring of $\Ga(\ul{\O}_\G)$ is precisely $\Ga(\O_\M)$.

Prominent examples of such $\Ga(\ul{\O}_\G)$ include Atiyah--Segal's \cite{equivktheory} equivariant topological $K$-theories $\gKU$ and $\gKO$ as well as equivariant topological modular forms $\gTMF$ \cite{davidandlennart}. For $\gKU$, take $\M=\Spec \KU$ and $\G = \mu_{\P^\infty} \subseteq \G_m$ to be the torsion of the multiplicative group, and for $\gKO$ works over the quotient $\Spec \KU / C_2$ by the complex conjugation action on $\KU$; see \Cref{pr:universalorientedtorus}. For $\gTMF$, set $\M = \M_\Ell^\ori$ to be the moduli stack of oriented elliptic curves of \cite{ec2} and $\G=\calE[\P^\infty] \subseteq \calE$ as the torsion of the universal oriented elliptic curve. The resulting $\gTMF$ is a globally equivariant refinement of topological modular forms $\TMF = \Ga(\M_\Ell^\ori)$.

The connection to \Cref{main_geometric} is simply the fact that, just as the $K$-fixed points of $\Ga(\ul{\O}_\G)$ are given by the global sections of $\Hom(\dual{K}, \G)$, its \emph{$K$-geometric fixed points} $\Phi^K \Ga(\ul{\O}_\G)$ are given by the global sections of the stack of injections $\Inj(\dual{K}, \G)$; see \cite[Th.F]{temperedglobal}.
\[\Ga \Inj(\dual{K}, \G) \simeq \Phi^K \Ga(\ul{\O}_\G)\]
The intuition for this equivalence has long existed in the literature, both sides being a localisation of the genuine fixed points away from data induced from proper subgroups.

Applying global sections (and associated 2-rings) to the affine flat $\Aut(\dual{K})$-torsor of \Cref{main_geometric} yields our main result in equivariant stable homotopy theory.

\begin{theoremalph}\label{main_algebraic}
    Let $K$ be a finite abelian group and $\G$ be an oriented $\P$-divisible group over a $0$-affine stack $\M$ with associated global $\E_\infty$-ring $A=\Ga(\ul{\O}_\G)$. Then the tautological $\Aut(\dual{K})$-action on $\Phi^K A$ yields the faithful $\Aut(\dual{K})$-Galois extensions
    \[\Phi^K A^{h \Aut(\dual{K})} \to \Phi^K A, \qquad \Perf(\Phi^K A)^{h \Aut(\dual{K})} \to \Perf(\Phi^K A)\]
    of $\E_\infty$-rings and 2-rings, respectively. In particular, if we view $K\leq H$ as a subgroup of a finite group $H$ and if the global Weyl homomorphism
    \begin{equation}\label{eq:weyl_homomorphism} W^\gl_H K = N_H K / C_H K \to \Aut(K) = \Aut(\dual{K}); \qquad hC_H K \mapsto (k\mapsto hkh^{-1})\end{equation}
    is injective, then the maps
    \[\Phi^K A^{h W_H^\gl K} \to \Phi^K A, \qquad \Perf(\Phi^K A)^{h W_H^\gl K} \to \Perf(\Phi^K A)\]
    are faithful $W_H^\gl K$-Galois extensions.
\end{theoremalph}

\begin{remark}
    Some comments regarding \Cref{main_algebraic} are in order:
        \begin{itemize}
            \item If $K$ is nonabelian, the geometric fixed points of tempered cohomology theories $\Phi^K \Ga(\ul{\O}_\G)$ vanish, see \cite[Pr.5.4.4]{temperedglobal}, hence our focus on abelian groups $K$.
            \item The $\Aut(\dual{K})$-action on the genuine fixed points of $\Ga(\ul{\O}_\G)$ is usually not Galois, see \Cref{rmk:geometric_fixed_points_better_than_genuine}, hence our focus on geometric fixed points.
            \item Any $\Aut(\dual{K})$-action on the geometric fixed points of the global sphere spectrum $\Phi^K \Sph \simeq S^0$ is necessarily trivial, hence our focus on tempered cohomology theories.
            \item The 0-affineness hypothesis holds for affine stacks as well as those associated with $\gKO$ and $\gTMF$ mentioned above; see \cite{akhilandlennart,reconstruction}. If one wants to remove this hypothesis, then the geometric statement \Cref{main_geometric} is the more natural one.
            \item The $\Aut(\dual{K})$-action on the $K$-geometric fixed points $\Phi^K \Ga(\ul{\O}_\G)$ above comes from \Cref{main_geometric}. In \Cref{ssec:action_comparison}, we show that this agrees with the usual $\Aut(\dual{K}) = \Aut(K)$-action one defines in equivariant stable homotopy theory in many cases of interest.
            \item The Galois extensions of \Cref{main_algebraic} are natural in maps of 0-affine stacks $\M' \to \M$. Base change for geometric fixed points \cite[Th.E]{temperedglobal} further shows that the Galois extensions of $\E_\infty$-rings for $\M'$ are tensored with those for $\M$ over $\Ga(\M) \to \Ga(\M')$.
        \end{itemize}
\end{remark}

The first examples of \Cref{main_algebraic} for $\gKU$ are well-known: the spectra $\Phi^K \gKU$ vanish for noncyclic $K$, and that otherwise we have
    \[\Phi^{C_n} \gKU \simeq \KU[\tfrac{1}{n}, \zeta_n],\]
where $\zeta_n$ is a primitive $n$\textsuperscript{th} root of unity in degree zero; see \cite[\textsection7.7]{tomdieck_transformation_groups}. The $\Aut(\dual{C_n}) = GL_1(\Z/n)$-action of \Cref{main_algebraic} recovers the familiar cyclotomic Galois extension of $\E_\infty$-rings
\[\KU[\tfrac{1}{n}] \to \KU[\tfrac{1}{n}, \zeta_n].\]
Applying this example to $A=\gKO$ yields the more curious Galois extension
\[\KO[\tfrac{1}{n}] \to \KU[\tfrac{1}{n}, \zeta_n + \bar{\zeta}_n],\]
where the $\Aut(C_n)$-action on the target is a twist of the residual $\KU$-linear $\Aut(C_n)/\{\pm 1\}$-action with the usual complex conjugation action on $\KU$; see \Cref{sssec:KO}.

Applying this theorem to $\gTMF$ also produces some new Galois extensions. For any $n\geq 2$, the geometric fixed points $\Phi^{C_n} \gTMF$ are an integral lift of topological modular forms with $\Ga_1(n)$-level structure $\TMF_1(n)$. Applying \Cref{main_algebraic} with $A=\gTMF$ and $K=C_n$ yields the faithful $\Aut(C_n)$-Galois extension
    \[\Phi^{C_n} \gTMF^{h\Aut(C_n)} \to \Phi^{C_n} \gTMF,\]
an integral lift of the usual faithful $\Aut(C_n)$-Galois extension
    \[\TMF_0(n) \to \TMF_1(n)\]
from \cite[\textsection7.2]{akhilandlennart} obtained by the usual $\Aut(C_n)$-action on $\Ga_1(n)$-level structures. In particular, $\Phi^{C_3}\gTMF^{hC_2}$ is the naïve delocalisation of of Mahowald--Rezk's $\TMF_0(3)$ of \cite{levelonethree} without inverting $3$.

Applying \Cref{main_algebraic} with $A=\gTMF$ and $K=C_n \times C_n$ recovers the $GL_2(\Z/n)$-Galois extension
\[\TMF[\tfrac{1}{n}] \to \TMF(n)\]
induced by $\Ga(n)$-level structures à la \cite[\textsection 7]{km}. These and related examples are explored in more detail in \Cref{sssec:TMF}.

%%%%%%%%%%%%%%%%%%%%%%%%%%%%%%%%%%%%%%%%%%%%%%%%%%%%%%%%%%%%%%%
\subsection*{Bonus: a further globally equivariant refinement}
The geometric nature of $\Phi^K \Ga(\ul{\O}_\G)$, the fact it arises as the global sections of a stack $\Inj(\dual{K}, \G)$ equipped with its own oriented $\P$-divisible group pulled back from $\M$, allows us to enhance \Cref{main_algebraic} further to a statement about the tempered cohomology theory internal to global spectra. In fact, the proof of the following theorem is just an application of \Cref{main_algebraic} together with the base change for geometric fixed points of \cite[Th.E]{temperedglobal}.

Write $\G_\sub$ for the base change of $\G$ from $\M$ to $\Sub(\dual{K}, \G)$ and $\G_\inj$ for the further base change to $\Inj(\dual{K}, \G)$.

\begin{theoremalph}\label{cor_main:globalisation}
    In the situation of \Cref{main_algebraic}, the map of ($\pi$-ambidextrous) global $\E_\infty$-rings
    \[\Ga(\ul{\O}_{\G_\sub}) \to \Ga(\ul{\O}_{\G_\inj})\]
    is a faithful $W_H^\gl K$-Galois extension recovering \Cref{main_algebraic} on underlying $\E_\infty$-rings.
\end{theoremalph}

This is a lot of additional structure compared to \Cref{main_algebraic}. For instance, that the $W_H^\gl K$-action of $\E_\infty$-rings on $\Phi^K \Ga(\G)$ lifts to a refinement in global $\E_\infty$-rings. By restricting from global spectra to genuine $G$-spectra for some finite group $G$, this induces a $W_H^\gl K$-Galois extension of $G$-equivariant $\E_\infty$-rings; we hope to come back to more examples of this in future work. There is yet another generalisation from global $\E_\infty$-rings to $\pi$-ambidextrous global $\E_\infty$-rings, see \Cref{cor:globalisations}, which also shows that all of the structure of the Galois extensions of \Cref{cor_main:globalisation} commute with all transfers induced by relatively $\pi$-finite morphisms of global spaces, see \cite[\textsection7.4]{ec3}, such as $\pi_1$-epimorphisms $BH \to BK$.

One can further iterate these ideas by applying \Cref{main_geometric} to $\G=\G_\sub$ or $\G_\inj$, producing a stack of injections over a stack of injections, and so on. It seems likely that the associated functor $\Cat(\Glo_\ab)^\otimes_{\pi-\st}$, see \cite[Df.1.3.1]{temperedglobal}, is also a $W_H^\gl K$-Galois extension. We leave these more eccentric avenues for another time and instead focus on other applications.

%%%%%%%%%%%%%%%%%%%%%%%%%%%%%%%%%%%%%%%%%%%%%%%%%%%%%%%%%%%%%%%
\subsection*{Decompositions of equivariant module categories}
Our main application of the torsors and Galois extensions above is to produce simple decompositions of $\infty$-categories of perfect $H$-equivariant modules over tempered cohomology theories. The goal here is to aid in the computation of invariants of $H$-equivariant ring spectra, such as Picard groups and algebraic $K$-theory, by decomposing these $\infty$-categories of perfect modules into more computable nonequivariant $\infty$-categories.

This story starts with \cite{krause_equivariantpicard,NPR2024}, wherein Krause and Naumann--Pol--Ramzi show that for a finite group $H$ and an $\E_\infty$-$H$-ring spectrum $A$, the category of perfect $A$-modules
\[\Perf_H(A) = \Mod_A(\Sp_H)^\omega\]
can be expressed as an iterated pullback of $\Perf(\Phi^K A)^{hW_H K}$ for subgroups $K\leq H$; here we do mean the \emph{classical Weyl group} $W_H K = N_H K / K$ and not its global variant $W_H^\gl K$. We call this a \emph{KNPR-decomposition} of $\Perf_H(A)$. The gluing data in these pullbacks is the Tate construction
\[\Perf(\Phi^K A)^{tW_H K} \in \tworing.\]
If many such Tate constructions vanish, then many pullbacks degenerate into products.

A simple instance of this is classical and occurs when the order of $H$ is invertible in $R$. In this case, the above Tate construction of 2-rings vanishes as Tate cohomology is $|H|$-torsion, giving
\[\Perf_H(R) {\simeq} \prod_{(K)\leq H} \Perf(\Phi^K R)^{hW_H K},\]
where the product ranges over conjugacy classes of subgroups of $H$; see \Cref{pr:rationality_decomposition}.

More generally, if $A \to B$ is a faithful $G$-Galois extension of $\E_\infty$-rings, then the Tate constructions $B^{tG}$ and $\Perf(B)^{tG}$ both vanish. By relating the actions of both $W_H K$ and $W_H^\gl K$ on $\Phi^K \Ga(\ul{\O}_\G)$, one can then use \Cref{main_algebraic} to drastically simplify various KNPR-decomposition associated with $\Perf_H(\Ga(\ul{\O}_\G))$.

\begin{example}\label{ex:intro_p_ell}
    Let $p,\ell$ be two primes such that $\ell | p-1$. Then for $A=\Ga(\ul{\O}_\G)$, the category $\Perf_{C_p\rtimes C_\ell}(A)$ sits in the pullback of 2-rings
\[\begin{tikzcd}
    {\Perf_{C_p\rtimes C_\ell}(A)}\ar[r]\ar[d]    &   {\Perf(\Phi^{C_\ell}A) \times \Perf(\Phi^{C_p}A^{hC_\ell})}\ar[d]  \\
    {\Perf(\Ga(\M))^{hC_p\rtimes C_\ell}}\ar[r] &   {\Perf(\Ga(\M))^{tC_p\rtimes C_\ell},}
\end{tikzcd}\]
    where the group actions on the lower row are trivial; see \Cref{pr:categorical_decomposition}. In this case, the Galois extensions of \Cref{main_algebraic} combined with Galois descent identify
    \[\Perf(\Phi^{C_p} A^{hC_\ell}) \simeq \Perf(\Phi^{C_p} A)^{hC_\ell}\]
    in the upper-right corner. For $\ell=2$ and $p$ odd, then $C_p\rtimes C_2=D_{2p}$ is the dihedral group of order $2p$ and for $p=3$ this is the symmetric group $S_3$. These groups have periodic cohomology, which could greatly aid in computing invariants of the lower row.
\end{example}

The above example works for a general $\G$, however, by specialising to equivariant topological $K$-theories $A=\gKU$ or $\gKO$, one obtains decompositions for many more finite groups $H$. These include all $p$-groups (\Cref{pr:KU_decomp_abelian}), with more details for $Q_8$ (\Cref{pr:KU_decomp_Q8}) and $D_8$ (\Cref{pr:KU_decomp_D8}), low order alternating and symmetric groups (\Cref{pr:KU_decomp_A4,pr:KU_decomp_S4,pr:KU_decomp_A5,pr:KU_decomp_A6}), and finally, the last remaining nonabelian simple group of order $\leq 500$ (\Cref{pr:KU_decomp_GL3F2}):

%\begin{example}
%    For $A=\gKU$ or $\gKO$, these pullbacks can be made even more explicit:
%\[\begin{tikzcd}
%    {\Perf_{D_{2p}}(A)}\ar[r]\ar[d]    &   {\substack{\Perf(A[\tfrac{1}{2}]) \\ \times \\ \Perf(A[\tfrac{1}{p}, \zeta_p + \bar{\zeta}_p])}}\ar[d]  \\
%    {\Perf(A)^{hD_{2p}}}\ar[r] &   {\Perf(A)^{tD_{2p}},}
%\end{tikzcd} \qquad \begin{tikzcd}
%    {\Perf_{C_7\rtimes C_3}(A)}\ar[r]\ar[d]    &   {\substack{\Perf(A[\tfrac{1}{3}, \omega]) \\ \times \\ \Perf(A[\tfrac{1}{7}, \sqrt{-7}])}}\ar[d]  \\
%    {\Perf(A)^{hC_7\rtimes C_3}}\ar[r] &   {\Perf(A)^{tC_7\rtimes C_3}.}
%\end{tikzcd}\]
%\end{example}

%\begin{example}
%    If $p$ is a prime and $H$ is any finite $p$-group and $A=\gKU$ or $\gKO$, then
%    \[\begin{tikzcd}
%        {\Perf_{H}(A)}\ar[r]\ar[d]    &   {\prod_{\substack{C\leq H \\ \mathrm{cyclic}}} \Perf(\Phi^{C}A)^{hW_H C}}\ar[d]    \\
%        {\Perf(A)^{hH}}\ar[r]         &   {\Perf(A)^{tH}}
%    \end{tikzcd}\]
%    is a pullback of 2-rings, where $C$ varies over conjugacy classes of cyclic subgroups of $H$.
%\end{example}

\begin{example}
    For $A=\gKU$ or $\gKO$, a pullback of 2-rings
    \[\begin{tikzcd}
        {\Perf_{\GL_3(\F_2)}(A)}\ar[r]\ar[d]    &   {\Perf(A[\tfrac{1}{2}])^{hV} \times \Perf(A[\tfrac{1}{2}]) \times \Perf(A[\tfrac{1}{3}]) \times \Perf(\Phi^{C_7} A^{hC_3})}\ar[d]    \\
        {\Perf(A)^{h\GL_3(\F_2)}}\ar[r]         &   {\Perf(A)^{t\GL_3(\F_2)},}
    \end{tikzcd}\]
    where $V=C_2\times C_2$ is Klein's Vierergruppe, and one can identify
    \begin{equation}\label{eq:c7geo}\Phi^{C_7} \gKU^{hC_3} \simeq \KU[\tfrac{1}{7}, \sqrt{-7}], \qquad \Phi^{C_7} \gKO^{hC_3} \simeq \KU[\tfrac{1}{7}].\end{equation}
\end{example}

As already highlighted, such decompositions have computational consequences. In forthcoming work with Luca Pol, we will use some of these decompositions to help compute Picard groups of various $\Perf_H(A)$. More broadly speaking, all of these simplified KNPR-decompositions immediately lead to fibre sequences for localising invariants $E$, such as nonconnective algebraic $K$-theory, topological Hochschild homology, and $p$-typical topological cyclic homology; see \Cref{rmk:localising_invariants}. For example, if $A$ is either $\gKU$ or $\gKO$, then for any localising invariant $E$, we have fibre sequences of spectra
    \[E_{Q_8}(A) \to E(\Perf(A)^{hQ_8}) \times E(\Perf(A[\tfrac{1}{2}])^{hV}) \times \prod_{x=i,j,k} E(A[\tfrac{1}{2}]) \to E(\Perf(A)^{tQ_8}),\]
where $E_{Q_8}(A)=E(\Perf_{Q_8}(A))$. We hope these Mayer--Vietoris sequences can facilitate more computations of such invariants in the future.

%%%%%%%%%%%%%%%%%%%%%%%%%%%%%%%%%%%%%%%%%%%%%%%%%%%%%%%%%%%%%%%
\subsection*{Outline}
This article is broken up into three sections: we begin with pure spectral algebraic geometry to prove \Cref{main_geometric} (\Cref{sec:SAG}), then translate this result into equivariant homotopy theory using Lurie's tempered cohomology theories and prove \Cref{main_algebraic,cor_main:globalisation} (\Cref{sec:ESHT_over_a_stack}), and finally, we review KNPR-decompositions and give a catalogue of simplifications of these categorical decompositions as an application of \Cref{main_algebraic} (\Cref{sec:categorical_decompositions}). In a short appendix (\Cref{appendix}), we summarise some permanence properties for morphisms in spectral algebraic geometry used in this article.

In more detail, after setting up the basic theory of torsors, the main theme of \Cref{ssec:torsor} is the interaction between torsors and various notions of Galois extensions. We begin by showing that $\pi_0$-Galois extensions yield affine torsors (\Cref{pr:pi0galois_to_torsor,pr:converse}) as do some Galois extensions of Rognes (\Cref{pr:chromatic_positive_answer}). Then we explore the Galois extensions in various $\infty$-categories induced by various flavours of torsor in \Cref{ssec:relationtoGalois}. In \Cref{ssec:construction_of_moduli}, we define the stack of subgroups and prove \Cref{main_geometric} that the stack of injections equipped with its tautological action defines a flat affine torsor. Putting these two subsections together, we obtain many examples of Galois extensions in various $\infty$-categories courtesy of \Cref{main_geometric} (\Cref{cor:galois_from_torsor}).

The passage from oriented $\P$-divisible groups in spectral algebraic geometry to equivariant homotopy theory via Lurie's tempered cohomology theories is reviewed in \Cref{ssec:from_tempered_to_global_spectra}; many more details can be found in \cite{temperedglobal}, however, the focus here is on the canonical residual actions on genuine and geometric fixed points and their relationship with tempered cohomology theories. This is mostly expository. The proof of \Cref{main_algebraic} then follows straight from \Cref{cor:galois_from_torsor}, and \Cref{cor_main:globalisation} is an exercise in global homotopy theory. In \Cref{ssec:action_comparison}, we detail when the above residual actions on fixed points agree with residual Weyl group actions classically studied in equivariant homotopy theory. Examples of the Galois extensions from \Cref{main_algebraic} are given in \Cref{ssec:examples}, focusing on those involving equivariant $K$-theory and equivariant elliptic cohomology.
    
Lastly, \Cref{sec:categorical_decompositions} is concerned with decompositions of $\infty$-categories of modules over equivariant ring spectra. We begin by reviewing KNPR-decompositions (\Cref{thm:npr_decompositions}), setting the scene for our decomposition (\Cref{cor:generic_decomposition}), and proving some simple algebraicity statements (\Cref{pr:rationality_decomposition}). In \Cref{ssec:general_decompositions}, we give two main families of examples of these decompositions. The first for split groups and any tempered cohomology theory (\Cref{pr:categorical_decomposition}), and the second for equivariant $K$-theory and a handful of explicit groups, ranging from all $p$-groups (\Cref{pr:KU_decomp_abelian,pr:KU_decomp_Q8,pr:KU_decomp_D8}) to all nonabelian simple groups of order less than $500$ (\Cref{pr:KU_decomp_A4,pr:KU_decomp_S4,pr:KU_decomp_A5,pr:KU_decomp_A6,pr:KU_decomp_GL3F2}).

%%%%%%%%%%%%%%%%%%%%%%%%%%%%%%%%%%%%%%%%%%%%%%%%%%%%%%%%%%%%%%%
\subsection*{Notation}
Throughout, we freely use the higher categorical language, so by a ``category'' we mean an ``$\infty$-category'' and suppress the mapping space notation by writing $\calC(-,-) = \Map_{\calC}(-,-)$. The category of presentable categories and left adjoint functors is written as $\PrL$, $\PrLst$ for the subcategory of stable presentable categories, and $\PrLstomega$ for the further subcategory of compactly generated stable presentable categories and left adjoint functors which preserve compact objects. The category of small stable idempotent complete categories and exact functors (featuring prominantly in \cite{blumgeptab}) is denoted by $\Cat^\perf$. Following \cite{mathew_Galois}, we write $\CAlg(\Cat^\perf) = \tworing$. We will use the symmetric monoidal identification $\Cat^\perf \simeq \PrLstomega$ by taking Ind-categories and compact objects.

We reference Rognes' $G$-Galois theory \cite{johnrognessdualisinggorups} for a finite group $G$ in the more general case of a semiadditive presentably symmetric monoidal $\infty$-category; the generalised proofs are the same. The elements $\zeta_n$ in a discrete ring will always be $n$th primitive roots of unity, with $\omega = \zeta_3$ and $i = \zeta_4$.

\begin{center}
    \emph{Fix finite groups $G,H$ and an abelian subgroup $K\leq H$ for this article.}
\end{center}

Write $\Stk=\Stk_\fpqc(\CAlg^\op)$ for the category of (fpqc-) stacks of \cite[Def.3.1.1.3]{reconstruction} and $\Stk^\heartsuit = \Stk_\fpqc(\CAlg^{\heartsuit,\op})$ for the category of classical stacks. A colimit over a functor $B\calG \to \Stk$ hitting a stack $\Y$ will be written as $\Y/ \calG$; the notation $\Y//\calG$ only appears in the context of \Cref{df:g-torsor} where $\calG$ is a group object in $\Stk$.

We use Lurie's notion of an oriented $\P$-divisible group of \cite{ec3} and its extension to stacks of \cite[\textsection4]{temperedglobal} (generalising that of \cite{elltempcomp}). Following \cite[Not.4.1.6]{temperedglobal}, given a preoriented $\P$-divisible group $\G$, we write $\G[-] \colon \Ab_\fin^\op \to \Stk$ for the underlying $\P$-divisible group and $\G(-)\colon \Glo_\ab \to \Stk$ for preorientation factoring $\G[-]$ through the classifying space and Pontryagin dual functor $B(\dual{-}) \colon \Ab_\fin^\op \to \Glo_\ab$; here $\Glo_\ab \subseteq \Spc$ is the subcategory of spaces spanned by those of the form $BG$ for a finite abelian group $G$. We abbreviate $\Ga(\O_\M) = \Ga(\M)$ and the notation $\Ga(\ul{\O}_\G)$ of \cite[Thm.A]{temperedglobal} as $\Ga(\G)$. 

\begin{center}
    \emph{Fix a base stack $\M$ and an oriented $\P$-divisible group $\G$ for this article.}
\end{center}

%%%%%%%%%%%%%%%%%%%%%%%%%%%%%%%%%%%%%%%%%%%%%%%%%%%%%%%%%%%%%%%
\subsection*{Acknowledgements}
Thank you to William Balderrama, Christian Carrick, Sil Linskens, and Luca Pol for our enjoyable and continuing collaborations that all influenced this work and my thinking in their own ways. Thank you as well to Jens Hornbostel, Magdalena K\c{e}dziorek, Shai Keidar, and Tommy Lundemo for some helpful conversations and exchanges, and doubly so to Christian, Sil, and Tommy for comments on a draft.

I would like to dedicate this article to Erna, with whom my mind was throughout the gestation period of this article.

Whilst writing this article, I was supported by the DFG-funded research training group GRK 2240: Algebro-Geometric Methods in Algebra, Arithmetic and Topology.

%%%%%%%%%%%%%%%%%%%%%%%%%%%%%%%%%%%%%%%%%%%%%%%%%%%%%%%%%%%%%%%
%%%%%%%%%%%%%%%%%%%%%%%%%%%%%%%%%%%%%%%%%%%%%%%%%%%%%%%%%%%%%%%
%%%%%%%%%%%%%%%%%%%%%%%%%%%%%%%%%%%%%%%%%%%%%%%%%%%%%%%%%%%%%%%
\section{Torsors and the moduli of subgroups}\label{sec:SAG}
One of the main focuses of this article is the systematic study of various notions of torsor in spectral algebraic geometry. We begin with the most general set-up to highlight which phenomenæ occur under which precise hypotheses.

%%%%%%%%%%%%%%%%%%%%%%%%%%%%%%%%%%%%%%%%%%%%%%%%%%%%%%%%%%%%%%%
\subsection{Basics of \texorpdfstring{$\calP$-$\calG$-}{P-G-}torsors}\label{ssec:torsor}
Recall the notions of a \emph{group object} \cite[Df.7.2.2.1]{htt}, \emph{groupoid object} \cite[Df.6.1.2.7]{htt}, and an \emph{action} of a group object \cite[Df.3.1]{nikolaus_schreiber_stevenson_general_theory}. This latter reference takes place in a topos, but it also makes perfect sense in the category $\Stk$ and its slices $\Stk_{/\M}$.

\begin{mydef}\label{df:g-torsor}
    Let $\calG$ be a group object in $\Stk_{/\M}$ and $\calP$ be a property of morphisms of stacks. A \emph{$\calP$-$\calG$-action} on an $\M$-stack $\Y$ is a $\calG$-action such that the structure map from the quotient $\Y//\calG \to \M$ satisfies $\calP$. A \emph{$\calP$-$\calG$-torsor} (over $\M$) is a map $\varphi \colon \Y \to \X$ in $\Stk_{/\M}$ together with a $\calP$-$\calG$-action on $\Y$ inducing an equivalence $\Y // \calG \simeq \X$. If $\calP$ is the vacuous property, ie, all morphisms of stacks are $\calP$, then we abbreviate these notions to \emph{$\calG$-action} and \emph{$\calG$-torsor}.
\end{mydef}

The variables $\calP$ and $\M$ can be used to force the quotient stack $\Y// \calG$ to be well-behaved, such as be representable by an affine or nonconnective spectral scheme. For instance, the universal $\calG$-torsor $\BG = \M // \calG$ is not an affine torsor if $\calG$ is not trivial.

\begin{remark}
    The discussion of torsors in \Cref{sssec:highprin,sssec:closureproperties} also applies in the context of $\Stk_\tau(\calC)$ of \cite[Def.2.1.1.5]{reconstruction} with only cosmetic changes; \Cref{ssec:relationtoGalois} should also broadly apply to generic quasi-coherent sheaf contexts (\cite[Df.2.1.3.2]{reconstruction}). For readability, we simply write and refer to $\Stk$.
\end{remark}

\subsubsection{Higher principality}\label{sssec:highprin}
As highlighted in \cite[\textsection 3]{nikolaus_schreiber_stevenson_general_theory}, this notion of a $\calG$-torsor automatically satisfies the principality condition that the map
\[\calG \times_\M \Y \xrightarrow{\simeq} \Y \times_\X \Y,\]
defined by the $\calG$-action on the left factor and the projection on the right factor, is an equivalence. Indeed, given a $\calG$-torsor $\Y \to \X$, consider the two commutative diagrams of $\M$-stacks
\[\begin{tikzcd}
    {\calG \times_\M \Y}\ar[r]\ar[d]    &   {\calG}\ar[r]\ar[d] &   {\M}\ar[d]  \\
    {\Y}\ar[r]                          &   {\M}\ar[r]          &   {\BG = \M / \calG}
\end{tikzcd}\qquad
\begin{tikzcd}
    {\Y \times_\X \Y}\ar[r]\ar[d]   &   {\Y}\ar[d]\ar[r]    &   {\M}\ar[d]  \\
    {\Y}\ar[r]                      &   {\X}\ar[r]          &   {\BG.}
\end{tikzcd}
\]
Starting with the left diagram, the left square is Cartesian by definition and the right square is Cartesian as $\Omega_\M \BG \simeq \calG$ from the universality of colimits and effectiveness of groupoids in $\Stk$, see \cite[Pr.2.1.2.9(2)]{reconstruction}; this reminds us of the classical fact that $\Omega BG \simeq G$ in spaces for a topological group $G$. On the other hand, in the right diagram, the left square is Cartesian by definition, and the right square is also Cartesian by the universality of colimits; this reminds us of the fibre sequence of spaces $X \to X_{hG} \to BG$. As the map $\Y \to \X \to \BG$ factors through $\M$, and as the two outer rectangles are Cartesian, we see that the upper-left objects are naturally equivalent.

The coherent version of this principality condition is the following:

\begin{prop}\label{pr:higher_principality}
    Given a group object $\calG$ in $\Stk_{/\M}$ and a $\calG$-torsor $\varphi\colon \Y \to \X$, then there is a natural identification of simplicial $\M$-stacks
\[(\Y//\calG)_\bullet = \left(\Y \Leftarrow \calG \times_\M \Y \Lleftarrow \calG^2 \times_\M \Y \cdots\right) \simeq \left(\Y \Leftarrow \Y \times_\X \Y \Lleftarrow \Y^{\times_\X 3} \cdots\right) = \Cech(\Y \to \X)\]
    between the groupoid stack defining the $\calG$-action on $\Y$ and the Čech nerve of $\varphi$. In particular, $\X \simeq \Y // G \simeq D_\varphi$, where $D_\varphi$ is the \emph{descent stack of $\varphi$} (\cite[Df.2.3.1.1]{reconstruction}) and $\varphi$ is an effective epimorphism.
\end{prop}

This is just a redressing of \cite[Pr.3.7]{nikolaus_schreiber_stevenson_general_theory} that is phrased with topoi (which $\Stk$ is not).

\begin{proof}
    The fact that groupoids are effective in $\Stk$, see \cite[Pr.2.1.2.9(3)]{reconstruction}, shows that the groupoid object $(\Y//\calG)_\bullet$, defined by the fact that $\Y$ has a $\calG$-action, is equivalent to the Čech nerve of $\varphi$ as simplicial $\M$-stacks, giving the desired idenfitication. The identification of $D_\varphi$ with the colimit of this Čech nerve is \cite[Lm.2.3.1.2]{reconstruction}.
\end{proof}

\subsubsection{General closure properties}\label{sssec:closureproperties}
For a reminder of some of the following notions, we refer the reader to \Cref{df:propertiy_of_morphisms}.

\begin{prop}\label{pr:G-torsors_and_basechange}
    Let $\calG$ be a group objects in $\Stk_{/\M}$, $\calP$ a property of morphisms of stacks, and $\varphi \colon \Y \to \X$ be a $\calP$-$\calG$-torsor.
    \begin{enumerate}
        \item If $\calP$ satisfies base change, then for any map of stacks $f\colon \M' \to \M$, the pullback of the $\calG$-action on $\Y$ along $f$ defines a $\calP$-$\calG$-torsor over $\M'$.
        \item If $g\colon \X' \to \X$ is a map such that the composite $\X' \to \X \to \M$ satisfies $\calP$, then the pullback of the $\calG$-action on $\Y$ along $g$ defines a $\calP$-$\calG$-torsor over $\M$.
        \item If $\calP$ satisfies descent, is closed under base change, and the structure map $\calG \to \M$ satisfies $\calP$, then $\varphi$ satisfies $\calP$.
    \end{enumerate}
\end{prop}

\begin{proof}
    Parts 1 and 2 follow by definition. For part 3, to show that $\varphi$ satisfies $\calP$, it suffices to show that its pullback against itself satisfies $\calP$, as $\varphi$ is an effective epimorphism by \Cref{pr:higher_principality}. Another application of \Cref{pr:higher_principality} identifies this pullback with the projection $\Y\times_\M \calG \to \Y$, which satisfies $\calP$ as the pullback of $\calG \to \M$ satisfying $\calP$ by assumption.
\end{proof}

Part 3 of \Cref{pr:G-torsors_and_basechange} applied to affine torsors (using part 1 of \Cref{pr:permanenceproperties_for_morphisms}) immediately yields:

\begin{cor}\label{pr:torsors_are_relatively_affine}
    Let $\calG$ be a group object in $\Stk_{/\M}$ and $\varphi \colon \Y \to \X$ be an affine $\calG$-torsor over $\M$. If the structure map $\calG \to \M$ is affine, then $\varphi$ is affine.
\end{cor}

One way to produce group objects in $\Stk_{/\M}$ is by sending a group object $\calG$ in spaces along the unique colimit-preserving functor $(\underline{-}) \colon \Spc \to \Stk_{/\M}$ that preserves terminal objects; we often omit the underline from our notation. In this situation, we can further identify the $\Y//\calG$ with the colimit of the functor $B\calG \to \Stk_{/\X}$ defining the $\calG$-action on $\Y$:
\[\Y // \underline{\calG} = \colim (\Y//\underline{\calG})_\bullet \simeq \underset{B\calG}{\colim} \Y = \Y/G\]
This allows for further easy manipulation of $\calG$-torsors.

\begin{cor}\label{cor:qcoh_and_fixedpoints}
    Given a group object $\calG$ in spaces and a $\calG$-torsor $\varphi\colon \Y \to \X$ over $\M$, then pullback $\varphi^\ast \colon \QCoh(\X) \to \QCoh(\Y)$ induces an equivalence $\QCoh(\X) \simeq \QCoh(\Y)^{h\calG}$.
\end{cor}

\begin{proof}
    As $\QCoh(-)\colon \Stk^\op \to \PrLst$ sends colimits of stacks to limits of categories we have
    \[\QCoh(\X) \simeq \QCoh(\Y//\calG) \simeq \QCoh(\underset{B\calG}{\colim} \Y) \simeq \lim_{B\calG} \QCoh(\Y) = \QCoh(\Y)^{h\calG}.\qedhere\]
\end{proof}

\subsubsection{Deligne--Mumford stacks and classical stacks}
As long as the group in spaces $\calG = G$ is a discrete finite group, the notion of a $G$-torsor plays well with Deligne--Mumford stacks.

\begin{prop}\label{pr:dm_and_torsors}
    The functor of points $h \colon \SpDM^\nc \to \Stk$ of \cite[Pr.3.2.2.2]{reconstruction} sends $G$-torsors to $G$-torsors. In particular, if $\X$ is a nonconnective spectral Deligne--Mumford stack with $G$-action, then the quotient stack $\X/G$ can equivalently be formed in $\SpDM^\nc$ or $\Stk$.
\end{prop}

This is intuitive, as the map $\X \to \X/G$ is a finite étale cover (\Cref{cor:finiteetaleness}), which can be used as an atlas for $\X/G$. As mentioned in \cite[Not.3.2.2.3]{sag}, quotient stacks $\X/G$ always exist in $\SpDM^\nc$.

\begin{proof}
    As $h$ preserves finite limits, it suffices to check that the quotients by $G$ taken in $\SpDM^\nc$ agree with those taken in $\Stk$. However, \cite[Pr.3.2.2.2(2)]{reconstruction} states that $h$ sends colimits of diagrams of nonconnective spectral Deligne--Mumford stacks consisting of étale morphisms to colimits in $\Stk$, and since $BG$ is a groupoid, a functor $BG \to \SpDM^\nc$ always factors through the wide subcategory of $\SpDM^\nc$ spanned by étale morphisms, so we are done.
\end{proof}

There is an analogous definition of torsors in $\Stk^\heartsuit$ following \Cref{df:g-torsor}. By \cite[Df.3.1.1.7]{reconstruction}, there is a colimit-preserving functor $(-)^\heartsuit \colon \Stk \to \Stk^{\heartsuit}$ characterised by its value on affines $\Spec A^\heartsuit = \Spec \pi_0 A$. As long as some flatness is assumed, then $(-)^\heartsuit$ preserves torsors.

\begin{prop}\label{pr:torsors_and_hearts}
    Let $\varphi\colon \Y \to \X$ be a $G$-torsor in $\Stk_{/\M}$ such that $\Y \to \M$ is a flat map between geometric stacks (\cite[Df.3.4.2.1]{reconstruction}). Then $\Y^\heartsuit \to \X^\heartsuit$ is a $\calG$-torsor in $\Stk_{/\M^\heartsuit}^\heartsuit$.
\end{prop}

\begin{proof}
    As $(-)^\heartsuit$ preserves colimits, it suffices to show that the simplicial objects of $\Stk^\heartsuit$
    \[(\Y^\heartsuit//\calG)_\bullet = \left(\Y^\heartsuit \Leftarrow \calG \times_{\M^\heartsuit} \Y^\heartsuit \Lleftarrow \cdots\right), \qquad (\Y//\calG)_\bullet^\heartsuit = \left(\Y^\heartsuit \Leftarrow (\calG \times_{\M} \Y)^\heartsuit \Lleftarrow \cdots\right)\]
    agree. In other words, we have to show that $(-)^\heartsuit \colon \Stk_{/\M} \to \Stk^\heartsuit_{/\M^\heartsuit}$ commutes with the above products, which is a consequence of \cite[Pr.3.4.3.3]{reconstruction}.
\end{proof}

\subsubsection{Affine torsors and \texorpdfstring{$\pi_0$}{pi0}-Galois extensions}
Standard examples of affine torsors arise from certain {Galois extensions}. Recall that $G$ is a finite discrete group.

\begin{mydef}\label{df:lurie_galois}
    A \emph{$\pi_0$-$G$-Galois action} on an $\E_\infty$-ring $B$ is a $G$-action such that the induced $G$-action on $\pi_0 B$ is free\footnote{Recall that an action of a group $G$ on a discrete commutative ring $R$ is \emph{free} if for all nonzero rings $A$, the set of homomorphisms $\Hom(R,A)$ is acted on freely by $G$.}. In this case, we write $A=B^{hG}$ and call the map of $\E_\infty$-rings $A \to B$ a \emph{$\pi_0$-$G$-Galois extension}.
\end{mydef}

By \cite[Pr.B.7.5.5]{sag}, if $A \to B$ is a $\pi_0$-$G$-Galois extension, then it is also finite étale and faithfully flat, the map $\pi_n A \to \pi_n B$ is the inclusion of $G$-invariant elements, and the natural map
\begin{equation}\label{eq:pi0galoisisgalois}
B\otimes_A B \to \prod_G B
\end{equation}
is an equivalence.

\begin{prop}\label{pr:pi0galois_to_torsor}
    Given a $\pi_0$-$G$-Galois extension of $\E_\infty$-rings $A\to B$, the induced map of stacks $\Spec B \to \Spec A$ is an affine $G$-torsor over $\Spec \Sph$. Moreover, if $B$ is a flat $\E_\infty$-algebra over an $\E_\infty$-ring $A'$ and $G$-acts through $A'$-algebra maps, then $\Spec B \to \Spec A$ is an affine flat $G$-torsor over $\Spec A'$.
\end{prop}

Using \Cref{pr:dm_and_torsors}, one can show that $\Spec B /G =\Spec A$ by referring to \cite[Pr.3.2.2.5]{sag}, however, one can also argue directly.

\begin{proof}
    The stack $\Spec A$ is clearly affine, so it suffices to identify $\Spec B / G$ with $\Spec A$. As $A \to B$ is faithfully flat, we can present $\Spec A$ as the colimit of the Čech nerve $\check{C}(A\to B)$ of $\Spec B \to \Spec A$. On the other hand, the quotient $\Spec B / G$ is given by the colimit of the action groupoid $(\Spec B // G)_\bullet = G^{\times \bullet} \times \Spec B$. It then suffices to show that the natural map of simplicial objects $(\Spec B // G)_\bullet \to \check{C}(A\to B)$ is an equivalence. As both are groupoids and groupoids are effective in $\Stk$ by \cite[Pr.2.1.2.9(3)]{reconstruction}, it suffices to check this identification for $\bullet=0,1$. For $\bullet=0$, the map is the identity on $\Spec B$, and for $\bullet=1$, this is the equivalence (\ref{eq:pi0galoisisgalois}).
    %By \Cref{pr:dm_and_torsors}, this quotient can be taken in $\Stk$ or in nonconnective spectral Deligne--Mumford stacks, which is identified with $\Spec B^{hG} = \Spec A$ courtesy of \cite[Pr.3.2.2.5]{sag}. 
    
    For the ``moreover'' statement, to check that $A' \to A$ is flat, it suffices to check after the faithfully flat extension $A \to B$, as flat morphisms are local on the source with respect to the flat topology \cite[Lm.B.1.4.2(2)]{sag}, and this composition is flat by assumption.
\end{proof}

This statement has a converse; thank you to Sil Linskens for asking about such a statement.

\begin{prop}\label{pr:converse}
    Let $A \to B$ be a flat morphism of $\E_\infty$-rings and let $G$ act on $B$ through $\E_\infty$-$A$-algebra maps. Then $\Spec B \to \Spec A$ is an affine $G$-torsor over $\Spec \Sph$ if and only if $A \to B$ is a $\pi_0$-$G$-Galois extension.
\end{prop}

\begin{proof}
    If $A \to B$ is a $\pi_0$-$G$-Galois extension, then $\Spec B \to \Spec A$ is a $G$-torsor by \Cref{pr:pi0galois_to_torsor}. Conversely, suppose that $\Spec B \to \Spec A$ is an affine $G$-torsor over $\Spec \Sph$. As this map is also flat as a map between affine (hence also geometric) stacks, \Cref{pr:torsors_and_hearts} shows that $\Spec \pi_0 B \to \Spec \pi_0 A$ is a $G$-torsor in $\Stk^\heartsuit_{/\Spec \pi_0 A}$ and is also clearly an affine $G$-torsor over $\Spec \Z$. As the quotient $\Spec \pi_0 A$ and the action groupoid $(\Spec \pi_0 B // G)_\bullet$ in $\Stk^\heartsuit$ are all classical affine schemes, this data defines a $G$-torsor in classical affine schemes. In particular, this immediately shows that $\pi_0 A \to \pi_0 B$ is a $G$-Galois extension in discrete commutative rings, as the equivalence of rings $\pi_0 B \otimes_{\pi_0 A} \pi_0 B \simeq \prod_G \pi_0 B$ follows from the equivalences
    \[\Spec \left(\prod_G \pi_0 B\right) \simeq \coprod_G \Spec \pi_0 B \simeq \Spec \pi_0 B \times_{\Spec \pi_0 A} \Spec \pi_0 B \simeq \Spec (\pi_0 B \otimes_{\pi_0 A} \pi_0 B),\]
    which hold for free in $\Stk^\heartsuit$ from the same argument as in \Cref{pr:higher_principality}. By \cite[Rmk.B.7.5.3]{sag}, this implies that $A \to B$ is also $\pi_0$-$G$-Galois.
\end{proof}

Another general feature controlling $G$-torsors is the following application of \Cref{pr:G-torsors_and_basechange}. Recall from the definition of a finite étale morphism of stacks from \Cref{df:definition_of_morphisms_of_stacks}.

\begin{cor}\label{cor:finiteetaleness}
    Let $\calG$ be a group object in $\Stk_{/\M}$ such that the structure map $\calG \to \M$ is finite étale. Then for any $\calG$-torsor $\varphi \colon \Y \to \X$, the map $\varphi$ is also finite étale. In particular, if $\calG$ corresponds to finite discrete group object in spaces $G$, then for any $G$-torsor $\varphi \colon \Y \to \X$, the map $\varphi$ is finite étale and faithfully flat.
\end{cor}

\begin{proof}
    This follows from part 3 of \Cref{pr:G-torsors_and_basechange} as finite étale morphisms satisfy descent and base change; see \Cref{pr:permanenceproperties_for_morphisms}. The ``in particular'' statement follows as $G \simeq \coprod_G \M$ in $\Stk_{/\M}$ and the fold map out of a finite coproduct is finite étale and faithfully flat. 
    %Alternative proof: Affineness is precisely \Cref{pr:torsors_are_relatively_affine}, so we are left to show that $\varphi$ is finite \'{e}tale. Consider an arbitrary $\Spec A \to \X$, and $\Spec B \to \Spec A$ for the pullback of $\varphi$ against $\Spec A$. This map between affine stacks is an effective epimorphism by base change as $\varphi$ is an effective epimorphism (\Cref{pr:higher_principality}), so \cite[Lm.A.13]{tokic_family_completion} states that there is an $\E_\infty$-$B$-algebra $C$ such that the composition $A \to B \to C$ is faithfully flat. To check that $A\to B$ is flat, it suffices to show that $C \to B\otimes_A C$ is flat, but this is the base change is identified with the diagonal map
    %\[C \to B\otimes_A C \simeq (B\otimes_A B)\otimes_B C \simeq (\prod_G B) \otimes C \simeq \prod_G C.\]
    %This shows that $\Spec B \to \Spec A$ is flat, and as it also defines an affine $G$-torsor by base change (\Cref{pr:G-torsors_and_basechange}), then \Cref{pr:converse} states that $A \to B$ is a $\pi_0$-$G$-Galois extension. By \cite[Pr.B.7.5.5]{sag}, the map $A\to B$ is finite \'{e}tale on $\pi_0$.
\end{proof}

The next result shows how the flat and affineness conditions in tandem place huge restrictions on $G$-torsors.

\begin{prop}\label{pr:flatness_and_affine_torsors}
    Let $\varphi \colon \Y \to \X$ be an affine $G$-torsor over $\M$. Then the following are equivalent:
    \begin{enumerate}
        \item The map $\Y \to \M$ is affine flat.
        \item The $G$-action on $\Y$ defines an affine flat $G$-torsor $\Y \to \X$ over $\M$.
        \item For each map of stacks $\Spec A \to \M$, the pulled back affine $G$-torsor
        \[\Spec C = \Y \times_\M \Spec A \to \X \times_\M \Spec A = \Spec B\]
        over $\Spec A$, induces a $\pi_0$-$G$-Galois extension $B \to C$ and $A \to B$ is flat.
    \end{enumerate}
\end{prop}

Note that as $\varphi \colon \Y \to \X$ is an affine $G$-torsor over $\M$, then by \Cref{pr:torsors_are_relatively_affine} shows $\varphi$ is affine, hence $\Y \to \M$ is affine by composition. For other permanence properties of morphisms of stacks, we refer the reader to \Cref{pr:permanenceproperties_for_morphisms}.

\begin{proof}
    As $\varphi$ is finite étale by \Cref{cor:finiteetaleness} and affine flat morphisms compose, clearly part 2 implies part 1. Assuming part 1, then given $\Spec A \to \M$ and using the notation as in part 3, we need to show that $A \to B$ is flat. Flat morphisms of $\E_\infty$-rings are local on source with respect to the flat topology, see \cite[Lm.B.1.4.2(2)]{sag}, so as $B \to C$ is faithfully flat courtesy of \Cref{cor:finiteetaleness}, it suffices to show that $A \to C$ is flat, which is precisely part 1. Part 2 implies part 3 by base change for torsors (part 1 of \Cref{pr:G-torsors_and_basechange}) and flat maps, together with the fact that affine flat $G$-torsors between affines are equivalently $\pi_0$-$G$-Galois extensions by \Cref{pr:converse}. Assuming part 3, then to show that $\X \to \M$ is flat, we have to show that $A \to B$ is flat, which is true by assumption.
\end{proof}

\subsubsection{Torsors and Rognes' Galois extensions}
The relationship between torsors and Rognes' general notion of Galois extension is subtle.

\begin{mydef}\label{def:galois_rognes}
    Let $\calC$ be a semiadditive symmetric monoidal category with all small limits and $G$ be a finite group. A \emph{$G$-Galois extension} in $\CAlg(\calC)$ is a morphism $A \to B$ and a $G$-action on $B$ through $\E_\infty$-$A$-algebra maps such that the two natural maps
    \[A \to B^{hG}, \qquad B\otimes_A B \to \prod_G B\]
    are equivalences in $\CAlg(\calC)$, where the second map is defined on the left factor of the domain as the diagonal map and on the right factor as $g\colon B\to B$ into the factor associated with $g\in G$. A $G$-Galois extension is \emph{faithful} if the functor $\Mod_A(\calC) \to \Mod_B(\calC)$ sending $M$ to $B\otimes_A M$ is conservative.
\end{mydef}

By (\ref{eq:pi0galoisisgalois}), we can immediately relate our two notions of Galois extension.

\begin{cor}
    A $\pi_0$-$G$-Galois extension of $\E_\infty$-rings is a faithful $G$-Galois extension.
\end{cor}

The converse is clearly false: $\KO \to \KU$ is a faithful $C_2$-Galois extension, see \cite[Pr.5.3.1]{johnrognessdualisinggorups} or \Cref{pr:chromatic_positive_answer} below, but it induces the identity on $\pi_0$. The relationship between $\pi_0$-$G$-Galois extensions and torsors of \Cref{pr:pi0galois_to_torsor} similarly fails to generalise to faithful $G$-Galois extensions. Indeed, the natural map of stacks
\[\Spec \KU / C_2 \to \Spec \KO\]
is not an equivalence; on underlying classical stacks this is the map $BC_2 \to \Spec \Z$. A more natural question replaces $\Spec A$ with $\Spec B / G$, as $\Ga(\Spec B/G) \simeq A$:

\begin{question}\label{q:Ggalois}
    If $A \to B$ is a faithful $G$-Galois extension of $\E_\infty$-rings, is $\Spec B \to \Spec B / G$ necessarily an affine $G$-torsor over some base $\M\neq \Spec B / G$?
\end{question}

If the base $\M$ is chosen to be affine, then $\Spec B/G$ itself must be affine, so if $A\to B$ is also flat, it must have been a $\pi_0$-$G$-Galois extension by \Cref{pr:converse}.

We can find certain base stacks $\M$ when $B$ has a chromatic flavour. Let us write $\M_{\FG}^\ori$ for the moduli stack of oriented formal groups of Lurie (\cite[Df.4.1.1.2]{reconstruction}) and $\M_{\FG\leq h}^\ori$ for the substack of oriented formal groups of finite height $\leq h$ (\cite[Df.4.3.1.2]{reconstruction}).

\begin{prop}\label{pr:chromatic_positive_answer}
    Let $B$ be a Landweber exact complex-periodic $\E_\infty$-ring with a $G$-action such that for the pullback along any closed point $x\colon \Spec \kappa \to \Spec \pi_0 B$, the stabiliser $G_x \leq G$ acts faithfully on the pullback of the classical Quillen formal group to $\kappa$. Then $\Spec B \to \Spec B/G$ is an affine flat $G$-torsor over $\M_{\FG}^\ori$. If $B$ is also of bounded chromatic height, then $A=B^{hG} \to B$ defines a faithful $G$-Galois extension of $\E_\infty$-rings and the associated $G$-torsor $\varphi\colon \Spec B \to \Spec B/G$ is affine over $\M_{\FG\leq h}^\ori$ for some finite height $h$.
\end{prop}

The last sentence above comes from \cite[\textsection6.2]{akhilandlennart}. The above also identifies 
\[\Spec B / G \simeq \M_A^\ori = \M_{\FG}^\ori \times \Spec A\]
as relatively affine stacks over $\M_{\FG\leq h}^\ori$, courtesy of \cite[Th.D]{reconstruction}.

\begin{proof}
    The faithfulness condition on the $G$-action is equivalent to asking that the map of classical stacks $\Spec \pi_0 B / G \to \M_{\FG}^\heartsuit$ is affine (\cite[\textsection6.2]{akhilandlennart}). The map $\al \colon \Spec B / G \to \M_{\FG}^\ori$ is by definition a flat map between geometric stacks, because on the cover $\Spec B \to \Spec B/G$ of the source, the composite $\Spec B \to \M_{\FG}^\ori$ is affine, courtesy of $\M_{\FG}^\ori$ having affine diagonal \cite[Cor.4.1.4.1]{reconstruction}, and flat, courtesy of $\Spec B$ being Landweber exact. Moreover, $\al$ is a relative nonconnective Deligne--Mumford stack by \cite[Cor.4.1.4.1]{reconstruction} again. We can now apply \cite[Cor.3.4.4.6]{reconstruction}, which states that as $\Spec B / G \to \M_{\FG}^\ori$ is affine on underlying classical stacks, the associated map of spectral stacks is affine. This gives the desired affine flat $G$-torsor over $\M_{\FG}^\ori$. If $B$ is of height $\leq h$ for all primes $p$, then the unique map affine $\Spec B / G \to \M_{\FG}^\ori$ factors as an affine map through $\M_{\FG\leq h}^\ori$. As $\M_{\FG\leq h}^\ori$ is 0-affine (\cite[Th.4.4.0.2]{reconstruction}), this implies that $\Spec B / G$ is also 0-affine as affine maps are 0-affine \cite[Pr.2.2.2.5(2)]{reconstruction} and 0-affine maps compose \cite[Pr.2.2.1.6]{reconstruction}. The fact that $A \to B$ is a faithful $G$-Galois extension now follows from \Cref{cor:0semiaffine_gives_galoisextension} below, as again, $\M_{\FG\leq h}^\ori$ is 0-affine.
\end{proof}

This result implies that faithful $C_2$-Galois extension $\KO \to \KU$ induces an affine $C_2$-torsor $\Spec \KU \to \Spec \KU / C_2$ over $\M_{\FG}^\ori$. More generally, it implies that for any finite subgroup $G$ of the Morava stabiliser group of height $n$, the quotient $\Spec E_n \to \Spec E_n / G$ defines an affine $G$-torsor over $\M_{\FG\leq n}^\ori$; a geometric model for \emph{height $n$ real $K$-theories}.

Playing with the Grothendieck topology on $\Stk$ leads to variants of \Cref{q:Ggalois} and suggests that the fpqc topology is not the best topology suited for studying $G$-Galois extensions in general.

\begin{remark}
    The previous proposition generalises to Landweber exact \emph{complex-orientable} $\E_\infty$-rings $B$ of finite height with this faithfulness condition, as long as we replace $\Stk$ with the appropriate category of stacks defined using $\pi_\ast$-fpqc topology; see \cite[Rmk.4.1.2.6]{reconstruction}.
\end{remark}

\begin{remark}\label{rmk:universal_descent_topology}
    A map of $\E_\infty$-rings $A \to B$ is called a \emph{universal descent morphism} (\cite[Df.D.3.1.1]{sag}) if the smallest stable subcategory of $\Mod_A$ generated by $A$-modules of the form $B\otimes_A M$ and closed under retracts is precisely $\Mod_A$. Let us generate a finitary topology on $\CAlg$ with basis the universal descent morphisms and write $\Stk_\ud$ for the associated category of stacks following \cite[Def.2.1.1.5]{reconstruction}. Then for a $G$-Galois extension of $\E_\infty$-rings, the induced map between affine objects $\Spec B \to \Spec A$ in $\Stk_\ud$ is a cover with respect to this topology, see \cite[Lm.9.1.2]{johnrognessdualisinggorups}, and $\Spec B / G \simeq \Spec A$ as in the proof of \Cref{pr:pi0galois_to_torsor}. In other words, the answer to the variant of \Cref{q:Ggalois} with $\Stk_\ud$ playing the role of $\Stk$ is ``yes, over the terminal base $\Spec \Sph$''.
\end{remark}

%%%%%%%%%%%%%%%%%%%%%%%%%%%%%%%%%%%%%%%%%%%%%%%%%%%%%%%%%%%%%%%
\subsubsection{Galois extensions associated with a torsor}\label{ssec:relationtoGalois}
We now describe how various torsors induce Galois extensions in several different scenarios. The first are of a relative algebraic flavour and the second more absolute and categorical.

\begin{prop}\label{pr:torsor_to_Galois_inQCoh}
    Let $\varphi\colon \Y \to \X$ be a $G$-torsor over $\X$. Then the induced map
    \[\O_\X \to \varphi_\ast \O_\Y\]
    is a faithful $G$-Galois extension in $\CAlg(\QCoh(\X))$.% In particular, for any affine morphism $f\colon \X \to \M$, the induced map
    %\[f_\ast \O_\X \to g_\ast \O_\Y\]
    %is a faithful $G$-Galois extension in $\CAlg(\QCoh(\M))$, where $g=f\varphi$.
\end{prop}

\begin{proof}
    First, we note that $\varphi$ is affine courtesy of \Cref{pr:torsors_are_relatively_affine} as $G \simeq \coprod_G \X$ is affine over $\X$. The equivalence of categories $\Stk^\aff_{/\X} \simeq \CAlg(\QCoh(\X))^\op$ of \cite[Th.2.2.3.1]{reconstruction} and unwinding the definitions, we see that $G$-torsors over $\X$ are sent to $G$-Galois extensions under this equivalence; products over $\X$ and tensor products over $\O_\X$ are identified via \cite[Pr.2.2.1.11(3)]{reconstruction}. For faithfulness, note the equivalences
    \[\Mod_{\O_\X}(\QCoh(\X)) \simeq \QCoh(\X), \qquad \Mod_{\varphi_\ast \O_\Y}(\QCoh(\X)) \simeq \QCoh(\Y),\]
    the first as $\O_\X$ is the unit of $\QCoh(\X)$ and the second is the definition of $\Y \to \X$ being $0$-affine, a consequence of affine morphisms being $0$-affine by \cite[Pr.2.2.2.5]{reconstruction}. Faithfulness of $\O_\X \to \varphi_\ast \O_\Y$ now follows by \Cref{pr:higher_principality} as $\varphi$ is always an effective epimorphism so $\varphi^\ast \colon \QCoh(\X) \to \QCoh(\Y)$ is conservative.
\end{proof}

For global sections to preserve Galois extensions, we need more assumptions on $\X$. Recall that a map $\varphi \colon \Y \to \X$ is \emph{0-affine} if $\QCoh(\Y) \simeq \Mod_{\varphi_\ast \O_\Y} (\QCoh(\X))$. Cancellation for 0-affine morphisms (\cite[Pr.2.2.1.6]{reconstruction}) implies that a 0-affine $G$-torsor over an 0-affine base is simply a $G$-torsor where the quotient $\X$ is 0-affine.

\begin{prop}\label{cor:0semiaffine_gives_galoisextension}
    The functor $\Ga \colon \Stk^\op \to \CAlg$ sends $0$-affine $G$-torsors over $\Spec \Sph$ to faithful $G$-Galois extensions of $\E_\infty$-rings.
\end{prop}

\begin{proof}
    Combine \Cref{pr:torsor_to_Galois_inQCoh} with the equivalence $\CAlg(\QCoh(\X)) \simeq \CAlg_{\Ga(\X)}$.
\end{proof}

Now onto our categorical examples.

\begin{prop}\label{pr:torsor_to_Galois_inPRL}
    The functor $\QCoh(-) \colon \Stk^\op \to \CAlg(\PrLst)$ sends $G$-torsors to faithful $G$-Galois extensions in $\CAlg(\PrLst)$.
\end{prop}

A generalised version of this statement for affine stacks appears below as \Cref{pr:mod_and_perf_preserve_Galois}.

\begin{proof}
    Let $\Y \to \X$ be a $G$-torsor. By \Cref{cor:qcoh_and_fixedpoints}, we have $\QCoh(\X) \simeq \QCoh(\Y)^{hG}$. Next, note that
    \[\begin{tikzcd}
        {\coprod_G \Y}\ar[r, "{\al}"]\ar[d, "\nabla"]   &   {\Y}\ar[d]  \\
        {\Y}\ar[r]                                      &   {\X}
    \end{tikzcd}\]
    is Cartesian by \Cref{pr:higher_principality}, and that $\Y \to \X$ is affine by \Cref{pr:torsors_are_relatively_affine}, leading to the equivalences
    \[\QCoh(\Y)\otimes_{\QCoh(\X)} \QCoh(\Y) \simeq \QCoh(\Y \times_\X \Y) \simeq  \QCoh(\coprod_G \Y) \simeq \prod_G \QCoh(\Y),\]
    the first from \cite[Pr.2.2.1.11]{reconstruction}, as the Cartesian square of stacks above is \emph{adjointable} (\cite[Def.2.2.1.8]{reconstruction}) by \cite[Pr.2.2.2.5]{reconstruction}, the second from \Cref{pr:higher_principality}, and the third as $\QCoh(-)$ sends colimits to limits. In particular, $\QCoh(-)$ sends $G$-torsors to $G$-Galois extensions. All $G$-Galois extensions in $\PrLst$ are faithful, as the comparison map $\QCoh(\Y)_{hG} \to \QCoh(\Y)^{hG}$ is an equivalence, see \cite[Pr.2.26]{categorical_ambidexterity_shay}, which is equivalent to faithfulness by the natural generalisation of \cite[Pr.6.3.3]{johnrognessdualisinggorups}.
\end{proof}

To obtain a ``small'' version of this statement, we need to assume that $\QCoh(\X)$ is compactly generated as a stable category. We borrow the following from \cite[Pr.3.9]{perfectstacks}.

\begin{mydef}
    A stack $\X$ is \emph{perfect} if $\QCoh(\X)$ is compactly generated and its compact and dualisable objects coincide. A morphism of stacks $\Y \to \X$ is \emph{perfect} if for each map from an affine $\Spec A \to \X$, the pullback $\Y \times_\X \Spec A$ is a perfect stack.
\end{mydef}

Affine stacks are clearly perfect, as are $0$-affine stacks by definition, and \cite[Pr.3.19]{perfectstacks} shows that nonconnective quasi-compact spectral schemes are perfect, expanding on ideas of Neeman \cite{neeman_grothendieck_duality} and Thomason--Trobaugh \cite{thomasontrobaugh}.

\begin{prop}\label{pr:torsor_to_Galois_2ring}
    The functor $\QCoh(-)^\omega \colon \Stk^\op \to \tworing$ sends perfect $G$-torsors $\Y \to \X$ over $\Spec \Sph$ to faithful $G$-Galois extensions in $\tworing$.
\end{prop}

A perfect $G$-torsor $\Y \to \X$ over $\Spec \Sph$ is simply one where $\X$ is a perfect stack.

To prove this statement, we use the following lemma.

\begin{lemma}\label{lm:compactobjects_pres_galois}
    Let $G$ be a finite group and let $f\colon A \to B$ be a faithful $G$-Galois extension in $\CAlg(\PrLst)$, where both $A$ and $B$ are rigidly compactly generated, so compact and dualisable objects agree, and $f$ preserves compact objects. Then $A^\omega \to B^\omega$ is a faithful $G$-Galois extension in $\tworing$.
\end{lemma}

\begin{proof}
    The data of the $G$-Galois extension $A \to B$ in $\calC=\CAlg(\PrLst)$ already lives in $\calC_\omega = \CAlg(\PrLstomega)$ by assumption. The first condition to be a $G$-Galois extension follows as the limit $B^{hG}$ in $\calC$ is also a limit in $\calC_{\omega}$. Indeed, this requires that $B^{hG} \simeq A$ lies in $\calC_{\omega}$, which is true by assumption, and that the projection $B^{hG} \to B$ preserves compact objects, which follows as $(-)^\dualisable$ preserves limits (\cite[Pr.4.6.1.11]{ha}) and in both $A$ and $B$, the collections of compact and dualisable objects agree:
    \[A^\omega = A^\dualisable \simeq (B^{hG})^\dualisable \simeq (B^\dualisable)^{hG} \to B^\dualisable = B^\omega.\]
    For the second condition to be a Galois extension in $\calC_\omega$, we observe that $\PrLstomega \to \PrLst$ preserves colimits, is symmetric monoidal, and both categories are semiadditive. It then immediately follows that
    \[B\otimes_A B \to \prod_G B\]
    is also an equivalence in $\calC_\omega$ as this holds in $\calC$ by assumption. For faithfulness, suppose that $M \in \Mod_A(\PrLstomega)$ satisfies $B\otimes_A M = 0$. As $\PrLstomega \to \PrLst$ is symmetric monoidal and a subcategory inclusion, $M$ is also an object of $\Mod_A(\PrLst)$ with $B\otimes_A M=0$ in $\PrLst$. From the faithfulness of the original extension in $\calC$, we see that $M=0$ in $\PrLst$, meaning it also vanishes in $\PrLstomega$, as desired. Finally, the symmetric monoidal equivalence $(-)^\omega\colon \PrLstomega\simeq \Cat^\perf$ finishes the proof.
\end{proof}

\begin{proof}[Proof of \Cref{pr:torsor_to_Galois_2ring}]
    As $\X$ is perfect and $\varphi\colon \Y \to \X$ is affine by \Cref{pr:torsors_are_relatively_affine}, then $\Y$ is easily seen to also be perfect. In particular, both $\QCoh(\Y)$ and $\QCoh(\X)$ are rigidly compactly generated and the pullback functor $\varphi^\ast \colon \QCoh(\X) \to \QCoh(\Y)$ preserves compact objects as it is symmetric monoidal. We can apply \Cref{lm:compactobjects_pres_galois} to the faithful $G$-Galois extension $\QCoh(\X) \to \QCoh(\Y)$ of \Cref{pr:torsor_to_Galois_inPRL}, and we are done.
\end{proof}

It seems reasonable now to include a variant of these results in the affine case; it will not play an explicit role elsewhere in this article.

\begin{prop}\label{pr:mod_and_perf_preserve_Galois}
    Let $G$ be a finite group and $A \to B$ be a faithful $G$-Galois extension of $\E_\infty$-rings. Then the induced maps
    \[\Mod_A \to \Mod_B, \qquad \Perf(A) \to \Perf(B)\]
    are both faithful $G$-Galois extensions in $\CAlg(\PrLst)$ and $\tworing$, respectively.
\end{prop}

An alternative proof would be to work in the category of universal descent stacks $\Stk_\ud$ of \Cref{rmk:universal_descent_topology}, proving a version of \Cref{pr:torsor_to_Galois_2ring} in this context.

\begin{proof}
    Suppose $A \to B$ is a faithful $G$-Galois extension. Galois descent \cite[Thm.9.4]{mathew_Galois} identifies
    \[\Mod_A \simeq \Mod_B^{hG},\]
    and the second condition follows from the equivalences
    \[\Mod_B \otimes_{\Mod_A} \Mod_B \simeq \Mod_{B\otimes_A B} \simeq \Mod_{\prod_G B} \simeq \prod_G \Mod_B.\]
    As all Tate constructions in $\PrLst$ die by ambidexterity, see \cite[Pr.2.26]{categorical_ambidexterity_shay}, by \cite[Pr.6.3.3]{johnrognessdualisinggorups}, we also see that $\Mod_A \to \Mod_B$ is faithful. This yields the faithful Galois extension in $\CAlg(\PrLst)$. The Galois extension in $\tworing$ follows from \Cref{lm:compactobjects_pres_galois}.
\end{proof}

%%%%%%%%%%%%%%%%%%%%%%%%%%%%%%%%%%%%%%%%%%%%%%%%%%%%%%%%%%%%%%%
\subsection{The moduli stack of subgroups}\label{ssec:construction_of_moduli}
We now define and study our main example of an affine $G$-torsor and the star of \Cref{main_geometric}. Recall that we have fixed an oriented $\P$-divisible group over a stack $\M$ and a finite abelian group $K$.

\subsubsection{Construction of moduli stack}
Recall the \emph{underlying Zariski space functor}
\[|-|\colon \Stk \to \Top, \qquad \X \mapsto |\X|,\]
the unique colimit-preserving extension of the assignment sending $\Spec A$ to its underlying Zariski space $|\Spec \pi_0 A|$; see \cite[Con.5.5.1]{temperedglobal}. Define $\Inj(\dual{K}, \G)$ as the substack of $\G[\dual{K}] = \G(BK)$ with the universal property that for any $\M$-stack $\sfZ$, the diagram of mapping spaces
\[\begin{tikzcd}
    {{\Stk_{/\M}}(\sfZ,\Inj(\dual{K}, \G))}\ar[r]\ar[d]    &   {{\Top_{/|\M|}}(|\sfZ|, |\G(BK)| \setminus \bigcup_{L\lneqq K} |\G(BL)|)}\ar[d]    \\
    {{\Stk_{/\M}}(\sfZ,\G(BK))}\ar[r]    &   {{\Top_{/|\M|}}(|\sfZ|, |\G(BK)|)}
\end{tikzcd}\]
is Cartesian; see \cite[Cor.5.5.7 \& Def.5.5.10]{temperedglobal}. In particular, the left vertical map is a monomorphism of spaces, hence $\Inj(\dual{K}, \G) \to \G(BK)$ is a monomorphism in $\Stk_{/\M}$.

\begin{prop}
    The natural map $\Inj(\dual{K}, \G) \to \M$ is affine and flat as is the underlying map of classical stacks $\Inj(\dual{K}, \G^\heartsuit) \to \M^\heartsuit$.
\end{prop}

\begin{proof}
    This first statement is contained within the proof of \cite[Cor.5.5.17]{temperedglobal}, which shows that after base change to a sufficiently nice cover of $\M$, the inclusion $\Inj(\dual{K}, \G) \to \G(BK)$ is an open immersion given by inverting an element. As both affine flat morphisms satisfy descent by \Cref{pr:permanenceproperties_for_morphisms}, this shows the spectral statement. For the statement on underlying classical stacks, we again appeal to descent in this setting, so it suffices to show that this holds locally on $\M^\heartsuit$, so after base change along $\Spec \pi_0 A \to \M^\heartsuit$ associated with any map $\Spec A \to \M$. We can identify
\[\Inj(\dual{K}, \G_{\Spec A})^\heartsuit = \Inj(\dual{K}, (\G_{\Spec A})^\heartsuit) \simeq \Inj(\dual{K}, \G^\heartsuit) \times_{\M^\heartsuit} \Spec \pi_0 A\]
via a slight generalisation of \cite[Pr.3.4.3.3]{reconstruction}. In \emph{ibid}, we work with \emph{flat morphisms between geometric stacks}, and such a statement formally upgrades to \emph{relative flat geometric stacks}, of which our affine and flat map $\Inj(\dual{K}, \G) \to \M$ is clearly a candidate. Details for this generalisation will appear shortly in \cite{geometricnorms}, so we refrain from repeating ourselves. Flatness is now also clear, as this can be checked locally on $\M^\heartsuit$ and flat maps between affine spectral stacks induce flat maps on underlying classical affine stacks by definition.
\end{proof}

The stack $\G[\dual{K}] = \G(BK)$ comes equipped with a natural $\Aut(BK)$-action as it comes from a functor $\G(-) \colon \Glo_\ab \to \Stk$. Our goal is now to show that the $\Aut(BK)$-action on $\G(BK)$ factors through $\Inj(\dual{K}, \G)$.

First, suppose that $\M = \Spec A$ is affine. Then for each commutative $\pi_0 A$-algebra $R$, the $R$-points of this stack of injections are by definition
\[%\Inj(\dual{K}, \G^\heartsuit)(R) = 
{\Stk^\heartsuit_{/\pi_0A}}(\Spec R, \Inj(\dual{K}, \G^\heartsuit)) \subseteq {\Stk^\heartsuit_{/\pi_0A}}(\Spec R, \G^\heartsuit(BK))) = \Ab(\dual{K},\G^\heartsuit(R)),\]
interpreting $\G^\heartsuit(BK)$ as $\Hom(\dual{K}, \G^\heartsuit)$, are precisely those homomorphisms $\varphi \colon \dual{K} \to \G^\heartsuit(R)$ which do not factor through $\dual{K} \to \dual{L}$ for any proper subgroups $L\leq K$, ie, the injections. This immediately yields the following two crucial classical observations.

\begin{lemma}\label{lm:discrete_action_factors}
    Over $\M = \Spec A$, the $\Aut(\dual{K})$-action on $\Hom(\dual{K}, \G^\heartsuit) = \G^\heartsuit(K)$ induces an $\Aut(\dual{K})$-action on $\Inj(\dual{K}, \G^\heartsuit)$. Moreover, this action on the set of $R$-valued points is free, where $R$ is a commutative $\pi_0 A$-algebra.
\end{lemma}

Turning this lemma into a spectral statement is mostly formal.

\begin{prop}\label{pr:action_factors_to_inj}
    The $\Aut(BK)$-action on $\G(BK)$ induces an $\Aut(BK)$-action on $\Inj(\dual{K},\G)$ such that the natural inclusion of stacks
    \[\Inj(\dual{K},\G) \to \G(BK)\]
    is naturally $\Aut(BK)$-, and hence by restriction also $\Aut(\dual{K})$-, equivariant.
\end{prop}

To prove this, we need a simple lemma that we cannot find in the literature.

\begin{lemma}\label{lm:monmorphisms_and_actions}
    Let $\calC$ be a category, $X,Y$ be objects of $\calC$, and $f\colon X\to Y$ be a monomorphism. If $\calG$ is a group object in spaces acting on $Y$, and for all $g\in \pi_0 \calG$, the map $gf\colon X \to Y$ lifts to a map $\bar{g} \colon X \to X$ along the monomorphism
    \[{\calC}(X,X) \to {\calC}(X,Y),\]
    then $X$ inherits a unique $\calG$-action such that $f$ is $\calG$-equivariant.
\end{lemma}

\begin{proof}
    The projection functor $\calC^{\Delta^1} \to \calC$ to the second factor defines a map of $\E_1$-spaces
    \[\calC^{\Delta^1}(f,f) \to \calC(Y,Y).\]
    We claim that this is a monomorphism, whose image consists of precisely those maps $Y \to Y$ in $\calC$ which lift (necessarily uniquely) to endomorphisms of $f$ in $\calC^{\Delta^1}$. Indeed, it suffices to see that the induced map of spaces is a monomorphism. This then follows from the pullback of mapping spaces
    \[\begin{tikzcd}
        {\calC^{\Delta^1}(f,f)}\ar[r]\ar[d] &   {\calC(Y,Y)}\ar[d]  \\
        {\calC(X,X)}\ar[r]                  &   {\calC(X,Y)}
    \end{tikzcd}\]
    from \cite[Pr.B.1]{arakawa_arrow_cat_appendix}, and the fact that the lower-horizontal map is a monomorphism. Our hypotheses on $\calG$ imply that the map of $\E_1$-spaces $\calG \to \calC(Y,Y)$ factors through the upper-horizontal map, giving us our desired action on $f$, and hence also on $X$ via projection onto the first factor.
\end{proof}

\begin{proof}[Proof of \Cref{pr:action_factors_to_inj}]
    First, we note that as $K$ is finite abelian, we can compute
    \[\pi_0 \Aut(BK) \simeq \Aut(K) \simeq \Aut(\dual{K}),\]
    the second equivalence using Pontryagin duality. From the universal property of $\Inj(\dual{K}, \G)$ and \Cref{lm:monmorphisms_and_actions}, it suffices to show that for each $\ga \in \Aut(\dual{K})$, that the composite
    \[\Inj(\dual{K}, \G) \to \G(BK) \xrightarrow{\ga} \G(BK)\]
    factors through $\Inj(\dual{K}, \G)$ upon taking underlying Zariski spaces. In other words, if $\ga$ is an automorphism of $K$, we need to show that the subspace $|\G(BK)| \setminus \bigcup_{L\lneqq K} |\G(BL)|$ is preserved by the action $\ga$. As
    \begin{equation}\label{eq:presentation_of_stack_category}
    \Stk_{/\M} \simeq \lim \Stk_{/\Spec A}
    \end{equation}
    for any presentation $\M = \colim \Spec A$ by \cite[Pr.2.1.2.9(4)]{reconstruction}, it suffices to prove this in the affine case of $\M=\Spec A$. By definition, the underlying Zariski subspace functor factors through the underlying classical stack functor $(-)^\heartsuit \colon \Stk \to \Stk^\heartsuit$, see \cite[Rmk.5.5.2]{temperedglobal}, so we are reduced to showing this statement after applying $(-)^\heartsuit$. The result now follows from \Cref{lm:discrete_action_factors}.
\end{proof}

\begin{mydef}\label{df:moduli_stack_of_subgroups}
    The \emph{stack of subgroups of $\G$ with type $\dual{K}$} is the quotient
    \[\Sub(\dual{K}, \G) = \Inj(\dual{K}, \G) / \Aut(\dual{K})\]
\end{mydef}

Let us see some explicit examples of these stacks of subgroups.

\begin{example}\label{ex:subforKO}
    Let $\M$ be a stack and $\G$ an oriented $\P$-divisible group such that on an affine flat cover $\M' \to \M$, the base change of $\G$ is isomorphic to $\mu_{\P^\infty}$, the torsion of the multiplicative group scheme $\G_m$ over $\M'$. For example, $\G$ could be the torsor of a \emph{torus}; see \Cref{pr:universalorientedtorus}. We claim that the structure map $\X = \Sub(\dual{K}, \G) \to \M$ in this case factors as an equivalence
    \[\Sub(\dual{K}, \G) \simeq \M \times \Spec \Sph[\tfrac{1}{n}].\]
    First, to see that $n$ is invertible on $\X$, it suffices to show that $n$ is invertible in $A$ for all maps $f \colon \Spec A \to \X$ from an affine stack such that the pullback of $\G$ to $A$ is isomorphic to $\G_m$ over $A$. Write the pullback of $\G$ to $\Spec A$ as $\G_A$. The map $f$ equips the $n$-torsion $\G_A[n] = \mu_n$ with a subgroup of the form $\dual{C_n}$. However, an injection $\dual{C_n} \to \mu_n$ over $A$ induces an injection of the same form over $\pi_0 A$, and it is well-known that classically such an injection only exists if $n$ is invertible in $A$, in which case this injection is an isomorphism. This shows that $n$ is invertible on both $\Y = \Inj(\dual{C_n}, \G)$ and $\X$, which implies that they are both étale $\M$. By Lurie's étale rigidity,\footnote{The geometric version of Lurie's étale rigidity we are referring to here is the statement that $(-)^\heartsuit\colon \Stk_{/\M} \to \Stk^\heartsuit_{/\M^\heartsuit}$ induces an equivalence between relatively affine and étale stacks on both sides. This easily follows from Lurie's affine étale rigidity \cite[Th.7.5.4.2]{ha} and the fact that the functor sending a stack $\X$ to $\Stk_{/\X}$ sends colimits to limits, which is precisely \cite[Pr.2.1.2.9(4)]{reconstruction}.} both $\Y$ and $\X$ are the unique étale and relatively affine spectral realisations of their underlying classical stacks over $\M^\heartsuit$. It therefore suffices to make the desired identification in the classical situation, where it is clear that $\Sub(\dual{C_n}, \G)^\heartsuit = \M^\heartsuit \times \Spec \Z[\tfrac{1}{n}]$, giving us the claimed equivalence above.
\end{example}

\begin{example}\label{ex:subforTMF}
    If $\sfE$ is an oriented elliptic curve over a stack $\M$, then its torsion $\sfE[\P^\infty]$ defines an oriented $\P$-divisible group over $\M$. We state the following in the universal case, where $\M=\M_\Ell^\ori$ is Lurie's moduli stack of oriented elliptic curves \cite[\textsection7]{ec2} and $\calE$ is the universal oriented elliptic curve. In this case, we claim that there are identifications
    \[\Inj(\dual{C_n\times C_n}, \calE[\P^\infty]) \simeq \M^\ori(n), \qquad \Sub(\dual{C_n\times C_n}, \calE[\P^\infty]) \simeq \M_\Ell^\ori \times \Spec \Sph[\tfrac{1}{n}],\]
    where $\M^\ori(n)$ is the unique \'{e}tale lift of the classical moduli stack of elliptic curves with $\Ga(n)$-level structure over $\M_\Ell^\ori$. Indeed, the same arguments of the previous example \Cref{ex:subforKO} work here, using the fact that if the $n$-torsion of an elliptic curve is given by $C_n \times C_n$, then $n$ is invertible in the base, courtesy of \cite[Cor.2.3.2]{km}. More generally, if $\sfE$ is an oriented elliptic curve over a general stack $\Y$, then by base change we have the natural identification
    \[\Sub(\dual{C_n \times C_n}, \sfE[\P^\infty]) \simeq \Y \times \Spec \Sph[\tfrac{1}{n}].\]
\end{example}

Generally, the stacks $\Sub(\dual{K}, \calE[\P^\infty])\times \Spec \Sph[\tfrac{1}{|K|}]$ are the unique complex-periodic refinements of those denoted as $\mathcal{M}_K$ from \cite[Def.2.26]{heckeontmf}. Integrally, they can be much more complicated and interesting; we will come back to more examples in \cite{geometricnorms} and future work.

\subsubsection{As an affine torsor}
\begin{theorem}[{\Cref{main_geometric}}]\label{thm:geometric_intext}
    The canonical map $\varphi \colon \Inj(\dual{K}, \G) \to \Sub(\dual{K}, \G)$ defines an affine flat $\Aut(\dual{K})$-torsor over $\M$.
\end{theorem}

By \Cref{cor:finiteetaleness}, the map $\varphi$ above is finite étale.

\begin{proof}
    To see that $\Sub(\dual{K}, \G)$ is affine over $\M$, we base change along a map of stacks $\Spec A \to \M$, reducing us to $\M=\Spec A$ and showing that $\Sub(\dual{K}, \G)$ is affine. We know that $\Inj(\dual{K},\G) = \Spec B$ is affine over $\Spec A$ and by \Cref{lm:discrete_action_factors}, the induced $\Aut(\dual{K})$-action on $\pi_0 B$ is free. In particular, $B^{h\Aut(\dual{K})} \to B$ is a $\pi_0$-$\Aut(\dual{K})$-Galois extension by definition, hence
    \[\Inj(\dual{K}, \G) = \Spec B \to \Spec B^{h\Aut(\dual{K})} \simeq \Spec B / \Aut(\dual{K}) \simeq \Sub(\dual{K}, \G)\]
    is an affine $\Aut(\dual{K})$-torsor by \Cref{pr:pi0galois_to_torsor}. In particular, $\Sub(\dual{K}, \G)$ is affine. Flatness then follows from \Cref{pr:flatness_and_affine_torsors} as $\Inj(\dual{K}, \G)$ is flat over $\M$.
\end{proof}

Combining \Cref{thm:geometric_intext} together with \Cref{pr:torsor_to_Galois_inQCoh,cor:0semiaffine_gives_galoisextension,pr:torsor_to_Galois_inPRL,pr:torsor_to_Galois_2ring} yields:

\begin{cor}\label{cor:galois_from_torsor}
    Writing $f\colon \X = \Sub(\dual{K},\G) \to \M$ and $g\colon \Y = \Inj(\dual{K}, \G) \to \M$ for the structure maps, the maps induced by the $\Aut(\dual{K})$-torsor of \Cref{thm:geometric_intext}
    \[\O_\X \to \varphi_\ast \O_\Y, \qquad \QCoh(\Sub(\dual{K},\G)) \to \QCoh(\Inj(\dual{K}, \G))\]
    are faithful $\Aut(\dual{K})$-Galois extension in $\CAlg(\QCoh(\X))$ and $\CAlg(\PrLst)$, respectively. If $\M$ is perfect, then
    \[\QCoh(\Sub(\dual{K},\G))^\omega \to \QCoh(\Inj(\dual{K}, \G))^\omega\]
    is a faithful $\Aut(\dual{K})$-Galois extension in $\tworing$, and if moreover $\M$ is $0$-affine, then
    \[\Ga(\Sub(\dual{K},\G)) \to \Ga(\Inj(\dual{K}, \G)), \qquad\Perf(\Ga(\Sub(\dual{K},\G))) \to \Perf(\Ga(\Inj(\dual{K}, \G)))\]
    are faithful $\Aut(\dual{K})$-Galois extensions in $\CAlg$ and $\tworing$, respectively.
\end{cor}

We close this section with an unresolved issue:

\begin{question}
    Does the quotient $\Inj(\dual{K}, \G) \to \Inj(\dual{K}, \G) / \Aut(BK)$ define an {affine} $\Aut(BK)$-torsor over $\M$?
\end{question}

This seems more natural than the restricted $\Aut(\dual{K})$-action, but we do not know how to show that the quotient stack $\Inj(\dual{K}, \G) / \Aut(BK)$ is affine in general. Skipping ahead, \Cref{pr:no_action_conflicts_part2_GM,pr:no_action_conflicts_part2_ellipticcurves} seem to suggest that this is not the case, as if the $\Aut(BK)$-action factors through the truncation to $\Aut(\dual{K})$, it seems unlikely that the $\Aut(BK)$-action also defines an affine torsor on the quotient. We have not yet found any explicit counterexamples, as in \Cref{pr:no_action_conflicts_part2_GM,pr:no_action_conflicts_part2_ellipticcurves}, we are forced to work with $n$ inverted, where $BC_n$-Galois extensions are uninteresting.

%%%%%%%%%%%%%%%%%%%%%%%%%%%%%%%%%%%%%%%%%%%%%%%%%%%%%%%%%%%%%%%
%%%%%%%%%%%%%%%%%%%%%%%%%%%%%%%%%%%%%%%%%%%%%%%%%%%%%%%%%%%%%%%
%%%%%%%%%%%%%%%%%%%%%%%%%%%%%%%%%%%%%%%%%%%%%%%%%%%%%%%%%%%%%%%
\section{Galois extensions from tempered cohomology theories}\label{sec:ESHT_over_a_stack}

%%%%%%%%%%%%%%%%%%%%%%%%%%%%%%%%%%%%%%%%%%%%%%%%%%%%%%%%%%%%%%%
\subsection{The canonical residual actions on geometric fixed points}\label{ssec:from_tempered_to_global_spectra}
Let $\Glo_\ab$ be the subcategory of spaces $\Spc$ spanned by groupoids of finite abelian groups $BH$. A \emph{global space} is an object of the presheaf category $\Spc_\ab^\gl = \Fun(\Glo_\ab, \Spc)$. Let $\Span$ be the span category on $\Spc_\ab^\gl$ with arbitrary backwards maps, and forwards maps those $f\colon X \to Y$ such that for all maps $T \to Y$ with $T\in \Glo_\ab$, the base change map $f_T$ is a faithful map of $1$-truncated groupoids; see \cite[\textsection2 \& Ex.2.1.12]{temperedglobal} for details. The category of functors $\Span \to \Sp$ whose restriction along $\Spc_\ab^\gl$, written as $\Sp^\gl_\ab$, preserves limits, is known to model Schwede's category of \emph{global spectra} with respect to the family of finite abelian groups of \cite{s}.

Let $X$ be a global spectrum and $K$ an abelian\footnote{The definition of $K$-fixed points and the canonical $\Aut(BK)$-action do not require $K$ to be abelian.} finite group. The (genuine) $K$-fixed point spectrum of $X$, written as $X^K$, are defined as the value of $X$ on $\ast / K = BK$.

\begin{mydef}\label{df:canonicalautBKaction}
Viewing a global spectrum $X$ as a functor $\Span \to \Sp$, the \emph{canonical $\Aut(BK)$-action} on $X^K$ is simply that induced by functoriality. If $X$ itself were an $\E_\infty$-object of $\Sp_\pi^\gl$, then this $\Aut(BK)$-action on $X^K$ is one of $\E_\infty$-rings. The induced $\Aut(BK)$-action on the $K$-geometric fixed points $\Phi^K X$ is also called the canonical action.
\end{mydef}

\begin{remark}
    This action can also be seen without the spectral Mackey functor model for global spectra. Indeed, as the $K$-fixed points functor $(-)^K \colon \Sp^\gl_\ab \to \Sp$ is (spectrally) corepresented by $\Sigma_+^\infty B_\gl K$, the suspension spectrum of the global classifying space for $K$, it is clear that the $K$-fixed points $X^{K} \simeq \map_{\Sp^\gl_\ab}(\Sigma_+^\infty B_\gl K, X)$ in $\Sp$ carry a natural action of $\Aut(B_\gl K) \simeq \Aut(BK)$.
\end{remark}

\subsubsection{Residual actions on tempered cohomology theories}
Our main source of global spectra in this article comes from \emph{oriented $\P$-divisible groups}, where these actions are quite explicit. Recall that our fixed $\P$-divisible group is a functor
\begin{equation}\label{eq:pdivisiblegroup}\G[-] \colon \Ab_\fin^\op \to \Stk\end{equation}
and its preorientation is a factorisation of $\G[-]$ over the Pontryagin dual and classifying space functor $B \dual{(-)} \colon \Ab_\fin^\op \to \Glo_\ab$, which we denote by
\begin{equation}\label{eq:preorientation}\G(-) \colon \Glo_\ab \to \Stk.\end{equation}
The tempered cohomology theory associated with $\G$ can be defined as the left Kan extension of $\G$ along the Yoneda embedding $\Glo_\ab \to \Spc_\ab^\gl$, followed by global sections, written as the limit-preserving functor
\begin{equation}\label{eq:tempered_cohomology_theory}
(\Spc_\ab^\gl)^\op \to \Stk^\op \xrightarrow{\Ga} \CAlg.
\end{equation}

According to \cite[Th.1.1.1]{temperedglobal}, or alternatively \cite[Th.E]{gepner2024global2ringsgenuinerefinements}, associated with our fixed oriented $\P$-divisible group $\G$, there is a natural $\E_\infty$-object $\Ga(\G)$ of $\Sp^\gl_\ab$ whose restriction to $\Spc_\ab^\gl$ along the backwards maps is precisely (\ref{eq:tempered_cohomology_theory}).\footnote{The more refined statement found in \cite{temperedglobal} is that $\Ga(\G)$ refines to a \emph{$\pi$-ambidextrous} global spectrum. This extra structure is not necessary yet; we will come back to this in \Cref{df:ofpiambi}.} By construction, the $K$-fixed points of $\Ga(\G)$ are given as
\[\Ga(\G)^K = \Ga (\G[\dual{K}]) = \Ga(\G(BK)). \]
In particular, the natural $\Aut(BK)$-action on $\Ga(\G)^K$ is already seen by the preoriented $\P$-divisible group $\G(-)$ of (\ref{eq:preorientation}) after taking global sections. Notice that this $\Aut(BK)$-action is one of stacks over the base stack $\G(\ast) = \M$.

By \Cref{pr:action_factors_to_inj}, the stack of injections $\Inj(\dual{K}, \G)$ inherits the canonical $\Aut(BK)$-action from $\G(BK)$ such that the natural inclusion map
\[\Inj(\dual{K}, \G) \to \G(BK)\]
is $\Aut(BK)$-equivariant. Moreover, there is a \emph{$K$-geometric fixed point stack} $\Phi^K \G$, an affine stack over $\G(BK)$, equipped with a unique equivalence
\begin{equation}\label{eq:unique_equivalence}\Inj(\dual{K}, \G) \simeq \Phi^K \G\end{equation}
of substacks of $\G(BK)$; see \cite[Th.F]{temperedglobal}. In particular, $\Phi^K \G$ inherits a unique $\Aut(BK)$-action such that the natural map $\Phi^K \G \to \G(BK)$ is $\Aut(BK)$-equivariant.

\begin{remark}
The global sections of the stack $\Phi^K \G$ are not guaranteed to recover the $K$-geometric fixed points of $\Ga(\G)$, as global sections (a kind of limit) need not commute with $K$-geometric fixed points (a kind of colimit). If our base stack $\M$ is \emph{0-semiaffine}, meaning the global sections functor $\Ga \colon \QCoh(\M) \to \Sp$ preserves colimits, then we also have a natural equivalence of $\E_\infty$-rings
\[\Ga(\Phi^K \G) \simeq \Phi^K \Ga(\G);\]
see \cite[Rmk.5.3.8]{temperedglobal}.
\end{remark}

The discussion above shows that this equivalence (\ref{eq:unique_equivalence}) is $\Aut(B{K})$-equivariant, the left side coming from the natural action of $\G(-)$ and the right side as a localisation of $\Ga(\G)^K$.

The canonical $\Aut(BK)$-action considered above is natural to consider, and will also be useful in \Cref{ssec:action_comparison}. The more prominent action in this article, however, is the restricted $\Aut(\dual{K})$-action.

\begin{mydef}
    The \emph{canonical $\Aut(\dual{K})$-action} on the $K$- (geometric) fixed points of a global spectrum $X$ by restricting the canonical $\Aut(BK)$-action along the map of group objects in spaces
    \[\Aut(\dual{K}) \to \Aut(BK)\]
    induced by applying Pontryagin duality and the classifying space functor.
\end{mydef}

By definition, the canonical $\Aut(\dual{K})$-action on $\Ga(\G)^K = \Ga \G[\dual{K}]$ is induced by its underlying $\P$-divisible group $\G[-]$ (\ref{eq:pdivisiblegroup}). The following summarises these equivariant identifications.

\begin{cor}\label{cor:identification_with_inj}
    For a finite abelian group $K$, the unique equivalence $\Phi^K \G \simeq \Inj(\dual{K},\G)$ of (\ref{eq:unique_equivalence}) between stacks over $\G(K)$ is naturally $\Aut(B{K})$-equivariant. In particular, it induces a unique equivalence of stacks over $\M$
    \[\Phi^K \G / \Aut(\dual{K}) \simeq \Sub(\dual{K}, \G).\]
    If $\M$ is 0-semiaffine, then the induced equivalence of $\E_\infty$-$\Ga(\G)^K$-algebras
    \[\Ga\Inj(\dual{K}, \G) \simeq\Phi^K \Ga(\G)\]
    is $\Aut(B{K})$-, hence also $\Aut(\dual{K})$-, equivariant.
\end{cor}

In particular, in the 0-semiaffine case, we can also identify
\begin{equation}\label{eq:geometric_model_for_fixedpoints}
  (\Phi^K \Ga(\G))^{h\Aut(\dual{K})} \simeq (\Ga \Inj(\dual{K}, \G))^{h\Aut(\dual{K})} \simeq \Ga(\Sub(\dual{K}, \G)),
\end{equation}
as $\Ga \colon \Stk^\op \to \CAlg$ sends colimits to limits.

\begin{proof}
    The inclusions of the substacks $\Phi^K \G \subseteq \G(BK) \supseteq \Inj(\dual{K}, \G)$ both commute with the natural $\Aut(BK)$-actions by the discussion above, hence the unique equivalence between $\Phi^K \G$ and $\Inj(\dual{K}, \G)$ as substacks of $\G(BK)$ is also $\Aut(BK)$-equivariant.
\end{proof}

\subsubsection{Proofs of main algebraic theorems}
Now we can prove \Cref{main_algebraic,cor_main:globalisation}. Let $H$ be a finite group equipped with an inclusion $K\leq H$ from our fixed abelian group $K$. All of the above $\Aut(\dual{K})$-actions induce actions of the \emph{global Weyl group} $W_H^\gl K = N_H K / C_H K$ through the homomorphism (\ref{eq:weyl_homomorphism})
\[W_H^\gl K \to \Aut(K) = \Aut(\dual{K}); \qquad hC_H K \mapsto (k\mapsto hkh^{-1});\]
the equality above is induced by Pontryagin duality.

\begin{mydef}\label{df:global_weyl_action}
    Precomposing the canonical $\Aut(\dual{K})$-action on the $K$- (geometric) fixed points of a global spectrum $X$ with the homomorphism (\ref{eq:weyl_homomorphism}) yields its \emph{global Weyl group action}. The pair $K\leq H$ is \emph{Weyl faithful} if the homomorphism (\ref{eq:weyl_homomorphism}) is injective.
\end{mydef}

\begin{theorem}[{\Cref{main_algebraic}}]\label{thm:algebraic_intext}
    Let $H$ be a finite group, $K$ be a Weyl faithful subgroup of $H$, and $\G$ be an oriented $\P$-divisible group over a $0$-affine stack $\M$ with associated $H$-ring spectrum $\Ga(\G)$. Then the natural maps
    \[\Phi^K \Ga(\G)^{hW_H^\gl K} \to \Phi^K \Ga(\G), \qquad \Perf(\Phi^K \Ga(\G))^{hW_H^\gl K} \to \Perf(\Phi^K \Ga(\G))\]
    are faithful $W_H^\gl K$-Galois extensions in $\CAlg$ and $\tworing$, respectively.
\end{theorem}

\begin{remark}\label{rmk:global_weyl_better_than_classical}
    The global Weyl group is the more natural candidate than the usual Weyl group for \Cref{main_algebraic}. Indeed, for $C_2 \leq C_4$, one has $W_{C_4} C_2 = C_2$ while $W_{C_4}^\gl C_2$ vanishes, and the only $\KU$-linear action on $\Phi^{C_2} \gKU = \KU[\tfrac{1}{2}]$ is trivial, hence the $W_{C_4} C_2$-action cannot be Galois. For an example of a pair $K\leq H$ which is also Weyl faithful, consider $H=\mathrm{Dic}_3$ to be the third dicyclic group of order $12$ and $K=C_3$ the unique subgroup of order $3$.
\end{remark}

\begin{remark}\label{rmk:geometric_fixed_points_better_than_genuine}
    The restriction to geometric fixed points rather than genuine fixed points is clear through examples; the $C_2$-action on $\gKU^{C_3}$ is not Galois. %For example, $\pi_0^{C_3} \gKU \simeq \Z[X]/X^3-1$ and corepresents elements of order $3$ with $C_2$-action sending an element of order $3$ to its inverse. This action is clearly not free on elements such as $1$.
    A moduli interpretation of genuine fixed points as the global sections of a homomorphism stack also makes it clear that this residual action is rarely free, and the null homomorphism will have a nontrivial stabiliser.
\end{remark}

\begin{proof}[Proof of \Cref{thm:algebraic_intext}]
    Combining \Cref{cor:galois_from_torsor,cor:identification_with_inj}, we see that
    \[\Phi^K \Ga(\G)^{h\Aut(\dual{K})} \simeq \Ga(f_\ast \O_\X) \to \Ga(g_\ast \O_\Y) \simeq \Phi^K \Ga(\G)\]
    is a faithful $\Aut(\dual{K})$-Galois extension of $\E_\infty$-rings. We then use the fact that the $W_H^\gl K$-action on $\Phi^K \Ga(\G)$ in question is defined by restricting this faithful Galois extension along an injective homomorphism. By \cite[Thm.1.2(a)]{johnrognessdualisinggorups}, which states that restriction of a faithful Galois extension along the inclusion of a subgroup produces a faithful Galois extension, we are done. The same argument holds in $\tworing$ as well.
\end{proof}

Onto \Cref{cor_main:globalisation}. To that end, let us recall a definition.

\begin{mydef}[{\cite[Def.1.2.2]{temperedglobal}}]\label{df:ofpiambi}
    Let $\Span_\pi$ be the span category on $\Spc_\ab^\ab$ with all backwards maps and whose forwards maps are those morphisms of global spaces $X\to Y$ such that for each $T\to Y$ with $T\in \Glo_\ab$, the pullback $X\times_Y T$ is a $\pi$-finite space; a space with finitely many path components, each with finitely many nonzero homotopy groups, all of which are finite. The category of \emph{$\pi$-ambidextrous global spectra} $\Sp_\pi^\gl$ is the category of functors $\Span_\pi \to \Sp$ whose restriction to $(\Spc_\ab^\gl)^\op$ preserves limits. An \emph{$\pi$-ambidextrous global $\E_\infty$-ring} is an $\E_\infty$-object of $\Sp_\pi^\gl$ with the Day convolution symmetric monoidal structure of \cite[Cor.2.2.3]{temperedglobal}.
\end{mydef}

The point of this enhancement of $\Sp^\gl_\ab$ is that objects of $\Sp^\gl_\pi$ have transfer maps against all relatively $\pi$-finite maps of global spaces; a lot more structure than $\Sp^\gl_\ab$. This is also the natural maximal such structure on tempered cohomology theories as suggested by tempered ambidexterity \cite[Thm.7.2.10]{ec3} and constructed in \cite[Thm.1.1.1]{temperedglobal}.

\begin{prop}\label{pr:general_torsors_to_global_Galois}
    The sheaf
    \[\Ga(\G_{(-)}) \colon (\Stk^{0-\aff}_{/\M})^\op \to \CAlg(\Sp^\gl_\pi), \qquad (\sfZ \to \M)\mapsto \Ga(\G_\sfZ)\]
    of \cite[Thm.1.1.1]{temperedglobal}, right Kan extended to all stacks and then restricted to stacks $\X$ over $\M$ which are $0$-affine, sends $H$-torsors $\Y \to \X$ to faithful $H$-Galois extensions for any finite group $H$.
\end{prop}

\begin{proof}
    Since this functor is a sheaf, we immediately have the identification
    \[\Ga(\G_\X) \simeq \Ga(\G_{\Y/H}) \simeq \Ga(\G_\Y)^{hH} \in \CAlg(\Sp^\gl_\pi).\]
    For the second Galois condition, that the map
    \[\Ga(\G_\Y)\otimes_{\Ga(\G_\X)} \Ga(\G_\Y) \to \prod_H \Ga(\G_\Y) \in \CAlg(\Sp^\gl_\pi)\]
    is an equivalence, we use the fact that the collection of geometric fixed point functors $\Phi^L \colon \CAlg(\Sp^\gl_\pi) \to \CAlg(\Sp)$ are jointly conservative, as $L$ ranges over all finite abelian groups; this follows as the forgetful functor $\Sp^\gl_\pi \to \Sp^\gl_\ab$ is conservative and that this fact holds for $\Sp_\ab^\gl$. Using the symmetric monoidality and exactness of geometric fixed points, it suffices to check that the maps
    \begin{equation}\label{eq:geometric_FP_checkingthing}\Phi^L \Ga(\G_\Y)\otimes_{\Phi^L \Ga(\G_\X)} \Phi^L \Ga(\G_\Y) \simeq \Phi^L\left(\Ga(\G_\Y)\otimes_{\Ga(\G_\X)} \Ga(\G_\Y)\right) \to \prod_H \Phi^L \Ga(\G_\Y)\end{equation}
    are equivalences for all finite abelian groups $L$. Using the base change property of geometric fixed points of \cite[Thm.E]{temperedglobal} as well as (\ref{eq:unique_equivalence}) as all stacks in sight are $0$-affine, we make the identification
    \[\Phi^L \Ga(\G_\sfZ) \simeq \Ga(\G_{\Inj(\dual{L}, \G_\sfZ)})\]
    for $\sfZ=\X,\Y$. The map of stacks
    \[\Inj(\dual{L}, \G_\Y) \to \Inj(\dual{L}, \G_\X)\]
    is an $H$-torsor as the pullback of the $H$-torsor $\Y \to \X$ over $\M$ along the map $\Inj(\dual{L}, \G) \to \M$, hence (\ref{eq:geometric_FP_checkingthing}) is an equivalence via \Cref{cor:0semiaffine_gives_galoisextension}. This shows that $\Ga(\G_\X) \to \Ga(\G_\Y)$ is an $H$-Galois extension in $\CAlg(\Sp^\gl_\pi)$, so only faithfulness is left to check. Suppose a module $M$ over $\Ga(\G_\X)$ has $M \otimes_{\Ga(\G_\Y)}\Ga(\G_\X) = 0$. To check if $M$ vanishes, we can again take geometric fixed points $\Phi^L M$ for all finite abelian $L$. Using the same arguments as above for $\Phi^L$, combined with the fact that \Cref{cor:0semiaffine_gives_galoisextension} gives us \emph{faithful} $H$-Galois extensions upon applying $\Phi^L$, we see that $\Phi^L M=0$ for each finite abelian $L$, so $M=0$, and our extension is faithful.
\end{proof}

\begin{cor}[{\Cref{cor_main:globalisation}}]\label{cor:globalisations}
    In the situation of \Cref{main_algebraic}, the map of $\pi$-ambidextrous $\E_\infty$-rings
    \[\Ga({\G_\sub}) \to \Ga({\G_\inj})\]
    is a faithful $W_H^\gl K$-Galois extension whose map of underlying $\E_\infty$-rings recovers \Cref{main_algebraic}.
\end{cor}

\begin{proof}
    By \Cref{main_geometric} and \Cref{pr:general_torsors_to_global_Galois}, the map of $\pi$-ambidextrous $\E_\infty$-rings in question is a faithful $\Aut(\dual{K})$-Galois extension, and restricting along an injection induces a faithful $W_H^\gl K$-Galois extension.
\end{proof}

%%%%%%%%%%%%%%%%%%%%%%%%%%%%%%%%%%%%%%%%%%%%%%%%%%%%%%%%%%%%%%%
\subsubsection{Comparing Weyl group actions}\label{ssec:action_comparison}
Recall that the \emph{Weyl group} of an inclusion of groups $K\leq H$ is the quotient $W_H K = N_H K / K$ and that $W_H K$ naturally acts on $BK$. One can see this either by viewing $W_H K$ as the automorphism group of $H$-equivariant automorphisms of $H/K$ and using that the functor sending a $G$-space to its orbifold quotient
\[\Spc_G \to \Spc^\gl_\ab; \qquad X \mapsto X//G\]
sends $H/K$ to $BK$, or equivalently, just by writing $BK = EH \times_H H/K$.

\begin{mydef}\label{df:canonical_weylaction}
    The \emph{canonical $W_H K$-action} on the $K$- (geometric) fixed points of a global spectrum $X$ is defined by composing the above $W_H K$-action on $BK$ with the canonical $\Aut(BK)$-action of \Cref{df:canonicalautBKaction}.
\end{mydef}

This canonical $W_H K$-action is the residual Weyl group action that an equivariant homotopy theorist is used to.

This combines with the natural quotient $W_H K \to W_H^\gl K$ and the discussion of \Cref{ssec:from_tempered_to_global_spectra} to give the following solid commutative diagram of group objects in spaces
\[\begin{tikzcd}
    {W_H K}\ar[d]\ar[r] &   {\Aut(BK)}\ar[d, "\tau", swap, shift right = 2]\ar[r]  &   {\Aut(X^K)}    \\
    {W_H^\gl K}\ar[r] &   {\Aut(K),}\ar[ru, dashed]\ar[u, shift right = 2, "B", dashed, swap] &   
\end{tikzcd}\]
where $\tau$ is the truncation map, the lower-horizontal map is (\ref{eq:weyl_homomorphism}).

\begin{remark}
    In this small subsection, we mostly ignore the isomorphism $\Aut(K) \simeq \Aut(\dual{K})$ induced by Pontryagin duality, and instead treat these groups as the same.
\end{remark}

Now for the standard warning, as much for ourselves as for the reader:

\begin{center}
    \emph{The canonical $\Aut(BK)$-action on $X^K$ need \textbf{not} factor through $\Aut(K)$. In particular, the $W_H K$-action need \textbf{not} factor through $W_H^\gl K$.}
\end{center}

In other words, the solid diagram above commutes, but the upper-left horizontal arrow does not necessarily factor through $B$ and hence also $W_H^\gl K$.

As an explicit example, we have the global sphere spectrum $\Sph$ with $\Sph^{C_2} = \Sph \oplus \Sigma^\infty_+ BC_2$ by the tom Dieck splitting. The $\Aut(BC_2)$-action on $\Sph^{C_2}$ acts nontrivially and tautologically on the $\Sigma^\infty_+ BC_2$-summand, however, $\Aut(C_2)$ vanishes.

Nevertheless, there are two situations of interest to us where the $\Aut(BK)$-action does factor through $\Aut(K)$, and where the $W_H K$-action factors through $W_H^\gl K$. The first depends on the group $H$, and the second on the global spectrum $X$.\footnote{The arguments of Balerrama from \cite[\textsection A.3]{balderrama_totalpoweroperations} can also be co-opted to show that these factorisations occur if the order $H$ acts as an isomorphism on the global spectrum $X$. We will not need such a statement here.} In the situations covered below, we can interpret the Galois extensions of \Cref{main_algebraic} as related to the residual Weyl group actions from equivariant homotopy theory.

\begin{prop}\label{pr:no_action_conflicts_part1}
    Let $H = N\rtimes Q$ be split, so the data of a normal subgroup $N \leq H$ and a homomorphism $N \to \Aut(Q)$ with $Q=H/N = W_H K$ the quotient. Then the canonical $Q$-action (\Cref{df:canonical_weylaction}) on the $N$- (geometric) fixed points of a global spectrum $X$ factors as a map of group objects in spaces
    \[W_H K \to \Aut(N) \xrightarrow{B} \Aut(BN).\]
\end{prop}

Thank you to William Balderrama for suggesting such a statement.

\begin{proof}
    To be concrete, let us work with the $N$-fixed points $X^N$; the argument for geometric fixed points is just a change of notation. Consider the commutative diagram of group objects in spaces
    \[\begin{tikzcd}
        {H}\ar[r]\ar[d] &   {\Aut(N)}\ar[d, "B"]    &   \\
        {Q}\ar[r]       &   {\Aut(BN)}\ar[r]    &   {\Aut(X^N)}
    \end{tikzcd}\]
    The lower-horizontal composite defines the canonical $Q = W_H N$-action on $X^N$. As the square above is Cartesian, a factorisation of the lower-right horizontal map through $B$ is precisely a splitting of the quotient map $H \to Q$, which comes from the expression of $H$ as the semi-direct product $N\rtimes Q$. This gives the desired factorisation.
\end{proof}

The next two examples depend on the global spectrum $X$ being a tempered cohomology theory associated with a geometric object such as a torus, see \Cref{ex:subforKO}, or an elliptic curve.

\begin{prop}\label{pr:no_action_conflicts_part2_GM}
    Let $A=\Ga(\G)$ where $\G$ is fpqc locally on $\M$ isomorphic to $\mu_{\P^\infty}$. For example, the torsion of a torus; see \Cref{pr:universalorientedtorus}. Then the $\Aut(BK)$-action on $\Phi^K A$ factors through $\Aut(K)$ and the $W_H K$-action factors through $W_H^\gl K$. Moreover, these actions factor in $\CAlg_{\Ga(\M)}$.
\end{prop}

\begin{proof}
    The group $K$ must be cyclic, else $\Phi^K A = 0$, so suppose $K=C_n$. In this case, we use \Cref{ex:subforKO} to identify
    \[\Sub(\dual{K}, \G) \simeq \M \times \Spec \Sph[\tfrac{1}{n}],\]
    so in particular, $\Inj(\dual{K}, \G)$ also naturally lives over $\M \times \Spec \Sph[\tfrac{1}{n}]$. The $\Aut(BK)$-action of $\Inj(\dual{K}, \G)$ from \Cref{pr:action_factors_to_inj} is not just one of stacks over $\M$, but as stacks over $\Sub(\dual{K}, \G)$; \emph{a priori} only the $\Aut(\dual{K})$-action lives over this subgroup stack. Write the quotient map as
    \[\varphi \colon \Y = \Inj(\dual{K}, \G) \to \Sub(\dual{K}, \G) = \X.\]
    The canonical $\Aut(BK)$-action on $\Phi^K A$ is then the composite
    \[\Aut(BK) \to \Aut_{\Stk_{/\X}}(\Y) \to \Aut_{\CAlg(\QCoh(\X))}(\varphi_\ast \O_\Y) \to \Aut_{\CAlg}(\Phi^K A);\]
    the subscripts indicate the category in which these automorphism groups are taking place. We claim that $\Aut_{\CAlg(\QCoh(\X))}(\varphi_\ast \O_\Y)$ is discrete as a group object in spaces. Indeed, by \Cref{cor:galois_from_torsor}, $\O_\X \to \varphi_\ast \O_\Y$ is a faithful $\Aut(\dual{K})$-Galois extension. By \cite[Lm.9.1.2]{johnrognessdualisinggorups}, this map witnesses $\varphi_\ast \O_\Y$ as a \emph{separable} $\E_\infty$-$\O_\X$-algebra, meaning that the multiplication map has a section as an $\O_\X$-$\O_\X$-bimodule. Ramzi's theory of separable algebras, in particular \cite[Cor.3.42]{ramzi_separability}, implies that the endomorphism space of $\varphi_\ast \O_\X$ as an $\E_\infty$-$\O_\X$-algebra is then discrete, hence the same is true for the automorphism space. 

    The map $\Aut(BK) \to \Aut_{\CAlg(\QCoh(\X))}(\varphi_\ast \O_\Y)$ then factors through the truncation of $\Aut(BK)$ as the target is discrete, which yields the desired factorisation. This then also forces the $W_H K$-action to factor through $\Aut(K)$, meaning it factors through the $W_H^\gl K$-action of \Cref{df:global_weyl_action}.
\end{proof}

The same proof yields the analogous result for an oriented $\P$-divisible group arising as the torsion of an elliptic curve using \Cref{ex:subforTMF}; this example will play no further role here.

\begin{prop}\label{pr:no_action_conflicts_part2_ellipticcurves}
    Let $A =\Ga(\G)$ where $\G$ is the torsion of an elliptic curve. Then for any $n\geq 1$, the $\Aut(B(C_n \times C_n))$-action on $\Phi^K A$ factors through $\Aut(C_n \times C_n)$.
\end{prop}

Naturally, we then ask the following:

\begin{question}
    Does the $\Aut(BC_n)$-action on the $C_n$-geometric fixed points of equivariant elliptic cohomology theories also factor through $\Aut(C_n)$? What about the $\Aut(BK)$-action on the $K$-geometric fixed points of an arbitrary tempered cohomology theory?
\end{question}

Similar arguments also show that the $\Aut(BK)$-action is a purely higher categorical phenomenon. For example, the following shows that the induced $\Aut(BK)$-action on $K$-homotopy groups of a global spectrum always factors through $\Aut(K)$. This is why Schwede \cite{s} only considers the action of $\Aut(K)$ (or $\mathrm{Out}(K)$ if $K$ is not abelian) acting on homotopy groups.

\begin{prop}\label{pr:1categorical_argurment}
    Let $X$ be a global spectrum, $Y$ an object in a category $\calC$ such that $\Aut(Y)$ is $0$-truncated, and $\Aut(X^K) \to \Aut(Y)$ a map of $\E_1$-monoids. Then the induced $\Aut(BK)$-action on $Y$ factors through $\Aut(K)$.
\end{prop}

\begin{proof}
    As $\Aut(Y)$ is $0$-truncated, the composite $\Aut(BK) \to \Aut(X^K) \to \Aut(Y)$ factors through the truncation map, as desired.
\end{proof}

%%%%%%%%%%%%%%%%%%%%%%%%%%%%%%%%%%%%%%%%%%%%%%%%%%%%%%%%%%%%%%%
\subsection{Examples}\label{ssec:examples}
In this subsection, we apply \Cref{main_algebraic} to particular oriented $\P$-divisible groups $\G$.

\subsubsection{Equivariant complex topological \texorpdfstring{$K$}{K}-theory}\label{sssec:KU}
Recall from the introduction that we write $\gKU$ for the global spectrum associated with the $\G = \mu_{\P^\infty}$ over $\Spec \KU$, see \cite[\textsection2.8]{ec3} and \cite[Ex.5.1.4]{temperedglobal}, and that $\Phi^{K} \gKU = 0$ if $K$ is not a cyclic group and otherwise
\begin{equation}\label{eq:geoFP_KU}\Phi^{C_n} \gKU \simeq \KU[\tfrac{1}{n}, \zeta_n], \qquad n\geq 2.\end{equation}

The case of the holomorph $H=\Hol(C_n)$ and $K=C_n$ was mentioned in the introduction: it yields a faithful $\Aut(C_n)$-Galois action on $\KU[\tfrac{1}{n}, \zeta_n]$. The linearity of this faithful Galois action over $\KU$---it arises from a torsor over $\Spec \KU$---combined with Lurie's étale rigidity (\cite[Th.7.5.4.2]{ha}) forces it to be the usual cyclotomic action. In particular, its fixed points are $\KU[\tfrac{1}{n}]$. One can also use \Cref{ex:subforKO} to compute these fixed points.

\begin{example}\label{ex:universalKUexample}
    If $n$ is odd, then restricting to the subgroup $C_2 = \{\pm 1\}\leq \Aut(C_n)$ decomposes the $\Aut(C_n)$-Galois extension as the composite of the two Galois extensions
    \[\KU[\tfrac{1}{n}] \xrightarrow{\Aut(C_n)/C_2} \KU[\tfrac{1}{n}, \zeta_n + \bar{\zeta}_n] \xrightarrow{C_2} \KU[\tfrac{1}{n}, \zeta_n].\]
\end{example}

\begin{example}\label{ex:intro_KU_two}
    Given $H=Q_8$ or $D_8$ either the quaternion group or dihedral group of order $8$, and $K=C_4$ is an index $2$ subgroup. In any of these cases, we have $W_H^\gl K = C_2$, and one can check by hand that $K\leq H$ is Weyl faithful. This yields the familiar $C_2$-Galois extension
    \[\KU[\tfrac{1}{2}] \simeq \Phi^{C_4} \gKU^{hC_2} \to \Phi^{C_4} \gKU \simeq \KU[\tfrac{1}{2}, i],\]
    where $i$ is a square root of $-1$. For $H=Q_8$, this is also an example where the quotient map $H \to H/K$ does not split.
\end{example}

These Galois extensions also recover a classical computation in chromatic homotopy theory.

\begin{example}
By \cite[\textsection4.6]{ec2}, each $K(n)$-local complex-periodic $\E_\infty$-ring $A$ comes equipped with an oriented \emph{Quillen $p$-divisible group} $\G^{\mathcal{Q}}_A$ at the ambient prime $p$ of height $n$. The associated tempered cohomology theory $\ul{A}$ is simply the Borel theory associated with $A$, meaning that for all finite groups $H$, the Atiyah--Segal comparison map
\[\ul{A}^{BH} \to A^{BH}\]
is an equivalence of $\E_\infty$-rings. In particular, $\Phi^K \ul{A} = A^{\tau K}$ is the \emph{proper Tate construction}, more-or-less by definition of the latter; see \cite[Ex.5.1.3]{temperedglobal}. In particular, $\Phi^{C_\ell} A = A^{tC_\ell}$ is the classical Tate construction for a prime $\ell$, and is only nonzero for $\ell=p$. It is then classical that 
\[\pi_\ast A^{tC_p} \simeq \pi_\ast A\llpar x\rrpar/[p]_f(x),\qquad |x|=-2\]
where $[p]_f(x)$ is the $p$-series of the Quillen formal group associated with $A$; see \cite[Lm.2.1]{ando_morava_sadofsky}, for example. For $A=\KU_p$, we see that the $\F_p^\times = \Aut(C_p)$-action on
    \[\Phi^{C_p} \ul{A} \simeq \KU_p^{tC_p} \simeq \Q_p[\zeta_p, u^\pm], \qquad |u|=2\]
is the cyclotomic action by \Cref{main_algebraic}. This action is well-known; see \cite{nikolausscholze}, for example.
\end{example}

\subsubsection{Equivariant real topological \texorpdfstring{$K$}{K}-theory}\label{sssec:KO}
Recall that one can define $\KO$ as the global sections of $\Spec \KU /C_2$, where the $C_2$-action is the classical complex conjugation action. As hinted in \cite[Rmk.3.8]{lurieecsurvey}, this quotient stack has a moduli interpretation as the \emph{moduli stack of oriented tori} $\M_\Tori^\ori$, a torus over a stack $\M$ being a flat affine abelian stack over $\M$ that is fpqc locally isomorphic to the multiplicative group scheme $\G_m$. In lieu of this characterisation, let us construct an oriented torus over $\Spec \KU / C_2$ directly. By \cite{akhilandlennart} or \cite[Ex.4.2.3.2 \& Th.4.4.0.2]{reconstruction}, the stack $\Spec \KU / C_2$ is 0-affine.

\begin{prop}\label{pr:universalorientedtorus}
    There is an oriented torus $\calT$ over $\Spec \KU / C_2$ whose base change to $\Spec \KU$ is isomorphic to $\G_m$. In particular, the underlying torsion object of $\calT$ is an oriented $\P$-divisible group $\calT[\P^\infty]$.
\end{prop}

In future work, we hope to prove that $\calT$ defines an equivalence
\[\calT \colon \Spec \KU / C_2 \xrightarrow{\simeq} \M_\Tori^\ori,\]
ie, to show that $\calT$ is the \emph{universal oriented torus}. Recall that an abelian group object in a category $\calC$ with finite products is a product-preserving functor $\Lat^\op \to \calC$ where $\Lat$ is the category of finitely generated free abelian groups.

\begin{proof}
    Let $\PreAb(-) \colon \Stk^\op \to \Cat$ be the functor sending a stack $\M$ to the category of \emph{preoriented abelian $\M$-stacks}
    \[\PreAb(\M) = \Ab(\Stk_{/\M})_{\BT/},\]
    where $\BT$ is the constant abelian $\M$-stack on the abelian group object $\BT = \Sigma^2 \Z$ in $\Ab(\Spc) = \Mod_\Z^\cn$. Let $\PreAb = \int_{\Stk} \PreAb(-)$ be the Cartesian unstraightening of this functor. The pair $\G_m$ over $\Spec \KU$ defines an object of $\PreAb$. Recall that the Adams operation $\psi^{-1} \colon \KU \to \KU$ can be defined using the universal property of $\KU$ as the orientation classifier of (the formal group of) $\G_m$
    \[\CAlg(\KU, \KU) \simeq \Ori(\widehat{\G}_m/\KU),\]
    as the endomorphism of $\KU$ associated with the orientation $[-1](e)$ of $\widehat{\G}_m$, where we twist the canonical orientation $e$ by the involution $[-1] \colon \G_m \to \G_m$; this same argument was given for $\KU\llpar q \rrpar$ in \cite[\textsection6]{globaltate}. Through the functor $\Tori \to \PreAb$ sending a torus $\T$ to its base change over its orientation classifier, the $C_2$-action on $\G_m$ over $\Spec \Sph$ induces a $C_2$-action on the pair $\G_m/\Spec \KU$ in $\PreAb$ explicitly seen via the pair of maps
    \[\psi^{-1} \colon \Spec \KU \to \Spec \KU, \quad [-1] \colon (\G_m,e) \to (\G_m, [-1]e) = (\psi^{-1})^\ast(\G_m, e),\]
    the first of stacks and the second of preoriented abelian stacks over $\Spec \KU$. We then define the pair $\calT$ over $\Spec \KU / C_2$ by taking the $C_2$-coinvariants of the above $C_2$-action in $\PreAb$; keep in mind that $\PreAb$ has all small colimits and finite limits as this is true for each $\PreAb(\M)$ and the functors $f^\ast \colon \PreAb(\M') \to \PreAb(\M)$ induced by maps of stacks $f\colon \M \to \M'$ preserve these. The canonical preorientation on $\calT$ defines an orientation as this can be checked on an fpqc cover of $\Spec \KU /C_2$, and its true by construction for $\G_m$ over $\Spec \KU$. Moreover, for each $n\geq2$ the $n$-fold multiplication map $[n] \colon \calT \to \calT$ is an affine and finite flat map over $\Spec \KU / C_2$, as it suffices to check this on an fpqc cover, and again it holds for $\G_m$ over $\Spec \KU$. In particular, the torsion of $\calT$ defines a $\P$-divisible group $\calT[\P^\infty]$. Moreover, the orientation of $\calT$ induces a orientation on its torsion $\calT[\P^\infty]$ by \cite[Pr.6.1]{elltempcomp}.
\end{proof}

Write $\gKO$ for the tempered cohomology theory associated with $\calT[\P^\infty]$ over $\Spec \KU / C_2$. One can show that $\gKO \to \gKU$ is a $C_2$-Galois extension in global $\E_\infty$-rings, which suggests that $\gKO$ is a good name for this theory; see \cite[Th.9.17]{nilpotenceanddescentinequivariant}.

Similar Galois extensions from the geometric fixed points of $\gKU$ also hold for $\gKO$, with some systematic changes. Let $\M = \Spec \KU /C_2$ and $\G=\calT[\P^\infty]$.

\begin{example}
By \cite[Thm.5.3.10]{temperedglobal}, there is a natural equivalence of $\E_\infty$-rings
\begin{equation}\label{eq:base_change_for_geoFPKO}\KU\otimes_{\KO} \Phi^{K} \gKO \xrightarrow{\simeq} \Phi^{K} \gKU.\end{equation}
As $\KO \to \KU$ is a faithful $C_2$-Galois extension, $\Phi^K \gKO$ also vanishes if and only if $K$ is noncyclic. For $K=C_n$ we have
\[\Phi^{C_n} \gKO = \Phi^{C_n} \Ga(\O_{\calT[\P^\infty]}) \simeq \Ga(\Phi^{C_n} \calT[\P^\infty]).\]
We can pull back the $C_2$-torsor $\Spec \KU \to \Spec \KU / C_2$ to a $C_2$-torsor $\Spec \Phi^{C_n} \gKU \to \Phi^{C_n}\calT[\P^\infty]$, see \Cref{pr:G-torsors_and_basechange,pr:chromatic_positive_answer}, and use this cover to compute the global sections above in terms of homotopy fixed points
\[\Ga(\Phi^{C_n} \calT[\P^\infty]) \simeq (\Phi^{C_n} \gKU)^{hC_2}.\]
For $n\geq 3$, this $C_2$-action on $\Phi^{C_n}\gKU$ is easy to describe: it is induced by the complex conjugation action on $\KU^{C_n}$. In particular, on $\pi_0 \Phi^{C_n} \gKU \simeq \Z[\tfrac{1}{n}, \zeta_n]$ this $C_2$-action sends $\zeta_n$ to $\zeta_n^{-1} = \bar{\zeta}_n$, the restriction of the $\Aut(C_n)$-action of \Cref{ex:universalKUexample}. In particular, this $C_2$-action on $\pi_0$ is free, so by \Cref{df:lurie_galois} and the discussion following it, we can identify
\[\Phi^{C_n} \gKO \simeq (\Phi^{C_n} \gKU)^{hC_2} \simeq \KU[\tfrac{1}{n}, \zeta_n + \bar{\zeta}_n]\]
as $\E_\infty$-$\KO$-algebras for $n\geq 3$. For instance, we have
\[\Phi^{C_3} \gKO \simeq \KU[\tfrac{1}{3}], \qquad \Phi^{C_4} \gKO \simeq \KU[\tfrac{1}{2}], \qquad \Phi^{C_5} \gKO \simeq \KU[\tfrac{1}{5}, \sqrt{5}].\]
For $n=2$, we similarly have
\[\Phi^{C_2} \gKO \simeq (\Phi^{C_2}\gKU)^{hC_2} \simeq \KU[\tfrac{1}{2}]^{hC_2} \simeq \KO[\tfrac{1}{2}] \simeq \KU[\tfrac{1}{2}] \oplus \Sigma^2 \KU[\tfrac{1}{2}]\]
by direct computation, all equivalences being multiplicative except the last.
\end{example}

\begin{remark}
    One can also argue with level structure for tori as in \cite{km}. Indeed, for $n\geq 3$, the classical stack $\Inj(\dual{C_n}, \mu_{\P^{\infty}})^\heartsuit$ over $\Spec \Z/C_2$ is affine. This follows immediately from the fact that an automorphism of a torus which is the identity on the torsion subgroup $\mu_n$ must be the identity, so the automorphism groups associated with the above stack are all trivial.
\end{remark}

We can now run variants of the examples for $\gKU$ from \Cref{sssec:KU} now with $\gKO$. For instance, \Cref{main_algebraic} shows that the map of $\E_\infty$-$\KO$-algebras
\[\KO[\tfrac{1}{n}] \simeq \Phi^{C_n} \gKO^{h\Aut(C_n)} \to \Phi^{C_n} \gKO \simeq \KU[\tfrac{1}{n}, \zeta_n + \bar{\zeta}_n],\]
the first equivalence follows from \Cref{ex:subforKO}, is a faithful $\Aut(C_n)$-Galois extension corresponding to an interesting $\Aut(C_n)$-action on the right side that twists of the residual $\KU$-linear $\Aut(C_n)/\{\pm 1\}$-action together with the usual $\KO$-linear complex conjugation action. This also highlights that this example is not a $\pi_0$-Galois extension in the sense of \Cref{df:lurie_galois}. In fact, we expect most nonaffine examples to not be $\pi_0$-Galois extensions, such as in the $\TMF$-examples to follow.

\subsubsection{Katz--Mazur and Tate \texorpdfstring{$K$}{K}-theory}
Recall the oriented $\P$-divisible group $\T_\KM$ over $\KU[q^\pm]$ constructed in \cite[Thm.4.1]{globaltate} whose associated tempered cohomology theory $\gKU_\KM$ we call \emph{global Katz--Mazur $K$-theory} \cite[Ex.5.1.6]{temperedglobal}; also known as \emph{quasi-elliptic cohomology} in \cite{globaltate,quasielliptic}. For a positive integer $n$, we have
\[\pi_\ast^{C_n} \gKU_\KM \simeq \prod_{i=0}^{n-1} \frac{\Z[q^\pm, x_i]}{x_i^n-q^m}.\]
We fall short of computing the geometric fixed points in this case, but we do want to show that $\Phi^{C_n} \gKU_\KM$ is integral, meaning no integer $n$ acts invertibly on this spectrum. This follows from the existence of a map of $\E_\infty$-rings $\Phi^{C_n} \gKU_\KM \to A$ with $A$ integral. Such an $A$ is given by Lurie's \emph{splitting algebra} $\mathfrak{S}_n$ of \cite[\textsection2.7]{ec3}, which is the universal $\E_\infty$-$\KU[q^\pm]$-algebra such that the base change of $\T_\KM[n]$ splits as $\mu_{n} \times \dual{C_n}$. Indeed, in \cite[Lm.7.27]{globaltate}, we compute the homotopy groups of Lurie's splitting algebra $\mathfrak{S}_n$ to be
\[\pi_\ast \mathfrak{S}_n \simeq \Z[q^{\pm1/n}, u^\pm], \qquad |u|=2\]
using the classical splitting algebra of Katz--Mazur. In particular, this splitting algebra comes with an injection $\dual{C_n} \to \T_\KM$, giving a map of $\E_\infty$-rings $\Phi^{C_n} \gKU_\KM \to \mathfrak{S}_n$.

Over $\KU\llpar q\rrpar$ this oriented $\P$-divisible group $\T_\KM$ is the torsion of the \emph{oriented Tate curve} $\T$. We write $\gKU_\Tate$ for the associated global $\E_\infty$-ring known as \emph{orbifold (or global) smooth Tate $K$-theory}. In \cite[Thm.D]{globaltate}, we identified
\[\gKU_\KM \otimes_{\KU[q^\pm]} \KU\llpar q \rrpar {\simeq} \gKU_\Tate\]
as global $\E_\infty$-rings. In particular, by base change for geometric fixed points \cite[Thm.5.3.10]{temperedglobal}, we obtain maps of $\E_\infty$-rings
\begin{equation}\label{eq:basechange_to_geoFP_of_Tate_ktheory}\Phi^{C_n} \gKU_\Tate \xleftarrow{\simeq} \Phi^{C_n} \gKU_\KM \otimes_{\KU[q^\pm]} \KU\llpar q \rrpar \to \mathfrak{S}_n \otimes_{\KU[q^\pm]} \KU\llpar q \rrpar,\end{equation}
again showing that $\Phi^{C_n} \gKU_\Tate$ is also integral.

One can also amalgamate these arguments with those for $\gKO$ when discussing $\gKO_\KM$ and $\gKO_\T$, defined from $\Spec \KU[q^\pm] / C_2$ and $\Spec \KU\llpar q \rrpar / C_2$; see \cite[(6.3)]{globaltate} for more. We hope to come back to more explicit computations in future work.

\subsubsection{Topological modular forms}\label{sssec:TMF}
Let $\M = \M_\Ell^\ori$ be the moduli stack of oriented elliptic curves from \cite[\textsection7]{ec2}. This stack is $0$-affine by the main theorem of \cite{akhilandlennart} or \cite[Thm.A \& Ex.4.2.3.4]{reconstruction}. We also write $\G = \calE[\P^\infty]$ for the torsion of the universal oriented elliptic curve over $\M$. The associated global spectrum is written as $\gTMF$. By \cite[Ex.5.5.20]{temperedglobal}, we have $\Phi^K \gTMF = 0$ for those finite abelian groups $K$ that require more than $2$ generators. Many examples of $\Phi^K \gTMF$ were already discussed in \cite[Ex.5.5.20]{temperedglobal} and more explicit formulas appear in \cite{geometricnorms}.

\begin{example}\label{ex:subforTMFgamma1(n)}
The $C_n$-case appeared in the introduction. This gave us the faithful $\Aut(C_n)$-Galois extensions
\begin{equation}\label{eq:refined_level_structures}
    \Phi^{C_n} \gTMF^{h\Aut(C_n)} \to \Phi^{C_n} \gTMF
\end{equation}
associated with the holomorph $H=\Hol(C_n)$; they are integral extensions of the more classical Galois extensions
\begin{equation}\label{eq:level_structures_example}
    \gTMF_0(n) \to \gTMF_1(n)
\end{equation} 
of \cite[\textsection7]{akhilandlennart}. By \cite[Thm.C]{globaltate}, there is a map of global $\E_\infty$-rings $\gTMF \to \gKU_\Tate$, so in particular, from (\ref{eq:basechange_to_geoFP_of_Tate_ktheory}) we see that $\Phi^{C_n} \gTMF$ is integral. Blue shift for geometric fixed points \cite[Cor.G]{temperedglobal} does show that a power of $v_1$ does act as an isomorphism on $\Phi^{C_n} \gTMF$ at each prime $p$ dividing $n$; also see \cite[Ex.5.5.20]{temperedglobal}. To see that (\ref{eq:refined_level_structures}) recovers (\ref{eq:level_structures_example}) after inverting $n$, note that the stack $\Inj(\dual{C_n},\calE[\P^\infty])$ is always affine over $\M_\Ell^\ori$, and this map is also étale after after inverting $n$. This reduces us to computing this stack together with its $\Aut(\dual{C_n})$-action classically, which is the definition of the étale stack $\M_1(n)$ over the classical moduli stack of elliptic curves. 
\end{example}

Computations of $\pi_\ast \Phi^{C_n} \gTMF$ can be made explicit too. A deeper discussion, including more general examples and actual proofs, appears in \cite{geometricnorms}. For instance, and thank you to William Balderrama for showing this example to us, there is an isomorphism of graded rings
    \[\pi_\ast \Phi^{C_3} \gTMF \simeq \Z[a_1, a_3, \Delta^{-1}], \qquad  |a_i|=2i, |u|=2, \Delta = a_3^3(a_1^3-27a_3).\]
In particular, the ring $\pi_\ast \TMF_1(3)$ is simply $\pi_\ast \Phi^{C_3} \gTMF$ with $3$ inverted. One can also compute the homotopy groups of the $\Aut(C_3) \simeq C_2$-homotopy fixed points of these $C_3$-geometric fixed points, and they lie in a pullback of classical rings
    \[\begin{tikzcd}
        {\pi_\ast \Phi^{C_3} \gTMF^{h\Aut(C_3)}}\ar[r]\ar[d] &   {\pi_\ast \TMF_0(3)}\ar[d] \\
        {\pi_\ast \Phi^{C_3} \gTMF}\ar[r] &   {\pi_\ast \TMF_1(3),}
    \end{tikzcd}\]
expressing that $\pi_\ast \Phi^{C_3} \gTMF^{h\Aut(C_3)}$ is also just the the naïve delocalisations of $\pi_\ast\TMF_0(3)$ of Mahowald--Rezk \cite{levelonethree}.

If $K$ requires at least two generators, a similar argument allows us to compute $\Inj(\dual{K}, \calE[\P^\infty])$ and its associated stack of subgroups; recall \Cref{ex:subforTMF}

\begin{example}
    Setting $K=C_n\times C_n$ and $H=\Hol(K)$, the global Weyl group in this case is $\GL_2(\Z/n)$, which then yields
    \[\TMF[\tfrac{1}{n}] \simeq \Phi^{C_n\times C_n} \gTMF^{h\GL_2(\Z/n)}  \to \Phi^{C_n\times C_n} \gTMF \simeq \TMF(n),\]
    recovering, on the nose, the classical faithful $\GL_2(\Z/n)$-Galois extension induced by $\Ga(n)$-level structures on elliptic curves; the first and last equivalences follow from \Cref{ex:subforTMF}.
\end{example}

\begin{remark}
    Combining \Cref{ex:subforTMFgamma1(n),ex:subforTMF} together, one also obtains spectral refinements of the various affine flat torsors appearing in \cite[(7.4.3)]{km}
    \[\begin{tikzcd}
        {\Inj(\dual{C_n \times C_n}, \calE[\P^\infty])}\ar[r, "{A_n}"]\ar[rrr, "{B_n}", bend left = 15]\ar[rrrr, bend right = 15, "{GL_2(\Z/n)}"]   &   {\Inj(\dual{C_n}, \calE[\P^\infty])}\ar[rr, "{GL_1(\Z/n)}"]  &&   {\Sub(\dual{C_n}, \calE[\P^\infty])}\ar[r]  &   {\M_\Ell^\ori}
    \end{tikzcd} \]
    after inverting $n$, where the names of the morphisms indicate what kind of torsor they are and subgroups of $GL_2(\Z/n)$ are abbreviated as
    \[A_n = \begin{pmatrix}
    1    &  \ast  \\
    0    &  \ast
    \end{pmatrix}, \qquad 
    B_n = \begin{pmatrix}
    \ast    &  \ast  \\
    0    &  \ast
    \end{pmatrix}.\]
\end{remark}

We round out this section with more examples where $K$ requires at least two generators of different orders. Let us write $K=C_m \times C_n$ for a pair of integers $m,n\geq 2$ with $m|n$.

\begin{example}
    If $n=mr$ and $(m,r)=1$, then one has
    \[C_m \times C_n \simeq (C_m\times C_m)\times C_{n/m}\]
    inducing identifications of stacks
    \[\Inj(\dual{C_m \times C_n}, \calE[\P^\infty]) \simeq \Inj(\dual{C_{n/m}}, \G_{\M^\ori(m)})\]
    In this case, we have identifications
    \[ \Phi^{C_m \times C_n} \gTMF \simeq \Phi^{C_{n/m}} \gTMF(m)\]
    where $\gTMF(m)$ is the tempered cohomology theory associated with the pullback of the universal elliptic curve to the moduli stack of oriented elliptic curves with $\Ga(m)$-level structure $\M_\Ell^\ori(m)$. Applying \Cref{main_algebraic} yields a faithful $W=\Aut(C_m \times C_n)$-Galois extension
    \[\Phi^{C_{r}} \gTMF(m)^{hW} \simeq \Phi^{C_m \times C_n} \gTMF^{hW} \to \Phi^{C_m \times C_n} \gTMF \simeq \Phi^{C_{r}} \gTMF(m).\]
    This is an almost integral (at least $r$ is not inverted) refinement of the faithful Galois extension
    \[\TMF_0(r) \simeq (\TMF(m) \otimes_{\TMF} \TMF_1(r))^{hW} \to \TMF(m) \otimes_{\TMF} \TMF_1(r)\]
    by the group
    \[W = \Aut(C_m \times C_n) \simeq \Aut(C_m \times C_m) \times \Aut(C_{r}).\]
\end{example}

On the other extreme, if $p|m$ if and only if $p|n$, then $\Inj(\dual{K}, \calE[\P^\infty])$ is étale and affine over $\M$. Using the standard geometric form of Lurie's étale rigidity, we are reduced to identifying their underlying classical stacks, which have already appeared in the literature; for example, $\Sub(\dual{K}, \calE[\P^\infty])$ is the same as the classical stack $\mathcal{M}_K$ of \cite[Def.2.26]{heckeontmf} equipped with the structure sheaf pulled back from $\M_\Ell^\ori$. 

We hope to come back to more explicit computations of these examples, as well as other examples, using the synthetic methods of \cite{smfcomputation,osyn,chua_c2TMF}. It could also be interesting to come back to higher height examples from geometry, such as the tempered cohomology theories associated with the \emph{topological automorphic forms} of \cite{taf}; also see \cite[\textsection 6.3]{luriestheorem}.

%%%%%%%%%%%%%%%%%%%%%%%%%%%%%%%%%%%%%%%%%%%%%%%%%%%%%%%%%%%%%%%
%%%%%%%%%%%%%%%%%%%%%%%%%%%%%%%%%%%%%%%%%%%%%%%%%%%%%%%%%%%%%%%
%%%%%%%%%%%%%%%%%%%%%%%%%%%%%%%%%%%%%%%%%%%%%%%%%%%%%%%%%%%%%%%
\section{Decompositions of equivariant perfect module categories}\label{sec:categorical_decompositions}

\begin{center}
    \emph{Fix an $H$-equivariant $\E_\infty$-ring $R \in \CAlg(\Sp_H)$ for this section.}
\end{center}

%%%%%%%%%%%%%%%%%%%%%%%%%%%%%%%%%%%%%%%%%%%%%%%%%%%%%%%%%%%%%%%
%\subsection{General KNPR-decompositions}\label{ssec:knpr_decompositions}
In \cite{krause_equivariantpicard}, Krause shows how the category of perfect $H$-spectra $\Sp^\omega_H$ can be decomposed using its geometric fixed points. This was generalised to a similar decomposition of $\Perf_H(R) = \Mod_H(R)^\omega$ in \cite{NPR2024} by Naumann--Pol--Ramzi. 

To be more precise, fix a family $\calF$ of subgroups of $K\leq H$ closed under subgroups and conjugacy. We write
\[R /\calF = R \otimes \tildE \calF \in \CAlg(\Sp_H),\]
where $\tildE \calF$ is the universal based $H$-space whose $K$-fixed points are $S^0$ for all $K\notin \calF$ and a single point for $K\in \calF$; see \cite[\textsection6.1]{nilpotenceanddescentinequivariant}, where this space is shown to have a natural $\E_\infty$-structure. We read $R/\calF$ as having ``killed the $\calF$-data in $R$''. Let us also abbreviate $\calF \cup \{(H)\} = \calF \cup H$ when adding a single conjugacy class $(H)$ not already in $\calF$. In particular, $R=R/\varnothing$.

The precise phrasing of the following statement appears as \cite[Thm.8.2]{equivariant_separable_algebras_NPR2}.

\begin{theorem}[{\cite[Thm.7.13]{NPR2024}}]\label{thm:npr_decompositions}
    Given a subgroup $K\leq H$ such that $K\notin \calF$ yet all proper subgroups of $K$ lie in $\calF$, then we have a Cartesian diagram of 2-rings
    \[\begin{tikzcd}
        {\Perf_H(R / \calF)}\ar[r]\ar[d, "{\Phi^K}"]   &   {\Perf_H(R / \calF \cup K)}\ar[d]   \\
        {\Perf(\Phi^K R)^{hW_H K}}\ar[r]    &   {\Perf(\Phi^K R)^{tW_H K},}
    \end{tikzcd}\]
    where the Tate construction takes place in $\Cat^\perf$.
\end{theorem}

We call such an iterated expression for $\Perf_H(R)$ a \emph{KNPR-decomposition}.

If $R$ is the underlying $H$-spectrum of a global spectrum $A$, then the above Weyl action on $\Perf(\Phi^K R)$ is the coherent action of \Cref{df:canonical_weylaction}. By \Cref{ssec:action_comparison}, this factors through the discrete action of $W_H^\gl K$ on $\Phi^K A=\Phi^K \Ga(\G)$ in some cases of interest.

An immediate consequence of \Cref{thm:npr_decompositions} is the following.

\begin{prop}\label{pr:generic_decomposition}
    Let $H$ be a finite group such that for each conjugacy class of nontrivial subgroups $K\leq H$ in the derived defect base of $R$, the 2-ring $\Perf(\Phi^K R)^{tW_H K}$ vanishes. Then we have a pullback of 2-rings
    \[\begin{tikzcd}
        {\Perf_H(R)}\ar[r]\ar[d]    &   {\prod_{(K)\leq H} \Perf(\Phi^K R)^{hW_H C}}\ar[d]  \\
        {\Perf(R)^{hH}}\ar[r]   &   {\Perf(R)^{tH}.}
    \end{tikzcd}\]
\end{prop}

\begin{proof}
    The first KNPR-decomposition is nontrivial, giving us the pullback of \Cref{thm:npr_decompositions} for $\calF = \varnothing$. For any nonempty family $\calF$, we have
    \[\Perf_H(R/\calF) \xrightarrow{\simeq} \Perf(\Phi^K R)^{hW_H K} \times \Perf_H(R/\calF \cup K),\]
    as by assumption the Tate term $\Perf(\Phi^K R)^{tW_H K}$ vanishes. This inductively gives the desired pullback of 2-rings.
\end{proof}

The following gives a nonexhaustive list of conditions when $\Perf(\Phi^K R)^{tW_H K}$ vanishes.

\begin{prop}\label{pr:generic_tate_vanishing}
    Let $K\leq H$ be a subgroup. Suppose that at least one of the following conditions holds:
    \begin{enumerate}
        \item The Weyl group $W_H K$ is trivial.
        \item The order of the Weyl group $W_H K$ is a unit in $\pi_0 \Phi^{K} R$.
        \item The Weyl action of $W_H K$ on $\Phi^K R$ witnesses $\Phi^K R^{hW_H K} \to \Phi^K R$ as a faithful $W_H K$-Galois extension.
    \end{enumerate}
    Then $\Perf(\Phi^K R)^{tW_H K}$ vanishes.
\end{prop}

We will use a variant of \cite[Lm.4.2]{krause_equivariantpicard} that we learned from Luca Pol.

\begin{lemma}\label{lm:42_krause_generalised}
    Let $G$ be a finite group, $R \in \CAlg^{BG}$, and $X,Y\in \Mod_R(\Sp^{BG})$ be modules which are dualisable in $\Mod_R$. Then there is a natural identification of $R$-modules
    \[\map_{\Perf(R)^{tG}}(X,Y) \simeq \map_{\Mod(R)}(X,Y)^{tG}.\]
    In particular, there is a natural identification of $\E_\infty$-$R$-algebras
    \[\map_{\Perf(R)^{tG}}(\1,\1) \simeq R^{tG}.\]
\end{lemma}

\begin{proof}
    One can copy Krause's proof of \cite[Lm.4.2]{krause_equivariantpicard}, the only difference being that Krause assumes that $G$ acts trivially on $R$. In that sense, one need only replace occurrences of $R[G]$ with $G \circledast R$ of \cite[Not.6.8]{NPR2024}, both simply being the image of $R$ under the left adjoint to the forgetful functor $\Mod_A^{hG} \to \Mod_A$.
\end{proof}

\begin{proof}[Proof of \Cref{pr:generic_tate_vanishing}]
    Vanishing follows from condition 1 from the definition of Tate fixed points. Conditions 2 and 3\footnote{For the proof of condition 3, one could also use \Cref{pr:mod_and_perf_preserve_Galois} together with the 2-ring generalisation of \cite[Pr.6.3.3]{johnrognessdualisinggorups}.} use \Cref{lm:42_krause_generalised} below, that the endomorphism spectrum of the unit of this Tate category is given by
    \[\map_{\Perf(\Phi^K R)^{tW_H K}}(\1,\1) \simeq \Phi^K R^{tW_H K}.\]
    If condition 2 holds, then $\Phi^K R^{tW_H K}$ dies by computation, and if condition 3 holds, this spectrum dies by Tate vanishing of a faithful Galois extension (\cite[Pr.6.3.3]{johnrognessdualisinggorups}). In both cases, the vanishing of the endomorphism spectrum of the unit implies the vanishing of the category, finishing the proof.
\end{proof}

Combining \Cref{pr:generic_decomposition,pr:generic_tate_vanishing} immediately yields the following:

\begin{cor}\label{cor:generic_decomposition}
    If for each nontrivial subgroup $K\leq H$ in the derived defect base of $R$, at least one of the three conditions of \Cref{pr:generic_tate_vanishing} holds, then we have a pullback of 2-rings
    \[\begin{tikzcd}
        {\Perf_H(R)}\ar[r]\ar[d]    &   {\prod_{(K)\leq H} \Perf(\Phi^K R)^{hW_H C}}\ar[d]  \\
        {\Perf(R)^{hH}}\ar[r]   &   {\Perf(R)^{tH}.}
    \end{tikzcd}\]
\end{cor}

Condition 2 of \Cref{pr:generic_tate_vanishing} is satisfied if the group order is invertible in $R$.

\begin{cor}\label{pr:rationality_decomposition}
    Let $H$ be a finite group with $|H| \in \pi_0^H R$ a unit. Then the map of 2-rings induced by geometric fixed points
    \[\Perf_H(R) \xrightarrow{\simeq} \prod_{(K)\leq H} \Perf(\Phi^H R)^{hW_H K},\]
    indexed over conjugacy classes of subgroups of $H$, is an equivalence. In particular, applying $\Ind(-)$ yields the equivalence in $\CAlg(\PrLst)$
    \[\Mod_H(R) \xrightarrow{\simeq} \prod_{(K) \leq H} \Mod(\Phi^H R)^{hW_H K}.\]
\end{cor}

This is a well-known result. If we invert all primes $p$, this is a consequence of the rational decompositions of equivariant spectra, see \cite[Thm.13.1]{david_magdalena_rationalmodels} for example; the resulting equivalence of upon taking $\E_\infty$-objects on both sides is also the degenerate case of \cite[Th.6.7]{algebraic_model_ucom_BDL}.

\begin{proof}
    Condition 2 of \Cref{pr:generic_tate_vanishing} is satisfied, so we have the pullback of \Cref{cor:generic_decomposition}. Moreover, the endomorphism spectrum $R^{tH}$ of $\Perf(R)^{tH}$ itself vanishes by \Cref{pr:generic_tate_vanishing}, which gives the desired product decomposition.
\end{proof}

\begin{remark}\label{rmk:excisionsquares}
    Each KNPR-decomposition induces a pullback in $\PrLst$ upon taking $\Ind$, and the induced functors
    \[\Ind(\Perf(\Phi^K R)^{hW_H K}) \to \Ind(\Perf(\Phi^K R)^{tW_H K})\]
    are localisations. Indeed, the original pullback appearing in \cite[Thm.7.11]{NPR2024} is for such Ind-objects, and the pullbacks of \Cref{thm:npr_decompositions} are obtained by applying $(-)^\omega$, which in this case agrees with dualisable objects. It is plainly stated in \cite[Thm.7.11]{NPR2024} that the above map is a localisation.
\end{remark}

\begin{remark}\label{rmk:localising_invariants}
Recall that a \emph{localising invariant} is a functor $E\colon \Cat^\perf \to \calC$ into a stable category $\calC$ that sends exact sequences, so diagrams that are simultaneously fibre and cofibre sequences, to fibre sequences in $\calC$. By \cite[Thm.18]{tamme_excision}, localising invariants satisfy a Meyer--Vietoris property with respect to \emph{excision squares}, so those diagrams
\begin{equation}\label{eq:excisionsquare}\begin{tikzcd}
    {A}\ar[r]\ar[d] &   {B}\ar[d]   \\
    {C}\ar[r]       &   {D}
\end{tikzcd}\end{equation}
in $\Cat^\perf$ such that the induced square in $\PrLst$ is a pullback after applying $\Ind$ and $\Ind(C) \to \Ind(D)$ is a localisation, meaning its right adjoint is fully faithful. In particular, by \Cref{rmk:excisionsquares}, we see that every pullback square in a KNPR-decomposition is an excision square. This implies that every pullback of 2-rings found in this section, which we will write as (\ref{eq:excisionsquare}) for a moment, induces a fibre sequence in $\calC$
\[E(A) \to E(B) \oplus E(C) \to E(D)\]
for any localising invariant $E$.
\end{remark}

%%%%%%%%%%%%%%%%%%%%%%%%%%%%%%%%%%%%%%%%%%%%%%%%%%%%%%%%%%%%%%%
\subsection{Decompositions of modules over tempered cohomology theories}\label{ssec:general_decompositions}
One of the motivating examples in writing this article was to use \Cref{main_algebraic} to give us more examples where condition 3 of \Cref{pr:generic_tate_vanishing} is satisfied. 

\subsubsection{The general case for a split group}
The following explicit statement was already given in the introduction.

\begin{cor}\label{pr:categorical_decomposition}
    Let $p,\ell$ be two primes such that $\ell| p-1$ and write $H=C_p \rtimes C_\ell$ via the faithful action $C_\ell \leq C_{p-1} \simeq \Aut(C_p)$. Suppose that base stack $\M$ is $0$-affine. Then there is a Cartesian diagram of 2-rings
    \[\begin{tikzcd}
    {\Perf_{H}(\Ga(\G))}\ar[r]\ar[d]    &   {\Perf(\Phi^{C_\ell}\Ga(\G)) \times \Perf(\Phi^{C_p}\Ga(\G)^{hC_\ell})}\ar[d]  \\
    {\Perf(\Ga(\M))^{hH}}\ar[r] &   {\Perf(\Ga(\M))^{tH}.}
\end{tikzcd}\]
\end{cor}

\begin{proof}
    The only conjugacy classes of nontrivial conjugacy classes of subgroups of $H$ in the derived defect base of $\Ga(\G)$ are $C_p$ and $C_\ell$; by \cite[Pr.5.4.4]{temperedglobal}, we see that $\Phi^{H} \Ga(\G) = 0$. Note that $W_H C_\ell$ is trivial, so condition 1 of \Cref{pr:generic_decomposition} is satisfied. For $C_p \leq H$, we first check by hand that this inclusion is Weyl faithful. Next, we use that our group $H$ is split to apply \Cref{pr:no_action_conflicts_part1}, which shows that the $W_H C_p$-action on $\Phi^{C_p} \Ga(\G)$ factors through the global Weyl group action. The isomorphisms
    \[W_H C_p = H / C_p \xrightarrow{=} H / C_p = W_H^\gl C_p \simeq C_\ell\]
    together with \Cref{main_algebraic} show us that this action is a faithful $C_\ell$-Galois action, giving condition 3 of \Cref{pr:generic_decomposition}. Having verified these conditions for all subgroups, we see that \Cref{cor:generic_decomposition} applies. Moreover, the above faithful $C_\ell$-Galois extension together with Galois descent identifies
    \[\Perf(\Phi^{C_p} \Ga(\G)^{hC_\ell}) \xrightarrow{\simeq} \Perf(\Phi^{C_p} \Ga(\G))^{hC_\ell},\]
    which finishes the proof.
\end{proof}

This seems to be one of the only large families of groups that admit such a simple decomposition for general $\G$. One could continue this pattern of simplifying the KNPR-decompositions. For instance, by writing $\Perf_{Q_8}(\Ga(\G))$ using two pullbacks of nonequivariant 2-rings, or $\Perf_{\SL_2(\F_3)}(\Ga(\G))$ using three pullbacks; we leave these and other examples in this direction to the reader. This suggests a finite-step process for computing invariants of these categories of perfect equivariant modules using long exact sequences, or a spectral sequence with a horizontal vanishing line on the $E_2$-page.

\subsubsection{Equivariant topological \texorpdfstring{$K$}{K}-theories}
These KNPR-decompositions simplify even more dramatically when more Tate constructions vanish. For instance, when $A$ is $\gKU$ or $\gKO$; recall \Cref{sssec:KU,sssec:KO}, respectively. Keep in mind that the derived defect base of these equivariant $\E_\infty$-rings is the collection of all cyclic groups and that the canonical $W_H K$-action on geometric fixed points always factors through $W_H^\gl K$ by \Cref{pr:no_action_conflicts_part2_GM}.

To abbreviate the statements below, recall that the first KNPR-decomposition (\Cref{thm:npr_decompositions}) always takes the form of the pullback of 2-rings
\begin{equation}\label{eq:help_to_translate}\begin{tikzcd}
    {\Perf_H(R)}\ar[r]\ar[d]    &   {\Perf_H(R/e)}\ar[d]    \\
    {\Perf(R)^{hH}}\ar[r]       &   {\Perf(R)^{tH},}
\end{tikzcd}\end{equation}
where $R/e$ is just our notation for $R \otimes \widetilde{E}H$. To recover similar decompositions to \Cref{pr:categorical_decomposition}, it then suffices to show that the functor
\[\Perf_H(R/e) \to \prod_{e\neq (K) \leq H} \Perf(\Phi^K R)^{hW_H K}\]
induced by the geometric fixed points functors is an equivalence. Let us see some more examples now.

Our first observation uses only the rationality of $\Phi^K \gKU$ and $\Phi^K \gKO$; no Galois extensions from \Cref{main_algebraic} are necessary. This starts to paint the picture.

\begin{prop}\label{pr:KU_decomp_abelian}
    For a prime number $p$, a finite $p$-group $H$, and $A=\gKU$ or $\gKO$, geometric fixed points induce an equivalence of 2-rings
    \[\Perf_H(A/e) \xrightarrow{\simeq} \prod_{e\neq (K)\leq H} \Perf(\Phi^{K}A)^{hW_H K}.\]
    If $H$ is abelian, the $W_H K$-actions are all trivial.
\end{prop}

\begin{proof}
    This is an immediate application of \Cref{cor:generic_decomposition} using condition 2 of \Cref{pr:generic_tate_vanishing}, as $|K| \in \pi_0 (\Phi^K A)^\times$ for each cyclic subgroup $K\leq H$. The triviality of the $H$-actions comes from the fact that these Weyl group actions factor through the global Weyl group $W_H^\gl K$, which always vanishes if $H$ is abelian.
    %The KNPR-decomposition associated with $\calF = \varnothing$ and $K=e$ is simply the pullback of 2-rings
    %\[\begin{tikzcd}
    %    {\Perf_{H}(A)}\ar[r]\ar[d]    &   {\Perf_{H}(A/e)}\ar[d]    \\
    %    {\Perf(A)^{hH}}\ar[r]         &   {\Perf(A)^{tH}.}
    %\end{tikzcd}\]
    %We know that $\Phi^K A$ vanishes for noncyclic groups, so to finish our proof, it suffices to show inductively that
    %\begin{equation}\label{eq:splitting_abelian_p-group}\Perf_{H}(A/\calF) \simeq \Perf_{H}(A/\calF\cup K) \times \Perf(\Phi^{K} A)^{hW_HC}\end{equation}
    %for subgroups $K\leq H$ not in $\calF$, but whose subgroups all lie in $\calF$. This is clear, as the KNPR-decomposition states that this holds if and only if the gluing term $\Perf(\Phi^{K} A)^{tW_HK}$ vanishes. Indeed, by \Cref{lm:42_krause_generalised} we compute the endomorphism spectrum of the unit of this Tate construction as
    %\[\map_{\Perf(\Phi^{K} A)^{tW_HK}}(\1,\1) \simeq \Phi^K A^{tW_H K} = 0,\]
    %vanishing coming from the fact that the order of $W_H K$ is invertible in $\Phi^K A$ by assumption. This yields (\ref{eq:splitting_abelian_p-group}) by induction, and hence the desired claim.
\end{proof}

%\begin{prop}
%    For a prime number $p$, there is a pullback of 2-rings
%    \[\begin{tikzcd}
%        {\Perf_{\He_p}(A)}\ar[r]\ar[d]    &   {\Perf(A[\tfrac{1}{p}])^{hC_p\times C_p} \times \prod_{p+1} \Perf(A[\tfrac{1}{2}])^{hC_p}}\ar[d]    \\
%        {\Perf(A)^{h\He_p}}\ar[r]         &   {\Perf(A)^{t\He_p}.}
%    \end{tikzcd}\]
%\end{prop}

%\begin{proof}
%    All cyclic subgroups of $\He_p$ are of order $p$, and there are $p+2$ conjugacy classes of such subgroups. All but one of these conjugacy classes has Weyl group $C_p$ with the trivial action, and the one exception is the center, hence has Weyl group $\He_p/C_p = C_p\times C_p$ acting trivially. In all cases, the gluing data in the associated KNPR-decomposition vanishes, and we obtain the desired pullback.
%\end{proof}

These decompositions can further simplify using condition 3 of \Cref{pr:generic_decomposition}, so if the $W_H K$-action on $\Phi^K A$ is {Galois}. This happens for the quaternion group $Q_8$, for example.

\begin{prop}\label{pr:KU_decomp_Q8}
    For $A=\gKU$ or $\gKO$, geometric fixed points induce an equivalence of 2-rings
    \[\Perf_{Q_8}(A/e) \xrightarrow{\simeq} {\Perf(A[\tfrac{1}{2}])^{hV} \times \prod_{x=i,j,k} \Perf(A[\tfrac{1}{2}])},\]
    where the $V$-actions are trivial.
\end{prop}

\begin{proof}
    Courtesy of \Cref{pr:KU_decomp_abelian}, we only have to identify $\Perf(\Phi^{C_4}A)^{hC_2} \simeq \Perf(A[\tfrac{1}{2}])$. By \Cref{main_algebraic}, we see that $(\Phi^{C_4} A)^{hC_2} \to \Phi^{C_4} A$ is a faithful $C_2$-Galois extension, and we compute the domain as $A[\tfrac{1}{2}]$ for both $A=\gKU$ and $\gKO$.
\end{proof}

A similar result holds for $D_8$ with the same proof idea.

\begin{prop}\label{pr:KU_decomp_D8}
    For $A=\gKU$ or $\gKO$, geometric fixed points induce an equivalence of 2-rings
    \[\Perf_{D_8}(A/e) \xrightarrow{\simeq} {\Perf(A[\tfrac{1}{2}])^{hV} \times \left( \Perf(A[\tfrac{1}{2}])^{hC_2} \right)^2 \times \Perf(A[\tfrac{1}{2}])}\]
    where the $C_2$- and $V$-actions are trivial.
\end{prop}

\begin{remark}
    Similar decompositions of $\Perf_H(A)$ into a pullback of similar nonequivariant 2-rings fail for groups of composite order. This already happens for $C_6$ and the dicyclic group of order $12$.
\end{remark}

One can continue to apply \Cref{cor:generic_decomposition} to $R=\gKU$ and $\gKO$ for many nonabelian groups $H$ too. The game is simply checking that one of the three conditions of \Cref{pr:generic_tate_vanishing} holds for all conjugacy classes of cyclic subgroups $H$. In the third case, one needs to check Weyl faithfulness too, which can be done by hand in what follows. The statements below are all proved using the same techniques we have seen above, so we only give the cyclic subgroup lattice up to conjugacy with Weyl group action, see \Cref{table:1}; except for the second-to-last example \Cref{pr:KU_decomp_A6}, which does not quite fit into \Cref{cor:generic_decomposition}, so we give a proof. We remind the reader of the generic pullback of 2-rings (\ref{eq:help_to_translate}) to help contextualise these results.

First, we consider alternating and symmetric groups.\footnote{Other interesting examples that we will not focus on here are subfamilies of the \emph{Frobenius groups} $F_{p^r} = \F_{p^r} \rtimes \F_{p^r}^\times$ for a prime $p$ and $r\geq 1$, as well as \emph{Heisenberg groups} $He_p$ for odd primes p which are the $p$-Sylow subgroups of $\GL_3(\F_p)$.}

\begin{table}
\begin{center}\begin{tabular}{|c|c|c|c|c|c|c|c|}
\hline
$H \diagdown (C)$   & $C_2$         & $C_2$         & $C_3$         & $C_3$ & $C_4$         & $C_5$         & $C_7$         \\ \hline
$A_4$               & $C_2$ (tr)    &               & $C_1$ (tr)    &       &               &               &               \\ 
$S_4$               & $V$ (tr)      & $C_2$ (tr)    & $C_2$ (Gal)   &       & $C_2$ (Gal)   &               &               \\ 
$A_5$               & $C_2$ (tr)    &               & $C_2$ (Gal)   &       &               & $C_2$ (Gal)   &               \\ 
$A_6$               & $V$ (tr)      &               & $S_3$         & $S_3$ &               & $C_2$ (Gal)   &               \\ 
$\GL_3(\F_2)$       & $V$ (tr)      &               & $C_2$ (Gal)   &       &               &               & $C_3$ (Gal)   \\ \hline
\end{tabular}\end{center}
\caption{Table with rows indexed by groups $H$ and columns by conjugacy classes of cyclic subgroups $(C)$ of $H$, and entries the Weyl group $W_H C$ together with (tr) or (Gal) indicating a trivial or Galois $W_H C$-action on $\Phi^C A$, respectively.}
\label{table:1}
\end{table}

\begin{prop}\label{pr:KU_decomp_A4}
    For $A=\gKU$ or $\gKO$, geometric fixed points induce an equivalence of 2-rings
    \[\Perf_{A_4}(A/e) \xrightarrow{\simeq} {\Perf(A[\tfrac{1}{2}])^{hC_2} \times \Perf(\Phi^{C_3} A)},\]
    where the $C_2$-actions are trivial.
\end{prop}

\iffalse
\begin{figure}
    \centering
    \[\begin{tikzcd}
	{C_2} & {C_3} \\
	& e
	\arrow["{C_2-\mathrm{trivial}}", from=1-1, to=1-1, loop, in=55, out=125, distance=10mm]
	\arrow[no head, from=1-1, to=2-2]
	\arrow["e", from=1-2, to=1-2, loop, in=55, out=125, distance=10mm]
	\arrow[no head, from=2-2, to=1-2]
\end{tikzcd}
\qquad
\begin{tikzcd}
	{C_4} \\
	{C_2} & {C_3} & {C_2}\\
	& e
	\arrow["{C_2-\mathrm{Galois}}", from=1-1, to=1-1, loop, in=145, out=215, distance=10mm]
	\arrow[no head, from=1-1, to=2-1]
	\arrow["{V-\mathrm{trivial}}", from=2-1, to=2-1, loop, in=145, out=215, distance=10mm]
	\arrow[no head, from=2-1, to=3-2]
	\arrow["{C_2-\mathrm{Galois}}", from=2-2, to=2-2, loop, in=55, out=125, distance=10mm]
	\arrow[no head, from=2-2, to=3-2]
	\arrow["{C_2-\mathrm{trivial}}", from=2-3, to=2-3, loop, in=55, out=125, distance=10mm]
	\arrow[no head, from=3-2, to=2-3]
\end{tikzcd}\]
    \caption{Cyclic subgroup lattices with Weyl group actions for $A_4$ (L) and $S_4$ (R).}
    \label{fig:a4s4}
\end{figure}
\fi

\begin{prop}\label{pr:KU_decomp_S4}
    For $A=\gKU$ or $\gKO$, geometric fixed points induce an equivalence of 2-rings
    \[\Perf_{S_4}(A/e) \xrightarrow{\simeq} {\Perf(A[\tfrac{1}{2}])^{hV} \times \Perf(A[\tfrac{1}{2}])^{hC_2} \times \Perf(A[\tfrac{1}{3}])  \times \Perf(A[\tfrac{1}{2}])},\]
    where the $C_2$- and $V$-actions are all trivial.
\end{prop}

\begin{prop}\label{pr:KU_decomp_A5}
    For $A=\gKU$ or $\gKO$, geometric fixed points induce an equivalence of 2-rings
    \[\Perf_{A_5}(A/e) \xrightarrow{\simeq} {\Perf(A[\tfrac{1}{2}])^{hC_2} \times \Perf(A[\tfrac{1}{3}]) \times \Perf(\Phi^{C_5} A^{hC_2})},\]
    where the $C_2$-action on module category is trivial and the $C_2$-action on $\Phi^{C_5} A$ is Galois.
\end{prop}

\iffalse
\begin{figure}
    \centering
    \begin{tikzcd}
	{C_2} & {C_3} & {C_5} \\
	& e
	\arrow["{C_2-\mathrm{trivial}}", from=1-1, to=1-1, loop, in=55, out=125, distance=10mm]
	\arrow[no head, from=1-1, to=2-2]
	\arrow["{C_2-\mathrm{Galois}}", from=1-2, to=1-2, loop, in=55, out=125, distance=10mm]
	\arrow[no head, from=1-2, to=2-2]
	\arrow["{C_2-\mathrm{Galois}}", from=1-3, to=1-3, loop, in=55, out=125, distance=10mm]
	\arrow[no head, from=2-2, to=1-3]
\end{tikzcd}
\quad
\begin{tikzcd}
	{C_4} \\
	{C_2} & {C_3} & {C_3} & {C_5} \\
	&& e
	\arrow["{C_2-\mathrm{Galois}}", from=1-1, to=1-1, loop, in=55, out=125, distance=10mm]
	\arrow[no head, from=1-1, to=2-1]
	\arrow["{V-\mathrm{trivial}}", from=2-1, to=2-1, loop, in=235, out=305, distance=10mm]
	\arrow[no head, from=2-1, to=3-3]
	\arrow["{S_3}", from=2-2, to=2-2, loop, in=55, out=125, distance=10mm]
	\arrow[no head, from=2-2, to=3-3]
	\arrow["{S_3}", from=2-3, to=2-3, loop, in=55, out=125, distance=10mm]
	\arrow[no head, from=2-3, to=3-3]
	\arrow["{C_2-\mathrm{Galois}}", from=2-4, to=2-4, loop, in=55, out=125, distance=10mm]
	\arrow[no head, from=3-3, to=2-4]
\end{tikzcd}
    \caption{Cyclic subgroup lattices with Weyl group actions of $A_5$ (L) and $A_6$ (R).}
    \label{fig:a5a6}
\end{figure}
\fi

For larger symmetric groups, such as $S_5$, the above arguments fall apart. Indeed, there is a subgroup of order $3$ in $S_5$ with Weyl group $V$, but the $V$-action is not Galois, so we have no obvious Tate vanishing. For alternating groups, we can continue.

\begin{prop}\label{pr:KU_decomp_A6}
    For $A=\gKU$ or $\gKO$, geometric fixed points induce an equivalence of 2-rings
    \[\Perf_{A_6}(A/e) \xrightarrow{\simeq} {\Perf(A[\tfrac{1}{2}])^{hV} \times \Perf(A[\tfrac{1}{2}]) \times \left(\Perf(A[\tfrac{1}{3}])^{hC_3}\right)^2 \times \Perf(\Phi^{C_5} A^{hC_2})},\]
    where the $V$- and $C_3$-actions above are trivial and the $C_2$-action is Galois.
\end{prop}

Thank you to Christian Carrick for helping to identify one of the upper-right factors.

\begin{proof}
    Analysing the cyclic subgroup lattice for $A_6$ above, there is two things to check between all of the arguments made above and the desired conclusion. These are the two equivalences of 2-rings
    \begin{equation}\label{eq:twoequivtworings}\Perf(\Phi^{C_3} A)^{hS_3} \simeq \Perf(A[\tfrac{1}{3}])^{C_3}, \qquad \Perf(\Phi^{C_3} A)^{tS_3} = 0.\end{equation}
    First, we decompose a homotopy limit into two separate components using the normal subgroup $C_3\leq S_3$ and its quotient
    \[\Perf(\Phi^{C_3} A)^{hS_3} \simeq (\Perf(\Phi^{C_3} A)^{hC_3})^{hC_2} \simeq (\Perf(\Phi^{C_3} A)^{hC_2})^{hC_3} \simeq \Perf(A[\tfrac{1}{3}])^{hC_3},\]
    the second comes from the fact that limits commute with functor categories, as the $C_3$-action is trivial, and the third by Galois descent, \Cref{main_algebraic}, and the computation $\Phi^{C_3} A^{hC_2} \simeq A[\tfrac{1}{3}]$. The second equivalence of (\ref{eq:twoequivtworings}) follows from the computation
    \[\map_{\Perf(\Phi^{C_3})^{tS_3}}(\1,\1) \overset{\text{(\ref{lm:42_krause_generalised})}}{\simeq} \Phi^{C_3} A^{tS_3} \xrightarrow{\simeq} \Phi^{C_3} A^{tC_2} = 0;\]
    the second equivalence can be checked on homotopy groups via the two Tate spectral sequences, and the third is the Tate vanishing of a Galois extension of \Cref{main_algebraic}.
\end{proof}

Finally, the group $\GL_3(\F_2)$, as the last nonabelian simple group of order less than $500$.

\begin{prop}\label{pr:KU_decomp_GL3F2}
    For $A=\gKU$ or $\gKO$, geometric fixed points induce an equivalence of 2-rings
    \[\Perf_{\GL_3(\F_2)}(A/e) \xrightarrow{\simeq} {\Perf(A[\tfrac{1}{2}])^{hV} \times \Perf(A[\tfrac{1}{2}]) \times \Perf(A[\tfrac{1}{3}]) \times \Perf(\Phi^{C_7} A^{hC_3})},\]
    where the $V$-action is trivial and the $C_3$-action is Galois.
\end{prop}

Using \Cref{sssec:KU}, we can make the upper-right corner even more explicit; see (\ref{eq:c7geo}).

\iffalse
\begin{figure}
    \centering
    \begin{tikzcd}
	{C_4} \\
	{C_2} & {C_3} & {C_7} \\
	& e
	\arrow["{C_2-\mathrm{Galois}}", from=1-1, to=1-1, loop, in=145, out=215, distance=10mm]
	\arrow[no head, from=1-1, to=2-1]
	\arrow["{V-\mathrm{trivial}}", from=2-1, to=2-1, loop, in=145, out=215, distance=10mm]
	\arrow[no head, from=2-1, to=3-2]
	\arrow["{C_2-\mathrm{Galois}}", from=2-2, to=2-2, loop, in=55, out=125, distance=10mm]
	\arrow[no head, from=2-2, to=3-2]
	\arrow["{C_3-\mathrm{Galois}}", from=2-3, to=2-3, loop, in=55, out=125, distance=10mm]
	\arrow[no head, from=3-2, to=2-3]
\end{tikzcd}
    \caption{Cyclic subgroup lattice with Weyl group actions for $\GL_3(\F_2)$.}
    \label{fig:gl3f2}
\end{figure}
\fi

%%%%%%%%%%%%%%%%%%%%%%%%%%%%%%%%%%%%%%%%%%%%%%%%%%%%%%%%%%%%%%%
%%%%%%%%%%%%%%%%%%%%%%%%%%%%%%%%%%%%%%%%%%%%%%%%%%%%%%%%%%%%%%%
%%%%%%%%%%%%%%%%%%%%%%%%%%%%%%%%%%%%%%%%%%%%%%%%%%%%%%%%%%%%%%%
\appendix
\section{Closure properties for morphisms of spectral stacks}\label{appendix}
Here we catalogue some generic facts about morphisms in $\Stk$ used in this article, akin to the appendices of \cite{hortzwedhorn}; we will try to keep a more complete up-to-date list \href{https://www.jmdavies.org/properties-of-morphisms-in-sag}{here}.

\begin{mydef}\label{df:propertiy_of_morphisms}
    Let $\calP$ be a property of morphisms in $\Stk$. We say the $\calP$ satisfies:
    \begin{enumerate}
        \item \emph{base change} if given a morphism $f\colon \Y \to \X$ satisfying $\calP$ and an arbitrary morphism $g\colon \X' \to \X$, then the base change $\Y \times_\X \X' \to \X'$ satisfies $\calP$.
        \item \emph{cancellation} if given two morphisms $f \colon \Y \to \X$ and $g\colon \X \to \sfW$ such that their composite satisfies $\calP$, then $f$ satisfies $\calP$.
        \item \emph{composition} if given two morphisms $g\colon \sfZ \to \Y$ and $f\colon \Y \to \X$ both satisfying $\calP$, then the composite $f\circ g$ satisfies $\calP$.
        \item \emph{descent} if given a morphism $f \colon \Y \to \X$ together with an effective epimorphism $\X' \to \X$ such that base change $\Y \times_\X \X' \to \X'$ satisfies $\calP$, then $f$ satisfies $\calP$.
        %\item \emph{LOCS (local on the source)} if given a morphism $f\colon \Y \to \X$ together with an effective epimorphism $g \colon \Y' \to \Y$ such that the composition $\Y' \to \Y \to \X$ satisfies $\calP$, then $f$ satisfies $\calP$.
    \end{enumerate}
\end{mydef}

Recall the following notions of morphisms, mostly cribbing from \cite{reconstruction}.

\begin{mydef}\label{df:definition_of_morphisms_of_stacks}
    A morphism of stacks $\Y \to \X$ is:
    \begin{enumerate}
        \item \emph{affine} if for all maps $\Spec A \to \X$ with source an affine stack, the fibre product $\Y \times_\X \Spec A = \Spec B$ is affine.
        \item \emph{affine flat} if it is affine and for all maps $\Spec A \to \X$ with source an affine stack, the induced map $A \to B$ is a flat map of $\E_\infty$-rings.
        \item \emph{affine étale} if it is affine and for all maps $\Spec A \to \X$ with source an affine stack, the induced map $A \to B$ is an étale map of $\E_\infty$-rings.
        \item \emph{finite flat} if it is affine flat and for all maps $\Spec A \to \X$ with source an affine stack, the induced map $\pi_0 A \to \pi_0 B$ is a finite map of commutative rings.
        \item \emph{finite étale} if it is both affine étale and finite flat.
    \end{enumerate}
\end{mydef}

\begin{prop}\label{pr:permanenceproperties_for_morphisms}
    The following properties of morphisms in $\Stk$ satisfy base change, composition, and descent: affine, affine flat, affine étale, finite flat, and finite étale. Moreover, affine morphisms satisfy cancellation if $g$ has affine diagonal (in particular, if $g$ is affine).
\end{prop}

We expect the usual cancellation properties for affine étale and finite étale morphisms, but we will not need them here.
%Affine morphisms clearly do not generally satisfy LOCS: the map $\Spec \Sph/C_2 \to \Spec \Sph$ is clearly not affine, but precomposition with the effective epimorphism $\Spec \Sph \to \Spec \Sph/C_2$ is the identity.

\begin{proof}
    Concerning affine morphisms: base change holds by definition and composition is easily seen to hold. For descent, recall the equivalence of categories
    \[\Stk_{/\X}^\aff \xrightarrow{\simeq} \CAlg(\QCoh(\X))^\op, \qquad (h\colon \sfZ \to \X) \mapsto h_\ast \O_\sfZ\]
    of \cite[Thm.2.2.3.1]{reconstruction}, where $\Stk^\aff_{/\X}$ is the category of stacks which are relatively affine over $\X$, and whose inverse is given by the relative spectrum construction $\relSpec_\X(-)$. To show that $\Y \to \X$ is affine, it suffices to show that the canonical map $\Y \to \relSpec_\X f_\ast \O_\Y$ is an equivalence. This map of stacks is an equivalence if and only if its base change to $\X'$ is an equivalence, which is true by assumption. For cancellation, the map $f$ can be written as the upper composite
    \[\begin{tikzcd}
        {\Y}\ar[r, "{\Ga_f}"]\ar[d, "f"] &   {\Y \times_\sfW \X}\ar[d]\ar[r, "{\pi_\X}"]   &   {\X}  \\
        {\X}\ar[r, "{\Delta_g}"]    &   {\X \times_\sfW \X} &
    \end{tikzcd}\]
    of the graph of $f$ with the projection to $\sfY$, where the square is Cartesian. As $g$ has affine diagonal, then $\Delta_g$ is affine, and $\pi_\X$ is also affine by base change, as it is the pullback of the composition $\Y \to \W$. As affine morphisms satisfy composition, we see that $f$ itself is affine. The parenthetical statement is justified by \Cref{lm:affine_has_affine_diagonal} below.

    Concerning affine flat morphisms: base change is again clear by definition and composition holds from the affine situation plus the fact that flat morphisms of $\E_\infty$-rings compose. Indeed, given flat morphisms of $\E_\infty$-rings $A \to B$ and $B \to C$, then the composite $\pi_0 A \to \pi_0 B \to \pi_0 C$ is flat and for each integer $n$ we have equivalences of $\pi_0 C$-modules
    \[\pi_0 C \otimes_{\pi_0 A} \pi_n A \simeq \pi_0 C \otimes_{\pi_0 B} \pi_0 B \otimes_{\pi_0 A} \pi_n A \simeq \pi_0 C \otimes_{\pi_0 B} \pi_n B \simeq \pi_n C.\]
    For descent, we already know that $f$ is affine, so we are left to show flatness. Then consider the following commutative diagram of stacks
    \begin{equation}\label{eq:reusable_cube}\begin{tikzcd}
        {\Spec B''}\ar[r]\ar[dd] &   {\Spec B'}\ar[rd]\ar[rr]\ar[dd]  &&  {\Y'}\ar[dd]\ar[rd]    &     \\
            &&   {\Spec B}\ar[rr]\ar[dd] &&   {\Y}\ar[dd, "f"]   \\
        {\Spec A''}\ar[r]       &   {\Spec A'}\ar[rd]\ar[rr]  &&  {\X'}\ar[rd]    &     \\
            &&   {\Spec A}\ar[rr] &&   {\X,}
    \end{tikzcd}\end{equation}
    where the map $\Spec A \to \X$ is arbitrary, $\X' \to \X$ is an effective epimorphism such that $\Y' \to \X'$ the base change of $f$ along $\X' \to \X$ is affine flat, the rest of the cube is defined by pullback, and the left square is defined by the observation \cite[Lm.A.13]{tokic_family_completion} that for an effective epimorphism $\Spec A' \to \Spec A$ there is an $\E_\infty$-$A'$-algebra $A''$ such that $A \to A''$ is faithfully flat. To see that $f$ is flat, we want to show that $A \to B$ is flat. The map $A \to A''$ is faithfully flat, so by descent for flat morphisms \cite[Pr.2.8.4.2(6)]{sag}, it suffices to show that $A'' \to A'' \otimes_A B \simeq B''$ is flat, but this map is the left-most vertical map in the diagram above, and is flat as $A' \to B'$ is flat by our assumption on $\X' \to \X$ and flat maps satisfy base change.

    Concerning finite flat morphisms: base change follows by definition and composition from the affine flat case and the classical fact that finite morphisms of rings compose. For descent, we already know that $f$ is affine flat by the previous points, so it suffices to see that the base change of $f$ along any $\Spec A \to \X$, written as $\Spec B \to \Spec A$, induces a finite map of rings on $\pi_0$. Consider the cube of Cartesian squares (\ref{eq:reusable_cube}), where $\X' \to \X$ is an effective epimorphism such that the base change of $f$ along $\X' \to \X$ is finite flat, and $A \to A''$ is faithfully flat. To show that $\pi_0 A \to \pi_0 B$ is finite, it suffices to check after faithfully flat base change \cite[\textsection C]{hortzwedhorn} such as to $\pi_0 A''$. However, $\pi_0 A'' \to \pi_0 B''$ is finite flat as the base change of the finite flat map $\pi_0 A' \to \pi_0 B'$, using that $\pi_0$ sends coCartesian diagrams of $\E_\infty$-rings where the morphisms are flat to coCartesian diagram of commutative rings by a degenerating Tor spectral sequence; see \cite[Lm.3.4.3.1]{reconstruction}, for example.

    Concerning affine étale morphisms: base change is clear by definition and composition holds as the composition of étale morphisms of $\E_\infty$-rings compose. Indeed, flat morphisms of $\E_\infty$-rings compose, so it suffices to show that étale morphisms between classical rings compose, which is stated in \cite[\textsection C]{hortzwedhorn}, for example. For descent, we argue as in the finite flat case above, which boils down to the classical fact that étale morphisms between commutative rings satisfy descent.

    Concerning finite étale morphisms: combine the results for finite flat and affine étale morphisms.
\end{proof}

\begin{lemma}\label{lm:affine_has_affine_diagonal}
    An affine morphism of stacks $f\colon \Y \to \X$ has an affine diagonal.
\end{lemma}

The point of the proof is that the diagonal morphism between affine stacks is affine.

\begin{proof}
    To show that $\Delta_f\colon \Y \to \Y\times_\X \Y$ is affine, it suffices to write $\Y\times_\X \Y$ as a colimit of affines, and check that the pullback of $\Delta$ to these affines is an affine morphism by descent for affine morphisms. To this end, write $\X \simeq \colim \Spec A$ as a colimit of affines. As $f$ is affine, the pullback $\Spec A\times_\X \Y \simeq \Spec A'$ is affine, and so $\Y \simeq \colim \Spec A\times_\X \Y\simeq \colim \Spec A'$ can also be written as a colimit of affines. Moreover, both of the projections $\Y\times_\X \Y \to \Y$ are affine by base change, so we can also write
    \[\Y\times_\X \Y \simeq (\colim \Spec A') \times_{\X} \Y \simeq \colim \Spec A''\]
    as a colimit of affines. We can now place $\Delta_f$ and its pullback along one of the maps $\Spec A'' \to \Y\times_\X \Y$ in the commutative diagram
    \[\begin{tikzcd}
        {\Spec A'}\ar[r]\ar[d]  &   {\Spec A''}\ar[d]\ar[r] &   {\Spec A'}\ar[d]    \\
        {\sfY}\ar[r, "{\Delta_f}"]  &   {\Y\times_\X \Y}\ar[r]  &   {\Y,}
    \end{tikzcd}\]
    where all squares are Cartesian. As $\Delta_f$ base changes along all such maps to an affine morphism, as a morphism between affine stacks, we see that $\Delta_f$ itself is affine.
\end{proof}

\begin{remark}
    In fact, this proof shows that affine morphisms are \emph{separated}, meaning that their diagonal is a \emph{closed immersion} using the latter notion from \cite[Df.5.5.3]{temperedglobal}. Indeed, this follows by identifying $A''$ with $A'\otimes_A A'$, which is a consequence of the above proof.
\end{remark}

\begin{remark}
    We note that all of \Cref{df:propertiy_of_morphisms} make sense in $\Stk_\tau(\calC)$, as does the definition of affine morphisms and the proof of part 1 of \Cref{pr:permanenceproperties_for_morphisms} as well as \Cref{lm:affine_has_affine_diagonal}.
\end{remark}

%%%%%%%%%%%%%%%%%%%%%%%%%%%%%%%%%%%%%%%%%%%%%%%%%%%%%%%%%%%%%%%
%%%%%%%%%%%%%%%%%%%%%%%%%%%%%%%%%%%%%%%%%%%%%%%%%%%%%%%%%%%%%%%
%%%%%%%%%%%%%%%%%%%%%%%%%%%%%%%%%%%%%%%%%%%%%%%%%%%%%%%%%%%%%%%

\addcontentsline{toc}{section}{References}
\scriptsize
\bibliography{references} 
\bibliographystyle{alpha}

\end{document}